\documentclass[smallextended,referee,envcountsect,]{svjour3}
\smartqed
\usepackage{graphicx}

\usepackage{pgfplots}
\usepackage{epstopdf}%
\usepackage{multirow}%
\usepackage{amssymb,amsfonts}%
\usepackage{mathrsfs}%
\usepackage[title]{appendix}%
\usepackage{xcolor}%
\usepackage{textcomp}%
\usepackage{manyfoot}%
\usepackage{booktabs}%
\usepackage{algorithm}%
\usepackage{algorithmicx}%
\usepackage{algpseudocode}%
\usepackage{listings}%
\usepackage{ragged2e}
\usepackage{fullpage}
\usepackage{changepage}
\usepackage{stfloats}
\usepackage[caption=false]{subfig}
\usepackage{extarrows}
\usepackage{hyperref}
\usepackage{cleveref}
\usepackage{tikz}
\usepackage{pifont}
\usepackage{enumitem}
\usepackage{bm}

\journalname{JOTA}

\renewcommand{\thesubsection}{\textbf{\arabic{section}.\arabic{subsection}}} 

\begin{document}

\title{Convergence Analysis of Hessian-Damped Tikhonov Regularized Dynamics with Oscillation Control for Convex-Concave Bilinear Saddle Point Problems}

\author{Bohan Zhang \textsuperscript{1} \and Xiaojun Zhang \textsuperscript{1*}}

\institute{Bohan Zhang \at
             School of Mathematical Sciences, University of Electronic Science and Technology of China \\
              Chengdu, 611731, Sichuan, P.R. China\\
              bohanzhang@std.uestc.edu.cn
           \and
              Xiaojun Zhang,  Corresponding author  \at
              School of Mathematical Sciences, University of Electronic Science and Technology of China \\
              Chengdu, 611731, Sichuan, P.R. China\\
              sczhxj@uestc.edu.cn
}

\date{Received: date / Accepted: date}

\maketitle

\begin{abstract}
In this paper, we propose a class of general second-order primal-dual dynamical systems with Tikhonov regularization and Hessian-driven damping for solving convex-concave bilinear saddle point problems. The proposed dynamical system incorporates five general time-varying terms: viscous damping, time scaling, extrapolation, Tikhonov regularization, and Hessian-driven damping parameters. Under suitable parametric conditions, we analyze the asymptotic convergence properties of the dynamical system by constructing appropriate Lyapunov functions. Specifically, we obtain the convergence rate of the primal-dual gap and the boundedness of trajectories in the proposed dynamical system, and provide some integral estimates. Furthermore, we theoretically prove that the trajectories generated by the dynamical system converge strongly to the minimum-norm solution of the saddle point problem, and fully demonstrate that Hessian-driven damping can effectively alleviate oscillations. Finally, numerical experiments are conducted to verify the validity of the above theoretical results.
\end{abstract}


\section{Introduction}\label{sec1}
\subsection{Problem Statement}
Oscillatory phenomena are inherent characteristics of nonlinear dynamical systems, which can give rise to issues such as trajectory divergence and system instability, and have become a critical bottleneck restricting technological advancement. Such phenomena are ubiquitous across a wide range of disciplinary fields, including civil engineering \cite{intro15,intro14}, aerospace engineering \cite{intro3,intro4}, energy and power systems \cite{intro13,intro12}, and fundamental physics and plasma dynamics \cite{intro9,intro8}, among others.

In optimization theory, the oscillation phenomenon is ubiquitous and of significant importance \cite{popular3,intro17,suanfa4,polyak1}. It refers to the non-monotonic fluctuations exhibited by iterative sequences, objective function values, or gradient norms during the convergence process. Rather than converging monotonically to the limit, such sequences fluctuate repeatedly around the optimal point or optimal value, and their iterative trajectories typically appear as zigzag paths or back-and-forth oscillatory jumps about the optimal solution \cite{intro16}. This oscillatory behavior constitutes a critical bottleneck restricting the performance of optimization algorithms and dynamical systems. Iterative oscillations decelerate the convergence rate of optimization algorithms, increase computational costs, and reduce solution efficiency \cite{intro18,intro17}; they also elevate the risk of entrapment in local optima, undermine iterative stability, and degrade the quality and reliability of the computed solutions \cite{intro19}.

In recent years, continuous-time dynamical methods, as one of the effective approaches for solving optimization problems, have attracted considerable attention from numerous researchers. A common strategy in this field is to introduce and tune various damping terms to suppress oscillatory behavior. Specifically, consider the unconstrained optimization problem:
\begin{equation}\label{eq1}
	\min_{x \in \mathbb{R}^n} \; f\left(x\right),
\end{equation}
where $\mathbb{R}^n$ denotes the $n$-dimensional Euclidean space. The dynamical system framework for first-order gradient flows, which serves as a continuous-time reformulation of the steepest descent method, can be traced back to the 1960s. Smale \cite{smale1961gradient} systematically analyzed the gradient flow structure of such problems from the dynamical systems perspective for the first time, and modeled them as the following first-order ordinary differential equation: 
\begin{equation*}
	\dot{x}(t) + \nabla f\left(x(t)\right) = 0. \quad (SD)
\end{equation*}
This method inherits the property of slow convergence from the continuous flow and manifests oscillatory behavior when applied to ill-conditioned problems.

Building upon the first-order gradient flow ODE $(SD)$, Polyak \cite{polyak1,polyak2} proposed solving this type of unconstrained optimization problem using a heavy ball system with friction, as follows:
\begin{equation*}
	\ddot{x}(t) + \alpha \dot{x}(t) + \nabla f\left(x(t)\right) = 0, \quad \left(HBF\right)
\end{equation*}
where $\alpha$ is a positive constant denoting the constant viscous damping. Polyak proved that under the condition that $f$ is a general convex and Lipschitz smooth function, $f(x(t))-\min f$ converges at a rate of $\mathcal{O}\left(\frac{1}{t}\right)$ and the trajectory
	weakly converges to a minimizer of $f$. He also demonstrated that the inertial mechanism can significantly accelerate the convergence speed of optimization algorithms and mitigate, to a certain extent, the oscillatory phenomena in the iterative process, although it cannot fundamentally eliminate them. This conclusion has laid a core theoretical foundation for the design of inertial optimization algorithms based on dynamical systems.

Alvarez et al. \cite{newton} transcended the limitation of the steepest descent method $(SD)$, which relies solely on first-order gradient information, by incorporating second-order Hessian information of the objective function. They proposed the following dynamical system:
\begin{equation*}
	\nabla^2f\left(x(t)\right)\dot{x}(t)+\nabla f\left(x(t)\right)=0. \quad (CN)
\end{equation*}
The Hessian term $\nabla^2f\left(x(t)\right)\dot{x}(t)$ in this dynamical system dynamically adjusts the damping strength according to the local curvature of the objective function to suppress zigzag oscillations. However, the dynamical system becomes ill-posed when the Hessian degenerates.

To address the inherent limitations of conventional optimization methods, including slow convergence in first-order gradient methods, system ill-posedness caused by Hessian degeneracy in continuous Newton methods, and oscillation-prone behavior of the inertial Heavy Ball with Friction method $(HBF)$, Alvarez et al. \cite{alvarez2002second} first proposed the pioneering Dynamic Inertial Newton system with Hessian-driven damping as follows:
\begin{equation*}
	\ddot{x}(t) + \alpha\dot{x}(t) + \gamma \nabla^2f\left(x(t)\right)\dot{x}(t) + \nabla f\left(x(t)\right) = 0, \quad \left(DIN\right)
\end{equation*}
where $\alpha > 0$ and $\gamma > 0$. $(DIN)$ is a well-posed dynamical system that incorporates both viscous damping and Hessian-driven damping. Under the assumption of general convex functions, it achieves a convergence rate of $\mathcal{O}\left(\frac{1}{t}\right)$, while for strongly convex functions, the dynamical system exhibits exponential convergence. This result is consistent with the behavior of $(HBF)$. Furthermore, similar to $(CN)$, $(DIN)$ can effectively suppress oscillations along the trajectory during the convergence process.

Building upon the dynamical system $(DIN)$, Bo\c{t} et al. \cite{ref22} developed a Tikhonov regularized dynamical system in which the constant viscous damping is replaced by an asymptotically vanishing damping term:
\begin{equation*}
	\ddot{x}(t)+\frac{\alpha}{t}\dot{x}(t)+\gamma\nabla^2f(x(t))\dot{x}(t)+\nabla{f}(x(t))+\epsilon(t)x(t)=0,
\end{equation*}
where $\alpha \geq 3$, $\gamma \geq 0$ and $\epsilon:[t_0,+\infty) \to [0,+\infty)$ is a nonincreasing function of class $C^1$ fulfilling $\lim\limits_{t \to +\infty}\,\epsilon(t) = 0$. They proved that for $\alpha \geq 3$, the error between the objective value $f\left(x(t)\right)$ along the dynamical system trajectory $x(t)$ and the global minimum $\min f$ of the convex function $f$ decays at an asymptotic rate of $\mathcal{O}\left(\frac{1}{t^2}\right)$. This dynamical system inherits the oscillation-suppression capabilities and rapid convergence properties of $(DIN)$, thereby addressing the issue of zigzag oscillations commonly observed in the trajectories of traditional inertial dynamical systems. Furthermore, the Tikhonov regularization term $\epsilon(t)x(t)$ allows us to overcome a fundamental limitation of $(DIN)$ by guaranteeing the strong convergence of trajectories to the minimum-norm solution.

For further relevant literature on dynamical systems with Hessian-driven damping in the context of unconstrained optimization problems, we refer the reader to references \cite{yueshu1,ref23,ref24,ref20,ref26,ref22,ref29,ref28}, wherein the effectiveness of Hessian-driven damping in suppressing oscillations has been rigorously established.

For convex optimization problems with linear equality constraints formulated as follows:
\begin{equation}\label{eq7}
	\begin{cases}
		\min\limits_{x \in \mathcal{X}} \; f(x) \\
		\text{s.t.} \; Ax = b
	\end{cases}
\end{equation}
where $\mathcal{X}$ and $\mathcal{Y}$ are two real Hilbert spaces, $f:\mathcal{X} \to \mathbb{R}$ is a continuously differentiable convex function, and $A: \mathcal{X} \to \mathcal{Y}$ is a continuous linear operator with $b \in \mathcal{Y}$. Several studies have validated that Hessian-driven methods can also significantly suppress oscillations. 

He et al. \cite{ref37} originally proposed a mixed-order primal-dual dynamical system with explicit Hessian-driven damping as follows:
\begin{equation*}
	\begin{cases}
		\ddot{x}(t) + \frac{\alpha}{t} \dot{x}(t) + \beta(t) \nabla_x \widehat{\mathcal{L}}(x(t), \lambda(t)) + \gamma(t) \frac{\mathrm{d}}{\mathrm{d}t} \left( \nabla_x \widehat{\mathcal{L}}(x(t), \lambda(t)) \right) = 0, \\
		\dot{\lambda}(t) - \eta(t) \nabla_\lambda \widehat{\mathcal{L}}\left(x(t) + \frac{t}{\alpha-1} \dot{x}(t), \lambda(t)\right) = 0,
	\end{cases}
\end{equation*}
where $\alpha>1$, and $\beta, \gamma, \eta: [t_0, +\infty) \to (0, +\infty)$ are three continuously differentiable functions. Here, $\widehat{\mathcal{L}}$ is the Lagrangian function of problem \eqref{eq7}, defined as $\widehat{\mathcal{L}}\left(x,\lambda\right) = f(x) + \langle \lambda, Ax-b \rangle$. They verified that, under the conditions of general Hessian-driven damping and scaling coefficients, the proposed dynamical system ensures the rapid convergence of the primal-dual gap, the objective residual, and the feasibility violation. 

For further studies on dynamical systems with Hessian-driven damping, readers may refer to \cite{ref39,ref37}. These works fully demonstrate that Hessian-driven damping is effective in suppressing oscillatory phenomena in the trajectory convergence process of dynamical systems for optimization problems with linear equality constraints.

Saddle point problems are closely related to the unconstrained optimization problem \eqref{eq1} and linearly constrained optimization problem \eqref{eq7} introduced above. As a core branch of optimization theory, they have found wide applications in many fields, including machine learning \cite{bg3,bg4}, image processing \cite{ref3,ref6}, and game theory \cite{bg2,bg1}. The convex-concave bilinear saddle point problem considered in this paper can be formulated as follows after endowing the space with the inner product $\langle \cdot,\cdot \rangle$ and defining the induced norm as $\| \cdot \|^2 = \langle \cdot,\cdot \rangle$:
\begin{equation}\label{eq13}
	\min_{x \in \mathbb{R}^n} \max_{y \in \mathbb{R}^m} \; \mathcal{L}(x, y) := f(x) + \langle Kx, y \rangle - g(y),
\end{equation}
where
\begin{equation}\label{eqH0}
	\begin{cases}
		f: \mathbb{R}^n \to \mathbb{R} \text{ and } g: \mathbb{R}^m \to \mathbb{R} \text{ are convex twice continuously differentiable functions,} \\
		K\in \mathbb{R}^{m\times n} \text{ is a continuous linear operator,} \\
		\text{the set } \Omega \text{ of primal-dual optimal solutions of } \eqref{eq13} \text{ is assumed to be nonempty.}
	\end{cases}
	\tag{$\mathbf{H_0}$}
\end{equation}
In view of the connection between problem \eqref{eq13} and \eqref{eq1}, \eqref{eq7}, it is natural to consider solving the convex-concave bilinear saddle point problem \eqref{eq13} via a second-order primal-dual dynamical system.

He et al. \cite{ref42} first developed a general second-order dynamical system as:
\begin{equation*}\label{eq14}
	\begin{cases}
		\ddot{x}(t) + \alpha(t) \dot{x}(t) +\beta(t) \nabla_x \mathcal{L}(x(t), y(t) + \theta(t) \dot{y}(t))=0, \\
		\ddot{y}(t) + \alpha(t) \dot{y}(t) -\beta(t) \nabla_y \mathcal{L}(x(t) + \theta(t) \dot{x}(t), y(t))=0,
	\end{cases}
	(MPDD)
\end{equation*}
where $\alpha,\beta,\theta:[t_0,+\infty) \to (0,+\infty)$ denote the viscous damping coefficient, time scaling coefficient, and extrapolation coefficient, respectively. They established that the primal-dual gap along the resulting trajectories achieves a convergence rate of $\mathcal{O}\left(\frac{1}{t^{2a}\delta(t)\beta(t)}\right)$.

Nevertheless, although several existing works \cite{ref41,ref42,ref43,ref44,ref45,ref40} have preliminarily investigated the dynamic behaviors of second-order primal–dual dynamical systems equipped with elaborately designed viscous damping parameters, time scaling coefficients and extrapolation coefficients, such dynamical systems still suffer from severe numerical oscillations in their trajectory convergence processes. These oscillations not only deteriorate the convergence rate, but also hinder the stable achievement of high-precision convergence. Accordingly, the design of dynamical systems that integrate fast convergence rate and oscillation suppression capability to satisfy the demand for high-precision numerical solution remains a significant open issue in the current research.

In this context, inspired by \cite{ref22,hessian2}, this paper constructs a class of general second-order primal-dual dynamical systems for solving the convex-concave bilinear saddle point problem \eqref{eq13}, which simultaneously contain Hessian-driven damping terms and Tikhonov regularization terms. Combining the Tikhonov regularization technique with the Hessian-matrix-driven damping technique can effectively suppress the transverse oscillations of the dynamical system trajectories and ensure that the trajectories generated by the dynamical system \eqref{eq15} converge strongly to the minimum-norm solution of the convex-concave bilinear saddle point problem \eqref{eq13}. The dynamical system is given as follows:
\begin{equation}\label{eq15}
	\begin{cases}
		\ddot{x}(t) + \alpha(t)\dot{x}(t) + \beta(t)\nabla_x\mathcal{L}_t(x(t),y(t)+\theta t\dot{y}(t)) + \gamma(t)\frac{\mathrm{d}}{\mathrm{d}t}\nabla_x\mathcal{L}_t(x(t),y(t)+\theta t\dot{y}(t))=0, \\
		\ddot{y}(t)+\alpha(t)\dot{y}(t) - \beta(t)\nabla_x\mathcal{L}_t(x(t)+\theta t\dot{x}(t),y(t))-\gamma(t)\frac{\mathrm{d}}{\mathrm{d}t}\nabla_y\mathcal{L}_t(x(t)+\theta t\dot{x}(t),y(t))=0,
	\end{cases}
\end{equation}
where $t_0 \geq t>0$, $\theta>0$, $\alpha:[t_0,+\infty) \to (0,+\infty)$ is a viscous damping parameter, $\beta:[t_0,+\infty) \to (0,+\infty)$ is a time scaling parameter, $\gamma:[t_0,+\infty) \to (0,+\infty)$ is a Hessian-driven damping parameter, and $\theta t$ is an extrapolation parameter. Moreover, $\epsilon(t)$ is a Tikhonov regularization parameter, which is $\mathcal{C}^1$ nonincreasing function such that $\displaystyle \lim_{t \to +\infty} \, \epsilon(t) = 0$. In particular, we innovatively incorporate the augmented Lagrangian form of the objective function in \eqref{eq13} into the dynamical system \eqref{eq15} to reduce the tedious scaling operations required for subsequent asymptotic analysis, where $\mathcal{L}_t:\mathbb{R}^n \times \mathbb{R}^m \to \mathbb{R}$ denotes the augmented Lagrangian saddle function, whose detailed definition can be found in Definition \ref{Lt}.

\subsection{Main Contributions}
This paper proposes a general primal-dual second-order dynamical system for solving the convex-concave bilinear saddle point problem \eqref{eq13}, and effectively mitigates the oscillations of trajectories and gap functions during the convergence process. In comparison to the dynamical systems proposed in \cite{ref42,hessian2}, the dynamical system \eqref{eq15} simultaneously incorporates viscous damping, extrapolation, time scaling, Hessian-driven damping, and Tikhonov regularization terms. The main contributions are summarized as follows:
\begin{itemize}
	\item[(i)] In the case where $\beta(t)\epsilon(t)$ decreases rapidly to zero, i.e., $\int_{t_0}^{+\infty} \, t \beta(t) \epsilon(t) \, \mathrm{d}t < +\infty$, we show that the convergence rate of the primal-dual gap along the trajectory $\left(x(t),y(t)\right)$ generated by the dynamical system \eqref{eq13} is $\mathcal{O}\left(\frac{1}{t^2\beta(t)}\right)$.
	\item[(ii)] In the case where $\beta(t)\epsilon(t)$ tends slowly to zero, i.e., $\int_{t_0}^{+\infty} \, \frac{\beta(t)\epsilon(t)}{t} \, \mathrm{d}t < +\infty$, we show that the convergence rate of the primal-dual gap along the trajectory $\left(x(t),y(t)\right)$ generated by the dynamical system \eqref{eq13} achieves $o\left(\frac{1}{\beta(t)}\right)$.
	\item[(iii)] Based on the stability theory of time-varying parameter non-autonomous dynamical systems and by applying the Lyapunov method, this paper theoretically verifies that the Hessian-driven damping in the constructed dynamical system \eqref{eq15} can effectively suppress the oscillation phenomenon during the convergence process. Furthermore, it is also proven that the strong convergence property of the trajectories generated by the dynamical system \eqref{eq15} to the minimum norm element of the saddle point set can be guaranteed.
	\item[(iv)] Numerical experiments are conducted to intuitively demonstrate that the Hessian-driven damping can significantly suppress the oscillation phenomenon, and the Tikhonov regularization term can ensure the strong convergence of the trajectories.
\end{itemize}

\subsection{Paper Organization}

This paper is organized as follows. Section \ref{sec2} presents some preliminary results, and introduces the Lyapunov theory. Section \ref{sec3} analyzes the asymptotic convergence properties of the proposed dynamical system for convex-concave bilinear saddle point problems by constructing a suitable Lyapunov function and provides special cases for illustration. Section \ref{sec4} discusses the strong convergence of the trajectory generated by the dynamical system \eqref{eq15} when the Tikhonov regularization parameter $\epsilon(t)$ decays to zero at an appropriate rate. Section \ref{sec5} verifies the theoretical strong convergence properties of this work through numerical experiments and visually demonstrates the effectiveness of Hessian-driven damping in suppressing oscillations. Section \ref{sec13} summarizes the key findings and contributions of this study, thereby concluding the paper.

\section{Preliminaries}\label{sec2}

\subsection{The (Augmented) Lagrangian Function}
\begin{definition}\label{andian}
	For the convex-concave bilinear saddle point problem \eqref{eq13}, a pair $(x^*,y^*) \in \mathbb{R}^n \times \mathbb{R}^m$ is said to be a saddle point of the Lagrangian function $\mathcal{L}$ if and only if
	\begin{equation}\label{eq19}
		\mathcal{L}(x^*, y) \leq \mathcal{L}(x^*, y^*) \leq \mathcal{L}(x, y^*), \quad \forall (x, y) \in \mathbb{R}^n \times \mathbb{R}^m.
	\end{equation}
\end{definition}
It is straightforward to verify that the set of primal-dual optimal solutions of problem \eqref{eq13} is exactly the set of saddle points of the Lagrangian $\mathcal{L}$, which is denoted by the nonempty set $\Omega$.
\begin{proposition}
	Let $K^*: \mathbb{R}^n \to \mathbb{R}^m$ denote the adjoint operator of $K$, then
	\begin{equation}\label{eq20}
		(x^*,y^*)\in\Omega \Leftrightarrow \begin{cases}
			\nabla_x \mathcal{L}(x^*, y^*) = \nabla f(x^*) + K^* y^*=0, \\
			\nabla_y \mathcal{L}(x^*, y^*) = -\nabla g(y^*) + K x^*=0.
		\end{cases}
	\end{equation}
\end{proposition}

\begin{definition}\label{Lt}
	For any $t_0 > 0$ and a function $\epsilon: [t_0, +\infty) \to (0, +\infty)$ associated with the Lagrangian function $\mathcal{L}$, we define the \emph{augmented Lagrangian function} $\mathcal{L}_t: \mathbb{R}^n \times \mathbb{R}^m \to \mathbb{R}$ as
	\begin{equation}\label{eq21}
		\begin{aligned}
			\mathcal{L}_t(x, y) :&= \mathcal{L}(x, y) + \frac{\epsilon(t)}{2}\left(\|x\|^2 - \|y\|^2\right) \\
			&= f(x) + \langle Kx, y \rangle - g(y) + \frac{\epsilon(t)}{2}\left(\|x\|^2 - \|y\|^2\right).
		\end{aligned}
	\end{equation}
\end{definition}
One can readily verify that for every $x \in \mathbb{R}^n$, $\mathcal{L}_t(x, \cdot)$ is $\epsilon(t)$-strongly concave in $y$, and for every $y \in \mathbb{R}^m$, $\mathcal{L}_t(\cdot, y)$ is $\epsilon(t)$-strongly convex in $x$. Motivated by this property, we establish the following proposition.

\begin{proposition}\label{proposition1}
	For the augmented Lagrangian function $\mathcal{L}_t$ defined in \eqref{eq21}, the following inequalities hold:
	\begin{subequations}
		\begin{equation}\label{eq31}
			\langle \nabla_x \mathcal{L}_t(x(t), y^*), x^* - x(t) \rangle 
			\leq \mathcal{L}(x^*, y^*) - \mathcal{L}(x(t), y^*) + \frac{\epsilon(t)}{2}\left(\|x^*\|^2 - \|x(t)\|^2 - \|x(t) - x^*\|^2\right),
		\end{equation}
		\text{and}
		\begin{equation}\label{eq31'}
			\langle \nabla_y \mathcal{L}_t(x^*, y(t)), y^* - y(t) \rangle 
			\geq \mathcal{L}(x^*, y^*) - \mathcal{L}(x^*, y(t))  + \frac{\epsilon(t)}{2}\left(\|y(t)\|^2 - \|y^*\|^2 + \|y(t) - y^*\|^2\right).
		\end{equation}
	\end{subequations}
\end{proposition}
\begin{proof}
	By the $\epsilon(t)$-strong convexity of $\mathcal{L}_t(\cdot, y^*)$ with respect to $x$,
	for any $x, z \in \mathbb{R}^n$, we have
	\[
	\mathcal{L}_t(z, y^*) \geq \mathcal{L}_t(x, y^*) + \left\langle \nabla_x \mathcal{L}_t(x, y^*), z - x \right\rangle + \frac{\epsilon(t)}{2} \|z - x\|^2.
	\]
	Taking $x = x(t)$ and $z = x^*$, we rearrange the inequality to obtain
	\[
	\left\langle \nabla_x \mathcal{L}_t(x(t), y^*), x^* - x(t) \right\rangle
	\leq \mathcal{L}_t(x^*, y^*) - \mathcal{L}_t(x(t), y^*) - \frac{\epsilon(t)}{2} \|x^* - x(t)\|^2.
	\]
	Substituting into \eqref{eq21} yields inequality \eqref{eq31}.
	
	By the $\epsilon(t)$-strong concavity of $\mathcal{L}_t(x^*, \cdot)$ with respect to $y$, the same argument yields inequality \eqref{eq31'}. This completes the proof of Proposition \ref{proposition1}.
\end{proof}

\subsection{Oscillation Suppression Theory} \label{Lyapunov method}

Lyapunov \cite{wending1} introduced a groundbreaking framework that permits stability analysis without the explicit solution of system equations. By constructing a suitably defined energy-like function, one can perform energy decay estimation and thereby determine the stability of dynamical systems in a convenient manner. This powerful analytical tool is referred to as the Lyapunov function.

Drawing upon the classical theory of dissipative dynamical systems \cite{lilun1,ref46,lilun0,lilun2}, a tailored reformulation of the constituents of the energy function enables effective oscillation suppression. Specifically, by replacing the kinetic energy term $\|\dot{x}\|^2$ in the classical Lyapunov function with a velocity term incorporating Hessian-driven damping, the total derivative of the Lyapunov function along system trajectories acquires a geometric dissipation term induced by the Hessian-driven damping. This construction allows the precise characterization and control of its oscillation-suppressing effect. The feasibility of this approach has been verified in \cite{yueshu1,ref24,yueshu2,ref45}.

The essential idea underlying Lyapunov's second method (the direct method) \cite{wending1} is as follows: if there exists a positive definite function $V(x)$ whose total derivative $\dot{V}(x)$ along trajectories is negative definite (or negative semi-definite), then the asymptotic stability (or stability) of the equilibrium point can be concluded.

Subsequently, LaSalle \cite{wending2,wending3} proposed the invariance principle, which further relaxes the requirements: even when $\dot{V}(x) \leq 0$, the trajectories converge to the largest invariant set contained in $\{\,x \mid \dot{V}(x) = 0\,\}$, thereby significantly extending the applicability of Lyapunov stability theory.

When extending the analysis to non-autonomous systems, the Lyapunov function becomes an explicit function of time, denoted by $V(t,x)$. Standard non-autonomous stability theory still requires its total derivative to satisfy $\dot{V}(t,x) \leq 0$. The principal challenge lies not in constructing such a $V$, but in establishing asymptotic convergence under the condition $\dot{V} \leq 0$, since LaSalle's invariance principle is not directly applicable. For optimization dynamics, satisfying such stringent definiteness conditions can be particularly difficult. A generalized framework \cite{wending4,wending5} relaxes the derivative condition to
\begin{equation*}
	\dot{V}(t,x) \leq h(t),
\end{equation*}
where $h(t)$ is an integrable function such that $\int_{0}^{+\infty} h(t)\,\mathrm{d}t < +\infty$. Under this relaxed condition, the boundedness of $V$ can still be guaranteed, and the convergence of trajectories can be further analyzed under additional structural assumptions. This generalized technique has been widely adopted in the stability analysis of nonlinear non-autonomous dynamical systems and optimization dynamics \cite{shangjie1,ref22,ref44,hessian2,ref28,ref36}.

\section{Asymptotic Analysis}\label{sec3}
The existence and uniqueness of the global solution to the dynamical system \eqref{eq15} are proved in detail in Appendix \ref{Appendix A}. Based on the foregoing oscillation suppression theory via the Lyapunov approach, this chapter constructs a Lyapunov function involving the Hessian-driven damping term. By differentiating the established Lyapunov function and performing inequality estimation, the integrable upper bound is derived. Further, several convergence rates and integral estimates are established to quantitatively characterize the oscillation suppression performance. The above theoretical derivations are discussed in two distinct cases, corresponding to two different vanishing rates of the Tikhonov regularization parameter, which are denoted as $ \int_{t_0}^{+\infty} t \beta(t) \, \epsilon(t) \, \mathrm{d}t < +\infty $ and $ \int_{t_0}^{+\infty} \, \frac{\beta(t)\epsilon(t)}{t} \, \mathrm{d}t < +\infty $, respectively. We further verify the feasibility and validity of the obtained theoretical results through two examples.

\subsection{Case $ \int_{t_0}^{+\infty} \, t \beta(t) \epsilon(t) \, \mathrm{d}t < +\infty $}\label{subsec1}
In this subsection, we study the asymptotic behavior of the dynamical system \eqref{eq15} under the hypothesis of $ \int_{t_0}^{+\infty} \, t \beta(t) \epsilon(t) \, \mathrm{d}t < +\infty $, which means that the Tikhonov regularization parameter $ \epsilon(t) $ decreases rapidly to zero. For any fixed $(x^*,y^*) \in \Omega$, we define the energy function $ \mathcal{E} : [t_0, +\infty) \to \mathbb{R} $ by
\begin{equation}\label{eq23}
	\mathcal{E}(t) = \mathcal{E}_1(t) + \mathcal{E}_2(t) + \mathcal{E}_3(t),
\end{equation}
where
\begin{equation*}
	\begin{gathered}
		\mathcal{E}_1(t) = t^2\beta(t)\left(\mathcal{L}(x(t),y^*)-\mathcal{L}(x^*,y(t))+\frac{\epsilon(t)}{2}\left(\left\|x(t)\right\|^2 +\left\| y(t)\right\|^2\right)\right), \\
		\mathcal{E}_2(t) = \frac{1}{2} \left\| \eta \left( x(t) - x^* \right) + t \left( \dot{x}(t) + \gamma(t) \nabla_x \mathcal{L}_t \left( x(t), y(t) + \theta t \dot{y}(t) \right) \right) \right\|^2+\frac{n(t)}{2} \left\| x(t) - x^* \right\|^2,
	\end{gathered}
\end{equation*}
and
\begin{equation*}
	\mathcal{E}_3(t) = \frac{1}{2} \left\| \eta \left( y(t) - y^* \right) + t \left( \dot{y}(t) - \gamma(t) \nabla_y \mathcal{L}_t \left( x(t) + \theta t \dot{x}(t), y(t) \right) \right) \right\|^2 +\frac{n(t)}{2} \left\| y(t) - y^* \right\|^2. 
\end{equation*}
The function $\mathcal{E}(t)$ is called a time-varying Lyapunov function for the dynamical system \eqref{eq15}, as it is non-negative and bounded under the assumptions specified below:
\begin{equation}\label{Lyapunov1}
	\eta = \frac{t\beta(t)}{\theta(t\beta(t)-\gamma(t)-t\dot{\gamma}(t))},
\end{equation}
\begin{equation}\label{Lyapunov2}
	n(t) = \frac{t\beta(t)\left(\gamma(t)+t\dot{\gamma}(t)\right)}{\theta\gamma(t)(t\beta(t)-\gamma(t)-t\dot{\gamma}(t))}.
\end{equation}

\begin{lemma}\label{lemma1}
	Suppose that Assumption \eqref{eqH0} is satisfied, $f$ is $L_1$-smooth and $g$ is $L_3$-smooth, while $\nabla^2f$ and $\nabla^2g$ are $L_2$-Lipschitz and $L_4$-Lipschitz, respectively. Assume that $ \alpha,\gamma,\beta,\epsilon:[t_0, +\infty) \to (0, +\infty) $ are $\mathcal{C}^1$, $ \theta>0 $ and $ t_0>0 $ are constants, $K$ is a continuous linear operator and $K^*$ is its adjoint operator. In addition, $\gamma(t)^2\theta^2t^2 \notin \{-\frac{1}{\sigma_1},-\frac{1}{\sigma_2},...,-\frac{1}{\sigma_j}\}$, where $\{\sigma_1,\sigma_2,...,\sigma_j\}$ denotes the set of non-zero eigenvalues of $KK^*$ and $K^*K$. If the following conditions are satisfied
	\begin{equation}\label{tiaojian1}
		0 \leq \gamma(t)+t\dot{\gamma}(t) < t\beta(t),
	\end{equation}
	\begin{equation}\label{tiaojian2}
		\alpha(t) = \frac{\beta(t)}{\theta\left(t\beta(t)-\gamma(t)-t\dot{\gamma}(t)\right)} + \frac{\dot{\gamma}(t)}{\gamma(t)} + \frac{2}{t},
	\end{equation}
	then, for any trajectory $ (x(t),y(t)) $ of the dynamical system \eqref{eq15} and any primal-dual optimal solution $ (x^*,y^*) \in \Omega $ of the convex-concave bilinear saddle point problem \eqref{eq13}, we can get
	\begin{equation}\label{eq24}
		\begin{aligned}
			\dot{\mathcal{E}}(t) &\leq \left( 2t\beta(t)+t^2\dot{\beta}(t)-\frac{1}{\theta}t\beta(t) \right) \bigl(\mathcal{L}(x(t),y^*)-\mathcal{L}(x^*,y(t))\bigr) \\
			&\quad +\frac{1}{2} \left( \left( 2t\beta(t)+t^2\dot{\beta}(t)-\frac{1}{\theta}t\beta(t)\right)\epsilon(t)+t^2\beta(t)\dot{\epsilon}(t) \right)\left(\left\|x(t)\right\|^2+\left\|y(t)\right\|^2\right) \\
			&\quad -\frac{t\gamma(t)+t^2\dot{\gamma}(t)}{\gamma(t)}\left(\left\|\dot{x}(t)\right\|^2+\left\|\dot{y}(t)\right\|^2\right) + t\gamma(t)\left(\gamma(t)+t\dot{\gamma}(t)-t\beta(t)\right) \Delta(t) \\
			&\quad +\frac{1}{2}\left(\frac{t\beta(t)\left(\alpha(t) + t\dot{\alpha}(t)\right)}{\theta(t\beta(t)-\gamma(t)-t\dot{\gamma}(t))}-\frac{1}{\theta}t\beta(t)\epsilon(t)\right) \left(\left\|x(t)-x^*\right\|^2+\left\|y(t)-y^*\right\|^2\right) \\
			&\quad + \frac{1}{2\theta}t\beta(t)\epsilon(t) \left(\left\|x^*\right\|^2+\left\|y^*\right\|^2\right) ,
		\end{aligned}
	\end{equation}
	where $\Delta(t):= \left\|\nabla_x\mathcal{L}_t\left(x(t),y(t)+\theta t\dot{y}(t)\right)\right\|^2 + \left\|\nabla_y\mathcal{L}_t\left(x(t)+\theta t\dot{x}(t),y(t)\right)\right\|^2.$
\end{lemma}

\begin{proof}
	Firstly, the time derivative of $\mathcal{E}_1(t)$ is as follows:
	\begin{equation}\label{eq25} 
		\begin{aligned}
			\quad \dot{\mathcal{E}}_1(t) =& \left(2t\beta(t)+t^2\dot{\beta}(t)\right)\left( \mathcal{L}(x(t),y^*) - \mathcal{L}(x^*,y(t)) + \frac{\epsilon(t)}{2}( \left\|x(t)\right\|^2+\left\|y(t)\right\|^2 ) \right) \\
			&+t^2\beta(t)\bigg(\langle\nabla_x\mathcal{L}(x(t),y^*),\dot{x}(t)\rangle-\langle\nabla_y\mathcal{L}(x^*,y(t)),\dot{y}(t)\rangle  \\
			&+\frac{\dot{\epsilon}(t)}{2}\left(\left\|x(t)\right\|^2+\left\|y(t)\right\|^2\right)+\epsilon(t)\left(\langle x(t),\dot{x}(t)\rangle+\langle y(t),\dot{y}(t)\rangle\right)\bigg).
		\end{aligned}
	\end{equation}
	Next, we consider the function $ \mathcal{E}_2(t) $. Let 
	\begin{equation*}\label{eq26}
		\mu(t):=\eta \left( x(t) - x^* \right) + t \left( \dot{x}(t) + \gamma(t) \nabla_x \mathcal{L}_t \left( x(t), y(t) + \theta t \dot{y}(t)\right) \right).
	\end{equation*}
	Then, 
	\begin{equation*}\label{eq27}
		\begin{aligned}
			\dot{\mu}(t)&=\left(\eta+1\right)\dot{x}(t) + \gamma(t)\nabla_x\mathcal{L}_t\left(x(t),y(t)+\theta t\dot{y}(t)\right) \\
			&\quad + t\left(\ddot{x}(t) + \gamma(t)\frac{d}{dt}\nabla_x\mathcal{L}_t\left(x(t),y(t)+\theta t\dot{y}(t)\right) 
			+ \dot{\gamma}(t)\nabla_x\mathcal{L}_t\left(x(t),y(t)+\theta t\dot{y}(t)\right)\right) \\
			&=\left(\eta+1-t\alpha(t)\right)\dot{x}(t) + \left(\gamma(t)+t\dot{\gamma}(t)-t\beta(t)\right)\nabla_x\mathcal{L}_t(x(t),y(t)+\theta t\dot{y}(t)),
		\end{aligned}
	\end{equation*}
	where the second equality is obtained by combining the first equation in the dynamical system \eqref{eq15}.
	Therefore,
	\begin{equation}\label{eq29}
		\begin{aligned}
			\langle\mu(t),\dot{\mu}(t)\rangle &= \langle \eta \left( x(t) - x^* \right) + t \left( \dot{x}(t) + \gamma(t) \nabla_x \mathcal{L}_t \left( x(t), y(t) + \theta t \dot{y}(t) \right) \right),\\
			&\quad \left(\eta+1-t\alpha(t)\right)\dot{x}(t) +\left(\gamma(t)+t\dot{\gamma}(t)-t\beta(t)\right)\nabla_x\mathcal{L}_t\left(x(t),y(t)+\theta t\dot{y}(t)\right) \rangle \\ 
			&= t\left(\eta+1-t\alpha(t)\right)\left\|\dot{x}(t)\right\|^2 + \eta\left(\eta+1-t\alpha(t)\right)\langle x(t)-x^*,\dot{x}(t) \rangle \\ 
			&\quad +t\gamma(t)\left(\gamma(t)+t\dot{\gamma}(t)-t\beta(t)\right)\left\|\nabla_x\mathcal{{L}}_t\left(x(t),y(t)+\theta t\dot{y}(t)\right)\right\|^2 \\
			&\quad +\eta\left(\gamma(t)+t\dot{\gamma}(t)-t\beta(t)\right)\langle \nabla_x\mathcal{{L}}_t\left(x(t),y(t)+\theta t\dot{y}(t)\right),x(t)-x^* \rangle \\
			&\quad + B(t)\langle \nabla_x\mathcal{{L}}_t\left(x(t),y(t)+\theta t\dot{y}(t)\right),\dot{x}(t) \rangle \\
			&= t\left(\eta+1-t\alpha(t)\right)\left\|\dot{x}(t)\right\|^2 + \eta\left(\eta+1-t\alpha(t)\right)\langle x(t)-x^*,\dot{x}(t)\rangle \\
			&\quad +t\gamma(t)\left(\gamma(t)+t\dot{\gamma}(t)-t\beta(t)\right)\left\|\nabla_x\mathcal{{L}}_t\left(x(t),y(t)+\theta t\dot{y}(t)\right)\right\|^2 \\
			&\quad +\eta\left(\gamma(t)+t\dot{\gamma}(t)-t\beta(t)\right)\langle \nabla_x\mathcal{{L}}\left(x(t),y^*\right)+\epsilon(t)x(t),x(t)-x^* \rangle \\
			&\quad +\eta\left(\gamma(t)+t\dot{\gamma}(t)-t\beta(t)\right)\langle K^*\left(y(t)-y^*+\theta t\dot{y}(t)\right),x(t)-x^* \rangle\\
			&\quad + B(t) \langle \nabla_x\mathcal{{L}}\left(x(t),y^*\right)+\epsilon(t)x(t),\dot{x}(t) \rangle + B(t) \langle K^*\left(y(t)-y^*+\theta t\dot{y}(t)\right),\dot{x}(t) \rangle ,
		\end{aligned}
	\end{equation}
	where we denote $B(t):= t\left(\gamma(t)+t\dot{\gamma}(t)-t\beta(t)\right) + t\gamma(t)\left(\eta+1-t\alpha(t)\right)$, the last equality follows from
	\begin{equation}\label{eq30}
		\nabla_x\mathcal{L}_t\left(x(t),y(t)+\theta t\dot{y}(t)\right) = \nabla_x\mathcal{L}\left(x(t),y^*\right)+\epsilon(t)x(t)+K^*\left(y(t)-y^*+\theta t\dot{y}(t)\right).
	\end{equation}
	By \eqref{Lyapunov1}, we can get
	\begin{equation}\label{eq32}
		\eta\left(\gamma(t)+t\dot{\gamma}(t)-t\beta(t)\right)=-\frac{1}{\theta}t\beta(t) \leq 0, \quad\forall t \geq t_0.
	\end{equation} 
	Combining \eqref{eq31}, \eqref{eq29} and \eqref{eq32}, for all $ t \geq t_0 $, we have
	\begin{equation}\label{eq34}
		\begin{aligned}
			\langle \mu(t),\dot{\mu}(t) \rangle &\leq t\left(\eta+1-t\alpha(t)\right)\left\|\dot{x}(t)\right\|^2 + \eta\left(\eta+1-t\alpha(t)\right)\langle x(t)-x^*,\dot{x}(t)\rangle \\
			&\quad +t\gamma(t)\left(\gamma(t)+t\dot{\gamma}(t)-t\beta(t)\right)\left\|\nabla_x\mathcal{{L}}_t\left(x(t),y(t)+\theta t\dot{y}(t)\right)\right\|^2 \\
			&\quad -\frac{1}{\theta}t\beta(t) \left( \mathcal{L}\left(x(t),y^*\right)-\mathcal{L}\left(x^*,y^*\right)+\frac{\epsilon(t)}{2}\left(\left\|x(t)\right\|^2-\left\|x^*\right\|^2+\left\|x(t)-x^*\right\|^2\right) \right) \\
			&\quad -\frac{1}{\theta}t\beta(t) \langle K^*\left(y(t)-y^*+\theta t\dot{y}(t)\right),x(t)-x^*\rangle\\
			&\quad + B(t) \langle \nabla_x\mathcal{{L}}\left(x(t),y^*\right)+\epsilon(t)x(t),\dot{x}(t)\rangle + B(t) \langle K^*\left(y(t)-y^*+\theta t\dot{y}(t)\right),\dot{x}(t)\rangle .
		\end{aligned}
	\end{equation}
	Moreover,
	\begin{equation}\label{eq35}
		\frac{d}{dt}\left( \frac{n(t)}{2}\left\|x(t)-x^*\right\|^2 \right) = \frac{\dot{n}(t)}{2}\left\|x(t)-x^*\right\|^2+n(t)\langle x(t)-x^*,\dot{x}(t)\rangle.
	\end{equation}
	According to \eqref{eq34} and \eqref{eq35}, we can obtain
	\begin{equation}\label{eq36}
		\begin{aligned}
			\dot{\mathcal{E}}_2(t) &= \langle \mu(t),\dot{\mu}(t)\rangle +\frac{d}{dt}\left( \frac{n(t)}{2}\left\|x(t)-x^*\right\|^2 \right) \\
			&\leq t\left(\eta+1-t\alpha(t)\right)\left\|\dot{x}(t)\right\|^2 + \frac{\dot{n}(t)}{2}\left\|x(t)-x^*\right\|^2 \\
			&\quad +t\gamma(t)\left(\gamma(t)+t\dot{\gamma}(t)-t\beta(t)\right)\left\|\nabla_x\mathcal{{L}}_t\left(x(t),y(t)+\theta t\dot{y}(t)\right)\right\|^2 \\
			&\quad -\frac{1}{\theta}t\beta(t) \left( \mathcal{L}\left(x(t),y^*\right)-\mathcal{L}\left(x^*,y^*\right)+\frac{\epsilon(t)}{2}\left(\left\|x(t)\right\|^2-\left\|x^*\right\|^2+\left\|x(t)-x^*\right\|^2\right) \right) \\
			&\quad -\frac{1}{\theta}t\beta(t) \langle K^*\left(y(t)-y^*+\theta t\dot{y}(t)\right),x(t)-x^*\rangle +\left( n(t)+\eta(\eta+1-t\alpha(t)) \right)\langle x(t)-x^*,\dot{x}(t) \rangle \\
			&\quad + B(t) \langle \nabla_x\mathcal{{L}}\left(x(t),y^*\right)+\epsilon(t)x(t),\dot{x}(t)\rangle + B(t) \langle K^*\left(y(t)-y^*+\theta t\dot{y}(t)\right),\dot{x}(t)\rangle ,
		\end{aligned}
	\end{equation}
	for all $ t \geq t_0 $.
	
	Similarly, for all $ t \geq t_0 $, we have
	\begin{equation}\label{eq37}
		\begin{aligned}
			\dot{\mathcal{E}}_3(t) &\leq t\left(\eta+1-t\alpha(t)\right)\left\|\dot{y}(t)\right\|^2  +  \frac{\dot{n}(t)}{2}\left\|y(t)-y^*\right\|^2 \\
			&\quad +t\gamma(t)\left(\gamma(t)+t\dot{\gamma}(t)-t\beta(t)\right)\left\|\nabla_y\mathcal{{L}}_t\left(x(t)+\theta t\dot{x}(t),y(t)\right)\right\|^2 \\
			&\quad -\frac{1}{\theta}t\beta(t) \left( \mathcal{L}\left(x^*,y^*\right)-\mathcal{L}\left(x^*,y(t)\right)+\frac{\epsilon(t)}{2}\left(\left\|y(t)\right\|^2-\left\|y^*\right\|^2+\left\|y(t)-y^*\right\|^2\right) \right) \\
			&\quad +\frac{1}{\theta}t\beta(t)\langle K\left(x(t)-x^*+\theta t\dot{x}(t)\right),y(t)-y^*\rangle + \left( n(t)+\eta(\eta+1-t\alpha(t)) \right)\langle y(t)-y^*,\dot{y}(t) \rangle \\
			&\quad - B(t) \langle \nabla_y\mathcal{{L}}\left(x^*,y(t)\right)-\epsilon(t)y(t),\dot{y}(t)\rangle - B(t) \langle K\left(x(t)-x^*+\theta t\dot{x}(t)\right),\dot{y}(t)\rangle .
		\end{aligned}
	\end{equation}
	
	Combining  \eqref{eq23}, \eqref{eq25}, \eqref{eq36} and \eqref{eq37}, we get
	\begin{equation}\label{eq38}
		\begin{aligned}
			\dot{\mathcal{E}}(t) &\leq \frac{1}{2\theta}t\beta(t)\epsilon(t) \left(\left\|x^*\right\|^2+\left\|y^*\right\|^2\right) + \left( 2t\beta(t)+t^2\dot{\beta}(t)-\frac{1}{\theta}t\beta(t) \right) \bigl(\mathcal{L}(x(t),y^*)-\mathcal{L}(x^*,y(t))\bigr) \\
			&\quad +\frac{1}{2} \left( \left( 2t\beta(t)+t^2\dot{\beta}(t)-\frac{1}{\theta}t\beta(t)\right)\epsilon(t)+t^2\beta(t)\dot{\epsilon}(t) \right)\left(\left\|x(t)\right\|^2+\left\|y(t)\right\|^2\right) \\
			&\quad +t \left(\eta+1-t\alpha(t)\right)\left(\left\|\dot{x}(t)\right\|^2+\left\|\dot{y}(t)\right\|^2\right) + t\gamma(t)\left(\gamma(t)+t\dot{\gamma}(t)-t\beta(t)\right) \Delta(t) \\
			&\quad +\frac{1}{2}\left(\dot{n}(t)-\frac{1}{\theta}t\beta(t)\epsilon(t)\right) \left(\left\|x(t)-x^*\right\|^2+\left\|y(t)-y^*\right\|^2\right) \\
			&\quad +\left(n(t)+\eta(\eta+1-t\alpha(t))\right) \left(\langle x(t)-x^*,\dot{x}(t)\rangle + \langle y(t)-y^*,\dot{y}(t)\rangle \right) \\
			&\quad +\left( t^2\beta(t)+B(t) \right) \left( \langle\nabla_x\mathcal{L}\left(x(t),y^*\right),\dot{x}(t)\rangle - \langle\nabla_y\mathcal{L}\left(x^*,y(t)\right),\dot{y}(t)\rangle \right) \\
			&\quad -\left(t^2\beta(t)+B(t)\right) \left(\langle K^*\dot{y}(t),x(t)-x^*\rangle - \langle K\dot{x}(t),y(t)-y^*\rangle \right) \\
			&\quad +\left( t^2\beta(t)+B(t) \right)\epsilon(t) \left(\langle x(t),\dot{x}(t)\rangle + \langle y(t),\dot{y}(t)\rangle \right) .
		\end{aligned}
	\end{equation}
	Furthermore, based on the specified formulation \eqref{Lyapunov1} for $\eta$, \eqref{Lyapunov2} for $n(t)$, and the assumptions \eqref{tiaojian1} and \eqref{tiaojian2} in the Lemma \ref{lemma1}, it can be inferred that the coefficients of the last four terms in \eqref{eq38} are zero. Moreover, we can obtain the expressions
	\begin{equation*}
		\begin{gathered}
			\eta+1-t\alpha(t) = -\frac{\gamma(t)+t\dot{\gamma}(t)}{\gamma(t)}, \\
			\dot{n}(t) = \frac{t\beta(t)\left(\alpha(t) + t\dot{\alpha}(t)\right)}{\theta(t\beta(t)-\gamma(t)-t\dot{\gamma}(t))}.
		\end{gathered}
	\end{equation*}
	In summary, the expression \eqref{eq24} can be derived based on \eqref{eq38}.
	
	This completes the proof of Lemma \ref{lemma1}.
\end{proof}

\begin{theorem}\label{theorem1}
	Let all hypotheses in Lemma \ref{lemma1} hold. Suppose that  $ \lim\limits_{t \to +\infty} \, t^2\beta(t) = +\infty $ and $ \epsilon(t) $ is a non-increasing function such that $\int_{t_0}^{+\infty} \, t\beta(t)\epsilon(t) \, \mathrm{d}t < +\infty$. Furthermore, for any $ t \geq t_0 $, the following inequalities are valid
	\begin{equation}\label{tiaojian3}
		\frac{\dot{\beta}(t)}{\beta(t)} \leq \frac{1-2\theta}{\theta t},
	\end{equation}
	\begin{equation}\label{tiaojian4}
		\alpha(t)+t\dot{\alpha}(t) \leq  \left( t\beta(t)-\gamma(t)-t\dot{\gamma}(t) \right)\epsilon(t).
	\end{equation}
	Then, for any trajectory $ (x(t), y(t)) $ of the dynamical system \eqref{eq15} and any $ (x^*,y^*) \in \Omega $ of the convex-concave bilinear saddle point problem \eqref{eq13}, the following conclusions hold: \\
	\textbf{(i)\;(Boundedness of Trajectory)} If there exists a constant $D>0$ such that 
	\begin{equation}\label{youjie}
		\frac{t\beta(t)\left(\gamma(t)+t\dot{\gamma}(t)\right)}{\gamma(t)(t\beta(t)-\gamma(t)-t\dot{\gamma}(t))} \geq D,
	\end{equation}  
	then the trajectory $(x(t), y(t))$ is bounded on $[t_0, +\infty)$;
	\vspace{5pt}
	\\
	\textbf{(ii)\;(Pointwise Estimates)}
	\begin{itemize}
		\item[$\cdot$] $\mathcal{L}\left(x(t),y^*\right) - 	\mathcal{L}\left(x^*,y(t)\right) = \mathcal{O}\left( \dfrac{1}{t^2\beta(t)} \right),$ \\[2pt]
		\item[$\cdot$] $
		\left\| \nabla f(x(t)) - \nabla f(x^*) \right\| = \mathcal{O}\left( \frac{1}{t \sqrt{\beta(t)}} \right), \quad \left\| \nabla g(y(t)) - \nabla g(y^*) \right\| = \mathcal{O}\left( \frac{1}{t \sqrt{\beta(t)}} \right),$ \\[2pt]
		\item[$\cdot$] If \eqref{youjie} holds, then \begin{equation}\label{yizhizhendang1}\begin{cases}
				\left\|\dot{x}(t) + \gamma(t) \nabla_x \mathcal{L}_t \left(x(t), y(t) + \theta t \dot{y}(t)\right) \right\| = \mathcal{O}\left( \frac{1}{t} \right), \\
				\left\|\dot{y}(t) - \gamma(t) \nabla_y \mathcal{L}_t \left(x(t) + \theta t \dot{x}(t), y(t)\right) \right\| = \mathcal{O}\left( \frac{1}{t} \right);
			\end{cases}
		\end{equation}
	\end{itemize}
	\vspace{5pt}
	\textbf{(iii)\;(Integral Estimates)} 
	\begin{itemize}
		\item[$\cdot$] $\int_{t_0}^{+\infty}\, t\gamma(t)\left(t\beta(t)-\gamma(t)-t\dot{\gamma}(t)\right) \Delta(t) \, \mathrm{d}t < +\infty,$ \\
		where $\Delta(t):= \left\|\nabla_x\mathcal{L}_t\left(x(t),y(t)+\theta t\dot{y}(t)\right)\right\|^2 + \left\|\nabla_y\mathcal{L}_t\left(x(t)+\theta t\dot{x}(t),y(t)\right)\right\|^2,$ \\[2pt]
		\item[$\cdot$] If $(2\theta-1)\beta(t)+\theta t\dot{\beta}(t) < 0$, $\forall t \geq t_0$, then
		\begin{equation*}
			\begin{cases}
				\int_{t_0}^{+\infty} \, \left( (1 - 2\theta)\beta(t) - \theta t \dot{\beta}(t) \right) t \left(\mathcal{L}\left(x(t),y^*\right)-\mathcal{L}\left(x^*,y(t)\right)\right) \, \mathrm{d}t < +\infty, \\
				\int_{t_0}^{+\infty} \, \left( (1 - 2\theta)\beta(t) - \theta t \dot{\beta}(t) \right) t \left\| \nabla f(x(t)) - \nabla f(x^*) \right\|^2 \, \mathrm{d}t < +\infty, \\
				\int_{t_0}^{+\infty} \, \left( (1 - 2\theta)\beta(t) - \theta t \dot{\beta}(t) \right) t \left\| \nabla g(y(t)) - \nabla g(y^*) \right\|^2 \, \mathrm{d}t < +\infty;
			\end{cases}
		\end{equation*}
		and if $0<\gamma(t)+t\dot{\gamma}(t)$, $\forall t \geq t_0$, then $\int_{t_0}^{+\infty} \, \frac{t\left(\gamma(t)+t\dot{\gamma}(t)\right)}{\gamma(t)} \left( \left\| \dot{x}(t) \right\|^2 + \left\| \dot{y}(t) \right\|^2 \right) \, \mathrm{d}t < +\infty$.
	\end{itemize}
\end{theorem}

\begin{proof}
	According to Lemma \ref{lemma1}, integrating \eqref{eq24} from $ t_0 $ to $ t $ yields
	\begin{equation}\label{eq50}
		\begin{aligned}
			\mathcal{E}(t) &- \int_{t_0}^{t} \, \left( 2\tau\beta(\tau)+\tau^2\dot{\beta}(\tau)-\frac{1}{\theta}\tau\beta(\tau) \right)  \left(\mathcal{L}\left(x(\tau),y^*\right)-\mathcal{L}\left(x^*,y(\tau)\right)\right) \, \mathrm{d}\tau \\
			&- \int_{t_0}^{t} \, \frac{1}{2}\left( \left( 2\tau\beta(\tau)+\tau^2\dot{\beta}(\tau)-\frac{1}{\theta}\tau\beta(\tau)\right)\epsilon(\tau)+\tau^2\beta(\tau)\dot{\epsilon}(\tau) \right)\left(\left\|x(\tau)\right\|^2+\left\|y(\tau)\right\|^2\right) \, \mathrm{d}\tau \\
			&+ \int_{t_0}^{t} \, \frac{\tau\gamma(\tau)+\tau^2\dot{\gamma}(\tau)}{\gamma(\tau)}\left(\left\|\dot{x}(\tau)\right\|^2+\left\|\dot{y}(\tau)\right\|^2\right) \, \mathrm{d}\tau - \int_{t_0}^{t} \, \tau\gamma(\tau)\left(\gamma(\tau)+\tau\dot{\gamma}(\tau)-\tau\beta(\tau)\right) \Delta(\tau) \, \mathrm{d}\tau \\
			&- \int_{t_0}^{t} \, \frac{1}{2}\left(\frac{\tau\beta(\tau)\left(\alpha(\tau) + \tau\dot{\alpha}(\tau)\right)}{\theta(\tau\beta(\tau)-\gamma(\tau)-\tau\dot{\gamma}(\tau))}-\frac{1}{\theta}\tau\beta(\tau)\epsilon(\tau)\right) \left(\left\|x(\tau)-x^*\right\|^2+\left\|y(\tau)-y^*\right\|^2\right) \, \mathrm{d}\tau \\
			&\leq \mathcal{E}(t_0)	+ \int_{t_0}^{t} \, \frac{1}{2\theta}\tau\beta(\tau)\epsilon(\tau) \left(\left\|x^*\right\|^2+\left\|y^*\right\|^2\right) \, \mathrm{d}\tau.
		\end{aligned}
	\end{equation}
	Based on the assumptions \eqref{tiaojian1}, \eqref{tiaojian2}, \eqref{tiaojian3} and \eqref{tiaojian4}, it is straightforward to derive the following inequalities
	\begin{flalign}
		&\ (a)\, 2t\beta(t)+t^2\dot{\beta}(t)-\frac{1}{\theta}t\beta(t) \leq 0;&\nonumber\\
		&\ (b)\, \frac{t\beta(t)\left(\alpha(t) + t\dot{\alpha}(t)\right)}{\theta(t\beta(t)-\gamma(t)-t\dot{\gamma}(t))}-\frac{1}{\theta}t\beta(t)\epsilon(t) \leq 0;&\nonumber\\
		&\ (c)\, \left(2t\beta(t)+t^2\dot{\beta}(t)-\frac{1}{\theta}t\beta(t)\right)\epsilon(t)+t^2\beta(t)\dot{\epsilon}(t) \leq 0;&\nonumber\\
		&\ (d)\, -\frac{t\gamma(t)+t^2\dot{\gamma}(t)}{\gamma(t)} \leq 0;&
		\nonumber\\
		&\ (e)\, t\gamma(t)\left(\gamma(t)+t\dot{\gamma}(t)-t\beta(t)\right) < 0,&\nonumber
	\end{flalign}
	where the inequalities $(a)$ and $(b)$ hold from \eqref{tiaojian3} and \eqref{tiaojian4}, the inequality $(c)$ is satisfied since $ \epsilon(t) $ is a non-negative function which satisfies $\dot{\epsilon}(t) \leq 0$, $(d)$ and $(e)$ hold due to the assumptions \eqref{tiaojian1}. By \eqref{eq19}, we have $ \mathcal{L}(x(t),y^*)-\mathcal{L}(x^*,y(t)) \geq 0 $. This together with the assumption $ \int_{t_0}^{+\infty}\,t\beta(t)\epsilon(t) \, \mathrm{d}t<+\infty $ and \eqref{eq50} implies that $ \mathcal{E}(t) $ is bounded on $ [t_0,+\infty) $ (i.e., $ \mathcal{E}(t) $ is a Lyapunov function), that is, there exists $ \tilde{D} \geq 0 $ such that the following formula holds
	\begin{equation*}\label{eq52}
		\mathcal{E}(t) \leq \mathcal{E}(t_0) + \int_{t_0}^{t} \, \frac{1}{2\theta}\tau\beta(\tau)\epsilon(\tau) \left(\|x^*\|^2+\|y^*\|^2\right) \, \mathrm{d}\tau \leq \tilde{D}, \quad \forall t \geq t_0.
	\end{equation*}
	Taking \eqref{eq23} into account, it is clear that 
	\begin{equation*}\label{eq53}
		\left\| \eta\left(x(t)-x^*\right)+t\left(\dot{x}(t)+\gamma(t)\nabla_x\mathcal{L}_t\left(x(t),y(t)+\theta t\dot{y}(t)\right)\right) \right\|
	\end{equation*}
	and
	\begin{equation*}\label{eq54}
		\left\| \eta\left(y(t)-y^*\right)+t\left(\dot{y}(t)-\gamma(t)\nabla_y\mathcal{L}_t\left(x(t)+\theta t\dot{x}(t),y(t)\right)\right) \right\|
	\end{equation*}
	are bounded for all $ t \geq t_0 $. And according to the known conditions \eqref{youjie}, we can get that the trajectory $ (x(t), y(t)) $ is bounded, i.e., 
	\begin{equation*}
		\frac{D}{2\theta}\left\|x(t)-x^*\right\|^2 \leq \frac{n(t)}{2}\left\|x(t)-x^*\right\|^2 \leq \mathcal{E}(t) < +\infty,
	\end{equation*}
	and
	\begin{equation*}
		\frac{D}{2\theta}\left\|y(t)-y^*\right\|^2 \leq \frac{n(t)}{2}\left\|y(t)-y^*\right\|^2 \leq \mathcal{E}(t) < +\infty.
	\end{equation*}
	Note that
	\begin{equation*}\label{eq55}
		\begin{aligned}
			& \quad \left\| t\left(\dot{x}(t)+\gamma(t)\nabla_x\mathcal{L}_t\left(x(t),y(t)+\theta t\dot{y}(t)\right)\right) \right\|^2 \\
			&\leq 2\left\| \eta\left(x(t)-x^*\right) \right\|^2 + 2\left\| \eta\left(x(t)-x^*\right)+t\left(\dot{x}(t)+\gamma(t)\nabla_x\mathcal{L}_t\left(x(t),y(t)+\theta t\dot{y}(t)\right)\right) \right\|^2.
		\end{aligned}	
	\end{equation*}
	Then, from the boundedness of $ \left\| \eta\left(x(t)-x^*\right)+t\left(\dot{x}(t)+\gamma(t)\mathcal{L}_t\left(x(t),y(t)+\theta t\dot{y}(t)\right)\right) \right\| $ and the trajectory $ (x(t),y(t)) $, we can duduce that
	\begin{equation*}\label{eq56}
		\left\| \dot{x}(t)+\gamma(t)\nabla_x\mathcal{L}_t\left(x(t),y(t)+\theta t\dot{y}(t)\right) \right\| = \mathcal{O}\left( \frac{1}{t} \right).
	\end{equation*}
	Similarly, we can show that
	\begin{equation*}\label{eq57}
		\left\| \dot{y}(t)-\gamma(t)\nabla_y\mathcal{L}_t\left(x(t)+\theta t\dot{x}(t),y(t)\right) \right\| = \mathcal{O}\left( \frac{1}{t} \right).
	\end{equation*}
	Moreover, by the definition of $ \mathcal{E}(t) $, it follows that
	\begin{equation*}\label{eq58}
		t^2\beta(t) \left(\mathcal{L}\left(x(t),y^*\right) - \mathcal{L}\left(x^*,y(t)\right)\right) \leq \mathcal{E}(t),
	\end{equation*}
	which together with the boundedness of $ \mathcal{E}(t) $ and the assumption $ \lim\limits_{t \to +\infty} \, t^2\beta(t) = +\infty $ that
	\begin{equation}\label{eq59}
		\mathcal{L}\left(x(t),y^*\right) - \mathcal{L}\left(x^*,y(t)\right) = \mathcal{O}\left(\frac{1}{t^2\beta(t)}\right).
	\end{equation}
	Since $f$ is convex and $\nabla f$ is $L_1$-Lipschitz continuous, $g$ is concave and $\nabla g$ is $L_3$-Lipschitz continuous, then there exists positive constant $L$ such that
	\begin{equation*}
		f(x(t)) \geq f(x^*) + \langle \nabla f(x^*), x(t) - x^* \rangle + \frac{1}{2L} \left\|\nabla f(x(t)) - \nabla f(x^*)\right\|^2
	\end{equation*}
	and
	\begin{equation*}
		g(y(t)) \leq g(y^*) + \langle \nabla g(y^*), y(t) - y^* \rangle + \frac{1}{2L} \left\|\nabla g(y(t)) - \nabla g(y^*)\right\|^2.
	\end{equation*}
	Thus, according to the definition of $\mathcal{L}$ in \eqref{eq13}, we get that
	\begin{equation}\label{new1}
		\begin{aligned}
			\mathcal{L}\left(x(t), y^*\right) - \mathcal{L}\left(x^*, y(t)\right) &= f(x(t)) - f(x^*) + g(y(t)) - g(y^*) + \langle Kx(t), y^* \rangle - \langle Kx^*, y(t) \rangle  \\
			&\geq \langle \nabla f(x^*), x(t) - x^* \rangle + \frac{1}{2L} \left\| \nabla f(x(t)) - \nabla f(x^*) \right\|^2 + \langle Kx(t),y^* \rangle \\
			&\quad - \langle \nabla g(y^*), y^* - y(t) \rangle + \frac{1}{2L} \left\| \nabla g(y(t)) - \nabla g(y^*) \right\|^2 - \langle Kx^*,y(t) \rangle \\
			&= \frac{1}{2L} \left\| \nabla f(x(t)) - \nabla f(x^*) \right\|^2 + \frac{1}{2L} \left\| \nabla g(y(t)) - \nabla g(y^*) \right\|^2,
		\end{aligned}
	\end{equation}
	where the last equality follows from the optimality conditions \eqref{eq20}. By combining \eqref{eq59} and \eqref{new1}, we can obtain that
	\begin{equation*}
		\left\| \nabla f(x(t)) - \nabla f(x^*) \right\| = \mathcal{O}\left( \frac{1}{t \sqrt{\beta(t)}} \right),\quad \left\| \nabla g(y(t)) - \nabla g(y^*) \right\| = \mathcal{O}\left( \frac{1}{t \sqrt{\beta(t)}} \right).
	\end{equation*}
	
	On the other hand, it is clear from \eqref{eq50} that
	\begin{equation}\label{eq60}
		\begin{aligned}
			& \quad \mathcal{E}(t) - \int_{t_0}^{t} \, \left( 2\tau\beta(\tau)+\tau^2\dot{\beta}(\tau)-\frac{1}{\theta}\tau\beta(\tau) \right)  \left(\mathcal{L}\left(x(\tau),y^*\right)-\mathcal{L}\left(x^*,y(\tau)\right)\right) \, \mathrm{d}\tau \\
			&\quad +\int_{t_0}^{t} \, \frac{\tau\gamma(\tau)+\tau^2\dot{\gamma}(\tau)}{\gamma(\tau)}\left(\left\|\dot{x}(\tau)\right\|^2+\left\|\dot{y}(\tau)\right\|^2\right)\, \mathrm{d}\tau \\
			&\quad - \int_{t_0}^{t} \, \tau\gamma(\tau)\left(\gamma(\tau)+\tau\dot{\gamma}(\tau)-\tau\beta(\tau)\right) \Delta(\tau) \, \mathrm{d}\tau \\
			&\leq \mathcal{E}(t_0) + \int_{t_0}^{t} \, \frac{1}{2\theta} \tau\beta(\tau)\epsilon(\tau) \left(\left\|x^*\right\|^2+ \left\|y^*\right\|^2\right)\, \mathrm{d}\tau,\quad \forall t \geq t_0.
		\end{aligned}
	\end{equation}
	Combining this with $ \int_{t_0}^{+\infty} \, t\beta(t)\epsilon(t) \, \mathrm{d}t < +\infty $ and noting that $ \mathcal{E}(t) \geq 0 $, for all $ t \geq t_0 $, we have
	\begin{equation*}
		\int_{t_0}^{+\infty} \, t\gamma(t)\left(t\beta(t)-\gamma(t)-t\dot{\gamma}(t)\right) \Delta(t) \, \mathrm{d}t < +\infty.
	\end{equation*}
	Furthermore, if $(2\theta-1)\beta(t)+\theta t\dot{\beta}(t) < 0$, $\forall t \geq t_0$, we obtain from \eqref{eq60} that
	\begin{equation}\label{new2}
		\int_{t_0}^{+\infty} \, \left( (1 - 2\theta)\beta(t) - \theta t \dot{\beta}(t) \right) t \left(\mathcal{L}\left(x(t),y^*\right)-\mathcal{L}\left(x^*,y(t)\right)\right) \, \mathrm{d}t < +\infty.
	\end{equation}
	Combining \eqref{new1} and \eqref{new2}, we can get that
	\begin{equation*}
		\int_{t_0}^{+\infty} \, \left( (1 - 2\theta)\beta(t) - \theta t \dot{\beta}(t) \right) t \left\| \nabla f(x(t)) - \nabla f(x^*) \right\|^2 \, \mathrm{d}t < +\infty
	\end{equation*}
	and
	\begin{equation*}
		\int_{t_0}^{+\infty} \, \left( (1 - 2\theta)\beta(t) - \theta t \dot{\beta}(t) \right) t \left\| \nabla g(y(t)) - \nabla g(y^*) \right\|^2 \, \mathrm{d}t < +\infty.
	\end{equation*}
	Moreover, if for all $t \geq t_0$, $0 < \gamma(t)+t\dot{\gamma}(t)$, we have
	\begin{equation*}
		\int_{t_0}^{+\infty} \, \frac{t\left(\gamma(t)+t\dot{\gamma}(t)\right)}{\gamma(t)} \left( \left\| \dot{x}(t) \right\|^2 + \left\| \dot{y}(t) \right\|^2 \right) \, \mathrm{d}t < +\infty.
	\end{equation*}
	
	This completes the proof of Theorem \ref{theorem1}.
\end{proof}

\begin{remark}
	Compared with the classical Lyapunov stability analysis framework \cite{wending1}, the approach adopted in this paper can be regarded as its natural generalization. In the first case studied in Section \ref{sec3}, we first construct an appropriate Lyapunov function $\mathcal{E}(t)$ and verify its nonnegativity, boundedness, and the dissipative structure of its time derivative. Based on these results, we prove that the primal-dual gap and the velocity-gradient coupling terms decay over time, and rigorously derive the corresponding convergence rates and integral estimates.
\end{remark}

\begin{remark}[\textbf{Oscillation Decay Rate Estimate}]\label{remark_zhendang}
	The Lyapunov method developed in subsection \ref{Lyapunov method} offers rigorous theoretical support for the oscillation-suppressing effect of Hessian-driven damping. For this work, the detailed interpretation is given below:
	
	\begin{description}
		\item[\textbf{[Dissipative Term]}] As shown in \eqref{eq24} of Lemma \ref{lemma1}, $\mathcal{E}(t)$ contains $t\gamma(t)\left(\gamma(t)+t\dot{\gamma}(t)-t\beta(t)\right)\Delta(t) < 0$ associated with Hessian-driven damping. This shows that $\Delta(t)$ is directly associated with the energy dissipation of $\mathcal{E}(t)$. As a consequence, the oscillation amplitude decays, and undesired oscillations are thus effectively suppressed.
		\item[\textbf{[Consistency Error Estimate]}] Theorem \ref{theorem1} verifies that the trajectory generated by the dynamical system \eqref{eq15} satisfies \eqref{yizhizhendang1}. This indicates that Hessian-driven damping imposes a strict coupling between the velocity variation and the instantaneous gradient direction. Consequently, it acts directly on $\dot{x}(t)$ to prohibit abrupt changes, which effectively suppresses oscillatory behavior.
		\item[\textbf{[Bounded Integral Estimate]}] Theorem \ref{theorem1} proves that $\int_{t_0}^{+\infty} \, t\gamma(t)\left(t\beta(t)-\gamma(t)-t\dot{\gamma}(t)\right)\Delta(t) \mathrm{d}t < +\infty$, indicating that the accumulated energy of the Hessian-driven damping term over time is finite. This guarantees that oscillations decay quickly and cannot be sustained.
	\end{description}
\end{remark}

\subsection{Case $ \int_{t_0}^{+\infty} \, \frac{\beta(t)\epsilon(t)}{t} \, \mathrm{d}t < +\infty $}\label{subsec2}
In this subsection, we analyze the asymptotic behavior of the dynamical system \eqref{eq15} under the hypothesis of $ \int_{t_0}^{+\infty} \, \frac{\beta(t)\epsilon(t)}{t}\, \mathrm{d}t < +\infty $, which means that the Tikhonov regularization parameter $ \epsilon(t) $ decreases slowly to zero. For any fixed $(x^*,y^*) \in \Omega$, the energy function $ \bar{\mathcal{E}}(t) $ is defined as
\begin{equation}\label{eq62}
	\bar{\mathcal{E}}(t) = \bar{\mathcal{E}}_1(t) + \bar{\mathcal{E}}_2(t) + \bar{\mathcal{E}}_3(t),
\end{equation}
where
\begin{equation*}
	\begin{gathered}
		\bar{\mathcal{E}}_1(t) = \beta(t) \left(\mathcal{L}\left(x(t),y^*\right)-\mathcal{L}\left(x^*,y(t)\right)+\frac{\epsilon(t)}{2}\left(\left\|x(t)\right\|^2+\left\|y(t)\right\|^2\right)\right), \\
		\bar{\mathcal{E}}_2(t) = \frac{1}{2} \left\| \bar{\eta}(t)\left(x(t)-x^*\right)+\dot{x}(t)+\gamma(t)\nabla_x\mathcal{L}_t\left(x(t),y(t)+\theta t\dot{y}(t)\right) \right\|^2 + \frac{\bar{n}(t)}{2}\left\|x(t)-x^*\right\|^2,
	\end{gathered}
\end{equation*}
and
\begin{equation*}
	\bar{\mathcal{E}}_3(t) = \frac{1}{2} \left\| \bar{\eta}(t)\left(y(t)-y^*\right)+\dot{y}(t)-\gamma(t)\nabla_y\mathcal{L}_t\left(x(t)+\theta t\dot{x}(t),y(t)\right) \right\|^2 + \frac{\bar{n}(t)}{2}\left\|y(t)-y^*\right\|^2.
\end{equation*}
We can conclude that \(\bar{\mathcal{E}}(t)\) is a Lyapunov function for the dynamical system \eqref{eq15} when it satisfies the following conditions:
\begin{equation}\label{Lyapunov3}
	\bar{\eta}(t) = \frac{\beta(t)}{\theta(t\beta(t)-\gamma(t)-t\dot{\gamma}(t))},
\end{equation}
\begin{equation}\label{Lyapunov4}
	\bar{n}(t) = \frac{\beta(t)\left(\gamma(t)+t\dot{\gamma}(t)\right)}{\theta t\gamma(t)(t\beta(t)-\gamma(t)-t\dot{\gamma}(t))}.
\end{equation}

\begin{lemma}\label{lemma2}
	Assume that Assumption \eqref{eqH0} holds, $f$ is $L_1$-smooth and $g$ is $L_3$-smooth, while $\nabla^2f$ and $\nabla^2g$ are $L_2$-Lipschitz and $L_4$-Lipschitz, respectively. Assume that $ \alpha,\gamma,\beta,\epsilon:[t_0, +\infty) \to (0, +\infty) $ are $\mathcal{C}^1$, $ \theta>0 $ and $ t_0>0 $ are constants, $K$ is a continuous linear operator and $K^*$ is its adjoint operator. In addition, $\gamma(t)^2\theta^2t^2 \notin \{-\frac{1}{\sigma_1},-\frac{1}{\sigma_2},...,-\frac{1}{\sigma_j}\}$, where $\{\sigma_1,\sigma_2,...,\sigma_j\}$ denotes the set of non-zero eigenvalues of $KK^*$ and $K^*K$. If the conditions \eqref{tiaojian1} and \eqref{tiaojian2} are satisfied, then for any trajectory (global solution) $ (x(t),y(t)) $ of the dynamical system \eqref{eq15} and any primal-dual optimal solution $ (x^*,y^*) \in \Omega $ of the convex-concave bilinear saddle point problem \eqref{eq13}, we have
	\begin{equation}\label{eq75}
		\begin{aligned}
			&\quad \frac{2}{t}\bar{\mathcal{E}}(t) + \dot{\bar{\mathcal{E}}}(t) \\
			&\leq \frac{1}{2\theta t}\beta(t)\epsilon(t)\left(\left\|x^*\right\|^2+\left\|y^*\right\|^2\right)+\left(\frac{2}{t}\beta(t)+\dot{\beta}(t)-\frac{1}{\theta t}\beta(t)\right) \left(\mathcal{L}\left(x(t),y^*\right)-\mathcal{L}\left(x^*,y(t)\right)\right) \\
			&\quad +\frac{1}{2}\left(\left(\frac{2}{t}\beta(t)+\dot{\beta}(t)-\frac{1}{\theta t}\beta(t)\right)\epsilon(t)+\beta(t)\dot{\epsilon}(t)\right) \left(\left\|x(t)\right\|^2+\left\|y(t)\right\|^2\right) \\
			&\quad +\frac{\beta(t)}{2\theta t}\left(\frac{t\dot{\alpha}(t)+\alpha(t)}{t\beta(t)-\gamma(t)-t\dot{\gamma}(t)}-\epsilon(t)\right) \left(\left\|x(t)-x^*\right\|^2+\left\|y(t)-y^*\right\|^2\right) \\
			&\quad -\frac{\gamma(t)+t\dot{\gamma}(t)}{\gamma(t)} \left(\left\|\dot{x}(t)\right\|^2+\left\|\dot{y}(t)\right\|^2\right)+\gamma(t)\left(\dot{\gamma}(t)-\beta(t)+\frac{1}{t}\gamma(t)\right) \Delta(t),
		\end{aligned}
	\end{equation}
	where $\Delta(t):= \left\|\nabla_x\mathcal{L}_t\left(x(t),y(t)+\theta t\dot{y}(t)\right)\right\|^2 + \left\|\nabla_y\mathcal{L}_t\left(x(t)+\theta t\dot{x}(t),y(t)\right)\right\|^2$.
\end{lemma}

\begin{proof}
	Note that
	\begin{equation*}\label{eq65}
		\begin{aligned}
			&\quad \frac{1}{2}\left\| \bar{\eta}(t)\left(x(t)-x^*\right) + \dot{x}(t)+\gamma(t)\nabla_x\mathcal{L}_t\left(x(t),y(t)+\theta t\dot{y}(t)\right) \right\|^2 \\
			&\leq \frac{1}{2}\bar{\eta}(t)^2\left\|x(t)-x^*\right\|^2 + \frac{1}{2}\left\|\dot{x}(t)\right\|^2 + \frac{1}{2}\gamma(t)^2\left\|\nabla_x\mathcal{L}_t\left(x(t),y(t)+\theta t\dot{y}(t)\right)\right\|^2 \\
			&+\gamma(t)\langle\nabla_x\mathcal{L}_t\left(x(t),y(t)+\theta t\dot{y}(t)\right),\dot{x}(t)\rangle + \bar{\eta}(t)\langle x(t)-x^*,\dot{x}(t)\rangle \\
			&+ \bar{\eta}(t)\gamma(t)\langle \nabla_x\mathcal{L}_t\left(x(t),y(t)+\theta t\dot{y}(t)\right),x(t)-x^*\rangle,
		\end{aligned}
	\end{equation*}
	thus, for all $ t \geq t_0 $ we can obtain
	\begin{equation}\label{eq66}
		\begin{aligned}
			\bar{\mathcal{E}}_2(t) &\leq \frac{1}{2}\left(\bar{n}(t)+\bar{\eta}(t)^2\right)\left\|x(t)-x^*\right\|^2 + \frac{1}{2}\left\|\dot{x}(t)\right\|^2 +\frac{1}{2}\gamma(t)^2\left\|\nabla_x\mathcal{L}_t\left(x(t),y(t)+\theta t\dot{y}(t)\right)\right\|^2 \\
			&+\gamma(t)\langle\nabla_x\mathcal{L}_t\left(x(t),y(t)+\theta t\dot{y}(t)\right),\dot{x}(t)\rangle + \bar{\eta}(t)\langle x(t)-x^*,\dot{x}(t)\rangle \\
			&+ \bar{\eta}(t)\gamma(t)\langle \nabla_x\mathcal{L}_t\left(x(t),y(t)+\theta t\dot{y}(t)\right),x(t)-x^*\rangle.
		\end{aligned}
	\end{equation}
	Similarly, for all $ t \geq t_0 $, we have
	\begin{equation}\label{eq67}
		\begin{aligned}
			\bar{\mathcal{E}}_3(t) &\leq \frac{1}{2}\left(\bar{n}(t)+\bar{\eta}(t)^2\right)\left\|y(t)-y^*\right\|^2 + \frac{1}{2}\left\|\dot{y}(t)\right\|^2 + \frac{1}{2}\gamma(t)^2\left\|\nabla_y\mathcal{L}_t\left(x(t)+\theta t\dot{x}(t),y(t)\right)\right\|^2 \\
			&-\gamma(t)\langle\nabla_y\mathcal{L}_t\left(x(t)+\theta t\dot{x}(t),y(t)\right),\dot{y}(t)\rangle + \bar{\eta}(t)\langle y(t)-y^*,\dot{y}(t)\rangle \\
			&- \bar{\eta}(t)\gamma(t)\langle \nabla_y\mathcal{L}_t\left(x(t)+\theta t\dot{x}(t),y(t)\right),y(t)-y^*\rangle.
		\end{aligned}
	\end{equation}
	Combining the construction of the energy function $ \bar{\mathcal{E}}(t) $ in \eqref{eq62} with \eqref{eq66} and \eqref{eq67}, we know that
	\begin{equation*}\label{eq68}
		\begin{aligned}
			\bar{\mathcal{E}}(t) &\leq \beta(t)\left(\mathcal{L}\left(x(t),y^*\right)-\mathcal{L}\left(x^*,y(t)\right)\right) +\frac{1}{2}\beta(t)\epsilon(t)\left(\left\|x(t)\right\|^2+\left\|y(t)\right\|^2)\right) \\
			&+\frac{1}{2}\left(\bar{n}(t)+\bar{\eta}(t)^2\right) \left(\left\|x(t)-x^*\right\|^2+\left\|y(t)-y^*\right\|^2\right) + \frac{1}{2}\left(\left\|\dot{x}(t)\right\|^2+\left\|\dot{y}(t)\right\|^2\right) \\
			&+\frac{1}{2}\gamma(t)^2\Delta(t)+\bar{\eta}(t)\left(\langle x(t)-x^*,\dot{x}(t)\rangle+\langle y(t)-y^*,\dot{y}(t)\rangle\right) \\
			&+\gamma(t)\left( \langle \nabla_x\mathcal{L}_t\left(x(t),y(t)+\theta t\dot{y}(t)\right),\dot{x}(t) \rangle - \langle \nabla_y\mathcal{L}_t\left(x(t)+\theta t\dot{x}(t),y(t)\right),\dot{y}(t) \rangle \right) \\
			&+\bar{\eta}(t)\gamma(t)\left( \langle \nabla_x\mathcal{L}_t\left(x(t),y(t)+\theta t\dot{y}(t)\right),x(t)-x^* \rangle - \langle \nabla_y\mathcal{L}_t\left(x(t)+\theta t\dot{x}(t),y(t)\right),y(t)-y^* \rangle \right).
		\end{aligned}
	\end{equation*}
	Recall the equation \eqref{eq30} and we also have
	\begin{equation*}\label{eq69}
		\nabla_y\mathcal{L}_t\left(x(t)+\theta t\dot{x}(t),y(t)\right) = \nabla_y\mathcal{L}\left(x^*,y(t)\right)-\epsilon(t)y(t)+K\left(x(t)-x^*+\theta t\dot{x}(t)\right),
	\end{equation*}
	thus we can get that
	\begin{equation}\label{eq70}
		\begin{aligned}
			\bar{\mathcal{E}}(t) &\leq \beta(t)\left(\mathcal{L}\left(x(t),y^*\right)-\mathcal{L}\left(x^*,y(t)\right)\right) +\frac{1}{2}\beta(t)\epsilon(t)\left(\left\|x(t)\right\|^2+\left\|y(t)\right\|^2)\right) \\
			&+\frac{1}{2}\left(\bar{n}(t)+\bar{\eta}(t)^2\right) \left(\left\|x(t)-x^*\right\|^2+\left\|y(t)-y^*\right\|^2\right) + \frac{1}{2}\left(\left\|\dot{x}(t)\right\|^2+\left\|\dot{y}(t)\right\|^2\right) \\
			&+\frac{1}{2}\gamma(t)^2\Delta(t)+\bar{\eta}(t)(\langle x(t)-x^*,\dot{x}(t)\rangle+\langle y(t)-y^*,\dot{y}(t)\rangle) \\
			&+\gamma(t)\left( \langle \nabla_x\mathcal{L}\left(x(t),y^*\right),\dot{x}(t) \rangle - \langle \nabla_y\mathcal{L}\left(x^*,y(t)\right),\dot{y}(t) \rangle \right) + \gamma(t)\epsilon(t)\left(\langle x(t),\dot{x}(t)\rangle + \langle y(t),\dot{y}(t)\rangle \right) \\
			&+\left(\gamma(t)-\bar{\eta}(t)\gamma(t)\theta t\right) \left(\langle K^*\left(y(t)-y^*\right),\dot{x}(t) \rangle-\langle K\left(x(t)-x^*\right),\dot{y}(t) \rangle \right) \\
			&+\bar{\eta}(t)\gamma(t)\left( \langle \nabla_x\mathcal{L}\left(x(t),y^*\right)+\epsilon(t)x(t),x(t)-x^* \rangle - \langle \nabla_y\mathcal{L}\left(x^*,y(t)\right)-\epsilon(t)y(t),y(t)-y^* \rangle \right).			
		\end{aligned}
	\end{equation}
	Differentiate $\bar{\mathcal{E}}(t)$ with respect to $t$, we have
	\begin{equation}\label{eq71}
		\begin{aligned}
			\dot{\bar{\mathcal{E}}}(t) &= \dot{\beta}(t) \left(\mathcal{L}\left(x(t),y^*\right)-\mathcal{L}\left(x^*,y(t)\right)\right)+\frac{1}{2}\left(\dot{\beta}(t)\epsilon(t)+\beta(t)\dot{\epsilon}(t)\right)\left(\left\|x(t)\right\|^2+\left\|y(t)\right\|^2\right) \\
			&\quad +\left(\frac{1}{2}\dot{\bar{n}}(t)+\bar{\eta}(t)\dot{\bar{\eta}}(t)\right) (\left\|x(t)-x^*\right\|^2+\left\|y(t)-y^*\right\|^2)+\left(\bar{\eta}(t)-\alpha(t)\right)\left(\left\|\dot{x}(t)\right\|^2+\left\|\dot{y}(t)\right\|^2\right) \\
			&\quad +\gamma(t)\left(\dot{\gamma}(t)-\beta(t)\right) \Delta(t) + \left( \beta(t)+A_2(t) \right)\epsilon(t) \left(\langle x(t),\dot{x}(t)\rangle + \langle y(t),\dot{y}(t)\rangle \right) \\
			&\quad +\left( \beta(t)+A_2(t) \right) \left( \langle \nabla_x\mathcal{L}\left(x(t),y^*\right),\dot{x}(t)\rangle - \langle \nabla_y\mathcal{L}\left(x^*,y(t)\right),\dot{y}(t)\rangle \right) \\
			&\quad +\left(\bar{n}(t)+\dot{\bar{\eta}}(t)+\bar{\eta}(t)\left(\bar{\eta}(t)-\alpha(t)\right)\right) \left(\langle x(t)-x^*,\dot{x}(t)\rangle + \langle y(t)-y^*,\dot{y}(t)\rangle \right) \\
			&\quad +\left(\theta tA_1(t)-A_2(t)\right) \left(\langle K^*\dot{y}(t),x(t)-x^*\rangle - \langle K\dot{x}(t),y(t)-y^*\rangle \right) \\
			&\quad +A_1(t) \left(\langle \nabla_x\mathcal{L}\left(x(t),y^*\right)+\epsilon(t)x(t),x(t)-x^*\rangle-\langle \nabla_y\mathcal{L}\left(x^*,y(t)\right)-\epsilon(t)y(t),y(t)-y^*\rangle \right),
		\end{aligned}
	\end{equation}
	where 
	\begin{equation}\label{eq72}
		A_1(t) = \gamma(t)\dot{\bar{\eta}}(t)+\bar{\eta}(t)\left(\dot{\gamma}(t)-\beta(t)\right),
	\end{equation}
	\begin{equation}\label{eq73}
		A_2(t) = \dot{\gamma}(t)-\beta(t)+\gamma(t)\left(\bar{\eta}(t)-\alpha(t)\right).
	\end{equation}
	Then, by a similar argument in Lemma \ref{lemma1}, it follows from \eqref{eq70}, \eqref{eq71}, \eqref{eq72} and \eqref{eq73} that
	\begin{equation}\label{eq74}
		\begin{aligned}
			&\quad \frac{2}{t}\bar{\mathcal{E}}(t) + \dot{\bar{\mathcal{E}}}(t) \\
			&\leq -\frac{1}{2}\bar{A}_1(t)\epsilon(t)\left(\left\|x^*\right\|^2+\left\|y^*\right\|^2\right)+\left(\frac{2}{t}\beta(t)+\dot{\beta}(t)+\bar{A}_1(t)\right) \left(\mathcal{L}\left(x(t),y^*\right)-\mathcal{L}\left(x^*,y(t)\right)\right) \\
			&\quad +\frac{1}{2}\left(\left(\frac{2}{t}\beta(t)+\dot{\beta}(t)+\bar{A}_1(t)\right)\epsilon(t)+\beta(t)\dot{\epsilon}(t)\right) \left(\left\|x(t)\right\|^2+\left\|y(t)\right\|^2\right) \\
			&\quad +\frac{1}{2}\left(\frac{2}{t}\left(\bar{n}(t)+\bar{\eta}(t)^2\right)+\dot{\bar{n}}(t)+2\bar{\eta}(t)\dot{\bar{\eta}}(t)+\bar{A}_1(t)\epsilon(t)\right) \left(\left\|x(t)-x^*\right\|^2+\left\|y(t)-y^*\right\|^2\right) \\
			&\quad +\left(\bar{\eta}(t)-\alpha(t)+\frac{1}{t}\right) \left(\|\dot{x}(t)\|^2+\|\dot{y}(t)\|^2\right)+\gamma(t)\left(\dot{\gamma}(t)-\beta(t)+\frac{1}{t}\gamma(t)\right) \Delta(t) \\
			&\quad +\left(\beta(t)+\bar{A}_2(t)\right) \left( \langle \nabla_x\mathcal{L}\left(x(t),y^*\right),\dot{x}(t) \rangle - \langle \nabla_y\mathcal{L}\left(x^*,y(t)\right),\dot{y}(t) \rangle \right) \\
			&\quad +\left(\beta(t)+\bar{A}_2(t)\right)\epsilon(t) \left(\langle x(t),\dot{x}(t)\rangle + \langle y(t),\dot{y}(t)\rangle \right) \\
			&\quad +\left(\bar{n}(t)+\dot{\bar{\eta}}(t)+\bar{\eta}(t)\left(\bar{\eta}(t)-\alpha(t)\right)+\frac{2}{t}\bar{\eta}(t)\right) \left(\langle x(t)-x^*,\dot{x}(t) \rangle+\langle y(t)-y^*,\dot{y}(t) \rangle \right) \\
			&\quad +\left(\theta t\bar{A}_1(t)-\bar{A}_2(t)\right) \left(\langle K^T\dot{y}(t),x(t)-x^*\rangle-\langle K\dot{x}(t),y(t)-y^*\rangle \right),
		\end{aligned}
	\end{equation}
	where 
	\begin{equation*}
		\bar{A}_1(t) = A_1(t) + \frac{2}{t}\bar{\eta}(t)\gamma(t),
	\end{equation*}
	\begin{equation*}
		\bar{A}_2(t) = A_2(t) + \frac{2}{t}\gamma(t).
	\end{equation*}
	Furthermore, based on the specified formulation \eqref{Lyapunov3} for $\bar{\eta}(t)$, \eqref{Lyapunov4} for $\bar{n}(t)$, the conditions \eqref{tiaojian1} and \eqref{tiaojian2} given in the Lemma \ref{lemma1}, it can be inferred that the coefficients of the last four terms in \eqref{eq74} are zero. Moreover, we can obtain the expressions
	\begin{equation*}
		\begin{gathered}
			\bar{A}_1(t) = -\frac{\beta(t)}{\theta t}, \\
			\frac{2}{t}\left(\bar{n}(t)+\bar{\eta}(t)^2\right)+\dot{\bar{n}}(t)+2\bar{\eta}(t)\dot{\bar{\eta}}(t)+\bar{A}_1(t)\epsilon(t) = \frac{\beta(t)}{\theta t}\left(\frac{t\dot{\alpha}(t)+\alpha(t)}{t\beta(t)-\gamma(t)-t\dot{\gamma}(t)}-\epsilon(t)\right), \\
			\bar{\eta}(t)-\alpha(t)+\frac{1}{t} = -\frac{\gamma(t)+t\dot{\gamma}(t)}{\gamma(t)}.
		\end{gathered}
	\end{equation*}
	In summary, the expression \eqref{eq74} can be derived based on \eqref{eq75}.
	
	This completes the proof of Lemma \ref{lemma2}.
\end{proof}

\begin{theorem}\label{theorem2}
	Let all hypotheses in Lemma \ref{lemma2} hold. Assume that $ \epsilon(t) $ is a $\mathcal{C}^1$ and non-increasing function such that $ \int_{t_0}^{+\infty} \, \frac{\beta(t)\epsilon(t)}{t}\, \mathrm{d}t<+\infty $, and $\beta(t)$, $\alpha(t)$ are $\mathcal{C}^1$ and positive satisfying \eqref{tiaojian3} and \eqref{tiaojian4}. Then, for any trajectory $ (x(t), y(t)) $ of the dynamical system \eqref{eq15} and any $ (x^*,y^*) \in \Omega $, we can obtain that when $\liminf\limits_{t \to +\infty} \, \beta(t) \neq 0$, 
	\begin{equation*}
		\lim_{t \to +\infty} \; \mathcal{L}\left(x(t),y^*\right) - \mathcal{L}\left(x^*,y(t)\right) = 0.
	\end{equation*}
	Specially, when $\lim\limits_{t \to +\infty} \, \beta(t) = +\infty$, we have
	\begin{equation*}
		\begin{gathered}
			\mathcal{L}\left(x(t),y^*\right) - \mathcal{L}\left(x^*,y(t)\right) = o\left(\frac{1}{\beta(t)}\right), \\
			\left\| \nabla f(x(t)) - \nabla f(x^*) \right\| = o\left( \frac{1}{\sqrt{\beta(t)}} \right), \quad \left\| \nabla g(y(t)) - \nabla g(y^*) \right\| = o\left( \frac{1}{\sqrt{\beta(t)}} \right).
		\end{gathered}
	\end{equation*}
\end{theorem}

\begin{proof}
	Based on the assumptions on the coefficients $ \epsilon(t),\beta(t),\gamma(t) $ in Theorem \ref{theorem2}, it follows that
	\begin{flalign}
		&\ (a)\, \frac{2}{t}\beta(t)+\dot{\beta}(t)-\frac{1}{\theta t}\beta(t) \leq 0;&\nonumber\\
		&\ (b)\, \frac{\beta(t)}{2\theta t}\left(\frac{t\dot{\alpha}(t)+\alpha(t)}{t\beta(t)-\gamma(t)-t\dot{\gamma}(t)}-\epsilon(t)\right) \leq 0;&\nonumber\\
		&\ (c)\, \frac{1}{2}\left(\left(\frac{2}{t}\beta(t)+\dot{\beta}(t)-\frac{1}{\theta t}\beta(t)\right)\epsilon(t)+\beta(t)\dot{\epsilon}(t)\right) \leq 0;&\nonumber\\
		&\ (d)\, -\frac{\gamma(t)+t\dot{\gamma}(t)}{\gamma(t)} \leq 0;&\nonumber\\
		&\ (e)\, \gamma(t)\left(\dot{\gamma}(t)-\beta(t)+\frac{1}{t}\gamma(t)\right) < 0.&\nonumber
	\end{flalign}
	where $(a)$ and $(b)$ hold directly from \eqref{tiaojian3} and \eqref{tiaojian4}, the third inequality $(c)$ is satisfied since $ \epsilon(t) $ is a positive and non-increasing function, and the last two inequalities hold due to the positivity of $\gamma(t)$ and \eqref{tiaojian1}. This together with estimation \eqref{eq75} implies that for all $t \geq t_0 $,
	\begin{equation}\label{eq78}
		\frac{2}{t}\bar{\mathcal{E}}(t)+\dot{\bar{\mathcal{E}}}(t) \leq \frac{1}{2\theta t}\beta(t)\epsilon(t)\left(\left\|x^*\right\|^2+\left\|y^*\right\|^2\right).
	\end{equation}
	Multiplying both sides of \eqref{eq78} by $ t^2 $ and integrating the obtained results on $ [t_0, t] $ yield
	\begin{equation}\label{eq79}
		\bar{\mathcal{E}}(t) \leq \frac{t_0^2\bar{\mathcal{E}}(t_0)}{t^2}+\frac{\left\|x^*\right\|^2+\left\|y^*\right\|^2}{2\theta t^2}\int_{t_0}^{t} \, \tau\beta(\tau)\epsilon(\tau)\, \mathrm{d}\tau.
	\end{equation}
	This allows us to use Lemma \ref{limit} with $s=t_0$, $ \varphi(t)=t^2 $ and $ \zeta(t)=\frac{\beta(t)\epsilon(t)}{t} $ to \eqref{eq79} to get that
	\begin{equation}\label{eq80}
		\lim_{t \to +\infty} \; \frac{1}{t^2} \int_{t_0}^{t} \tau \beta(\tau) \epsilon(\tau)\, \mathrm{d}\tau = 0,
	\end{equation}
	where the condition $ \int_{t_0}^{+\infty} \, \frac{\beta(t)\epsilon(t)}{t}\, \mathrm{d}t<+\infty $ is used. Since $ \bar{\mathcal{E}}(t) \geq 0 $, it follows from \eqref{eq79} and \eqref{eq80} that $ \lim\limits_{t \to +\infty} \, \bar{\mathcal{E}}(t) = 0 $ which indicates that $ \bar{\mathcal{E}}(t) $ is a Lyapunov function and moreover the definition of $ \bar{\mathcal{E}}(t) $ yields that
	\begin{equation*}\label{eq81}
		\lim_{t \to +\infty} \; \beta(t)\left(\mathcal{L}\left(x(t),y^*\right)-\mathcal{L}\left(x^*,y(t)\right)\right) = 0,
	\end{equation*}
	which implies
	\begin{equation*}\label{eq82}
		\lim_{t \to +\infty} \; \mathcal{L}\left(x(t),y^*\right)-\mathcal{L}\left(x^*,y(t)\right) = 0,
	\end{equation*}
	when $ \liminf\limits_{t \to +\infty} \, \beta(t) \neq 0 $, and
	\begin{equation}\label{eq83}
		\mathcal{L}\left(x(t),y^*\right)-\mathcal{L}\left(x^*,y(t)\right) = o\left(\frac{1}{\beta(t)}\right),
	\end{equation}
	when $ \liminf\limits_{t \to +\infty} \, \beta(t)=+\infty $. Recall the estimation \eqref{new1} and multiply at its both sides the positive $\beta(t)$ to get that
	\begin{equation*}
		\frac{\beta(t)}{2L}\left\|\nabla f(x(t))-\nabla f(x^*)\right\|^2 \leq \beta(t)\left(\mathcal{L}\left(x(t),y^*\right)-\mathcal{L}\left(x^*,y(t)\right)\right)
	\end{equation*}
	and
	\begin{equation*}
		\frac{\beta(t)}{2L}\left\|\nabla g(y(t))-\nabla g(y^*)\right\|^2 \leq \beta(t)\left(\mathcal{L}\left(x(t),y^*\right)-\mathcal{L}\left(x^*,y(t)\right)\right).
	\end{equation*}
	According to \eqref{eq83}, this implies
	\begin{equation*}
		\left\|\nabla f(x(t))-\nabla f(x^*)\right\| = o\left(\frac{1}{\sqrt{\beta(t)}}\right), \quad \left\|\nabla g(y(t))-\nabla g(y^*)\right\| = o\left(\frac{1}{\sqrt{\beta(t)}}\right).
	\end{equation*}
	
	This completes the proof of Theorem \ref{theorem2}.
\end{proof}

\subsection{Particular Cases}\label{subsec5}
In this subsection, we analyze two specific cases by selecting particular values for the parameters $\theta$, $\beta(t)$, $\gamma(t)$, and $\alpha(t)$, and derive the corresponding conclusions. It should be noted that we fix $\beta(t) = t^\beta$ (where $\beta \geq 0$), and further classify and discuss the possible values of $\alpha(t) + t\dot{\alpha}(t)$ to present illustrative examples. \\
\textbf{Case 1:} $\bm{\alpha(t)+t\dot{\alpha}(t) \leq 0}$

We consider the choice $\alpha(t)=\frac{\alpha}{t}$, $\gamma(t)=\gamma t^{\beta+1}$, and $\beta(t)=t^\beta$. Under the assumptions \eqref{tiaojian1}, \eqref{tiaojian2} and \eqref{tiaojian3}, it holds that $\theta = \frac{1}{\left(\alpha-\beta-3\right)\left(1-\gamma\beta-2\gamma\right)} \leq \frac{1}{\beta+2}$. Building on Theorems \ref{theorem1} and \ref{theorem2}, we can derive the result presented below.

\begin{theorem}\label{teli1}
	In the dynamical system \eqref{eq15}, let
	\begin{equation*}
		\alpha(t)=\frac{\alpha}{t},\quad \gamma(t)=\gamma t^{\beta+1},\quad \beta(t)=t^\beta,\quad \theta = \frac{1}{\left(\alpha-\beta-3\right)\left(1-\gamma\beta-2\gamma\right)},
	\end{equation*}
	where $\alpha$, $\gamma$ and $\beta$ are constants with $0 < \gamma \leq \frac{1}{\beta+2}\left(1-\frac{\beta+2}{\alpha-\beta-3}\right)$, $\alpha>2\beta+5$ and $\beta \geq 0$. It should be noted that we set $t_0 \geq 1$. Suppose that $\epsilon:[t_0,+\infty) \to \mathbb{R}_+$ is both $\mathcal{C}^1$ and non-increasing and $\left(x(t),y(t)\right)$ is a global solution of the dynamical system \eqref{eq15}. Then, for any fixed $\left(x^*,y^*\right) \in \Omega$, the following conclusions hold:
	\\
	\noindent{\rm{(i)}}\;\; If $\int_{t_0}^{+\infty} \, t^{\beta + 1} \epsilon(t) \, \mathrm{d}t < +\infty$, then
	\vspace{-5pt}
	\begin{adjustwidth}{5mm}{}
		\begin{flalign*}
			&\begin{array}{l@{\quad}l}
				\textbf{(Boundedness of the Trajectory)} & \text{the trajectory }\left(x(t),y(t)\right)\text{ is bounded};\\[5pt]
				\textbf{(Pointwise Estimates)} & 
				\left\{
				\begin{array}{l}
					\mathcal{L}\left(x(t),y^*\right) - 	\mathcal{L}\left(x^*,y(t)\right) = \mathcal{O}\left( t^{-\beta-2} \right), \\[5pt]
					\left\|\dot{x}(t) + \gamma(t) \nabla_x \mathcal{L}_t \left(x(t), y(t) + \theta t \dot{y}(t)\right) \right\| = \mathcal{O}\left( \frac{1}{t} \right), \\[5pt]
					\left\|\dot{y}(t) - \gamma(t) \nabla_y \mathcal{L}_t \left(x(t) + \theta t \dot{x}(t), y(t)\right) \right\| = \mathcal{O}\left( \frac{1}{t} \right), \\[5pt]
					\left\| \nabla f(x(t)) - \nabla f(x^*) \right\| = \mathcal{O}\left( t^{-\frac{\beta+2}{2}} \right), \\[5pt]
					\left\| \nabla g(y(t)) - \nabla g(y^*) \right\|= \mathcal{O}\left( t^{-\frac{\beta+2}{2}} \right);
				\end{array}
				\right. \\[40pt]
				\textbf{(Integral Estimates)} & 
				\left\{
				\begin{array}{l}	
					\int_{t_0}^{+\infty} \, t \left( \| \dot{x}(t) \|^2 + \| \dot{y}(t) \|^2 \right) \, \mathrm{d}t < +\infty, \\[5pt]
					\int_{t_0}^{+\infty} \, t^{2\beta+3} \Delta(t) \, \mathrm{d}t < +\infty;
				\end{array}
				\right.
			\end{array}&
		\end{flalign*}
	\end{adjustwidth}
	\qquad If $\gamma < \frac{1}{\beta+2}\left(1-\frac{\beta+2}{\alpha-\beta-3}\right)$, $\forall t \geq t_0$, then 
	\begin{equation*}
		\begin{cases}
			\int_{t_0}^{+\infty} \, t^{\beta+1} \left(\mathcal{L}\left(x(t),y^*\right)-\mathcal{L}\left(x^*,y(t)\right)\right) \, \mathrm{d}t < +\infty, \\[5pt]
			\int_{t_0}^{+\infty} \, t^{\beta+1} \left\| \nabla f(x(t)) - \nabla f(x^*) \right\|^2 \, \mathrm{d}t < +\infty, \\[5pt]
			 \int_{t_0}^{+\infty} \, t^{\beta+1} \left\| \nabla g(y(t)) - \nabla g(y^*) \right\|^2 \, \mathrm{d}t < +\infty;
		\end{cases}
	\end{equation*}
	{\rm{(ii)}}\;\; If $\int_{t_0}^{+\infty} \, t^{\beta - 1} \epsilon(t) \, \mathrm{d}t < +\infty$, then
		\begin{itemize}
			\item[(a)] $\lim\limits_{t \to +\infty} \, \mathcal{L}\left(x(t),y^*\right) - \mathcal{L}\left(x^*,y(t)\right) = 0;$
			\item[(b)] $\text{Specially, when } \beta>0, \text{ we have }$ \begin{equation*}
				\begin{cases}
					\mathcal{L}\left(x(t),y^*\right) - \mathcal{L}\left(x^*,y(t)\right) = o\left(t^{-\beta}\right), \\
					\left\| \nabla f(x(t)) - \nabla f(x^*) \right\| = o\left(t^{-\frac{\beta}{2}}\right), \\
					\left\| \nabla g(y(t)) - \nabla g(y^*) \right\| = o\left(t^{-\frac{\beta}{2}}\right).
				\end{cases}
			\end{equation*}
		\end{itemize}
\end{theorem}

\begin{proof}
	By simple calculation, it is easy to check that
	\begin{equation*}
		\begin{gathered}
		\alpha(t)+t\dot{\alpha}(t) = 0, \\
		(2\theta - 1) t \beta(t) + \theta t^2 \dot{\beta}(t) = \left((2 + \beta)\theta - 1\right) t^{\beta + 1}, \\
		\gamma(t)+t\dot{\gamma}(t) = \gamma\left(\beta+2\right) t^{\beta+1},
		\end{gathered}
	\end{equation*}
	which together with $\theta = \frac{1}{\left(\alpha-\beta-3\right)\left(1-\gamma\beta-2\gamma\right)} \leq \frac{1}{\beta+2}$ implies the parameters satisfy the conditions \eqref{tiaojian1}, \eqref{tiaojian2}, \eqref{tiaojian3}, \eqref{tiaojian4} and all conditions in Theorem \ref{theorem2} when $t \geq t_0 >0$. In particular, if $0 < \gamma < \frac{1}{\beta+2}\left(1-\frac{\beta+2}{\alpha-\beta-3}\right)$, then in combination with $\beta(t)=t^\beta$, it is straightforward to verify that $(2\theta - 1) t \beta(t) + \theta t^2 \dot{\beta}(t)<0$, for $\forall t \geq t_0$. Therefore, Theorem \ref{teli1} is a direct consequence of Theorem \ref{theorem1} and Theorem \ref{theorem2}. 
	
	This completes the proof of Theorem \ref{teli1}.
\end{proof}

\begin{remark}
	Combining this example with the Remark \ref{remark_zhendang}, it can be concluded that $\int_{t_0}^{+\infty} \, t^{2\beta+3} \Delta(t) \, \mathrm{d}t < +\infty,$ $\left\|\dot{x}(t) + \gamma(t) \nabla_x \mathcal{L}_t \left(x(t), y(t) + \theta t \dot{y}(t)\right) \right\| = \mathcal{O}\left( \frac{1}{t} \right),$ and $\left\|\dot{y}(t) - \gamma(t) \nabla_y \mathcal{L}_t \left(x(t) + \theta t \dot{x}(t), y(t)\right) \right\| = \mathcal{O}\left( \frac{1}{t} \right)$ can directly reflect that the oscillation phenomenon has been effectively suppressed.
\end{remark}

\begin{flushleft}
	\textbf{Case 2:} $\bm{0 < \alpha(t)+t\dot{\alpha}(t) \leq \left( t\beta(t)-\gamma(t)-t\dot{\gamma}(t) \right)\epsilon(t)}$
	
	We consider the choice $\alpha(t)=\frac{3}{t}+\frac{(\beta+1)t^{\beta+2}-2}{t^{\beta+3}+2t}$, $\gamma(t) =\frac{\gamma}{t}+\frac{\gamma}{2}t^{\beta+1}$, $\beta(t) = t^\beta$. Moreover, under the assumption \eqref{tiaojian2}, it holds that $\theta = \frac{\gamma}{2-2\gamma-\gamma\beta}$. Thus we can get from Theorems \ref{theorem1} and \ref{theorem2} the following theorem.
\end{flushleft}

\begin{theorem}\label{teli2}
	In the dynamical system \eqref{eq15}, let
	\begin{equation*}
		\begin{gathered}
			\alpha(t)=\frac{2+2\gamma}{\gamma t}+\frac{(\beta+1)t^{\beta+2}-2}{t^{\beta+3}+2t},\quad \gamma(t) =\frac{\gamma}{t}+\frac{\gamma}{2}t^{\beta+1}, \\
			\beta(t) = t^\beta,\quad \theta = \frac{\gamma}{2-2\gamma-\gamma\beta},\quad \epsilon(t)=\frac{4(\beta+2)^2}{(2-2\gamma-\gamma\beta)t^{\beta+3}},
		\end{gathered}
	\end{equation*}
	where $\gamma$ and $\beta$ are constants with $0<\gamma\leq\frac{1}{\beta+2}$ and $\beta \geq 0$. It should be noted that we set $t_0 \geq 1$ and $\left(x(t),y(t)\right)$ is a global solution of the dynamical system \eqref{eq15}. Then, for any $\left(x^*,y^*\right) \in \Omega$, the following conclusions hold:
   	\begin{adjustwidth}{5mm}{}
		\begin{flalign*}
			\begin{array}{l@{\quad}l}
				\textbf{(Boundedness of the Trajectory)} & \text{the trajectory }\left(x(t),y(t)\right)\text{ is bounded};\\[5pt]
				\textbf{(Pointwise Estimates)} & 
				\left\{
				\begin{array}{l}
					\mathcal{L}\left(x(t),y^*\right) - 	\mathcal{L}\left(x^*,y(t)\right) = \mathcal{O}\left( t^{-2-\beta} \right), \\[5pt]
					\left\|\dot{x}(t) + \gamma(t) \nabla_x \mathcal{L}_t \left(x(t), y(t) + \theta t \dot{y}(t)\right) \right\| = \mathcal{O}\left( \frac{1}{t} \right), \\[5pt]
					\left\|\dot{y}(t) - \gamma(t) \nabla_y \mathcal{L}_t \left(x(t) + \theta t \dot{x}(t), y(t)\right) \right\| = \mathcal{O}\left( \frac{1}{t} \right), \\[5pt]
					\left\| \nabla f(x(t)) - \nabla f(x^*) \right\| = \mathcal{O}\left( t^{-\frac{\beta+2}{2}} \right), \\[5pt]
					\left\| \nabla g(y(t)) - \nabla g(y^*) \right\|= \mathcal{O}\left( t^{-\frac{\beta+2}{2}} \right);
				\end{array}
				\right. \\[40pt]
				\textbf{(Integral Estimates)} & 
				\left\{
				\begin{array}{l}	
					\int_{t_0}^{+\infty} \, \frac{t^{\beta+3}}{2+t^{\beta+2}} \left( \| \dot{x}(t) \|^2 + \| \dot{y}(t) \|^2 \right) \, \mathrm{d}t < +\infty, \\[5pt]
					\int_{t_0}^{+\infty} \, \left( 2t^{\beta+1}+t^{2\beta+3} \right) \Delta(t) \, \mathrm{d}t < +\infty, \\[5pt]
					\int_{t_0}^{+\infty} \, t^{\beta+1} \left(\mathcal{L}\left(x(t),y^*\right)-\mathcal{L}\left(x^*,y(t)\right)\right) \, \mathrm{d}t < +\infty, \\[5pt]
					\int_{t_0}^{+\infty} \, t^{\beta+1} \left\| \nabla f(x(t)) - \nabla f(x^*) \right\|^2 \, \mathrm{d}t < +\infty, \\[5pt]
					\int_{t_0}^{+\infty} \, t^{\beta+1} \left\| \nabla g(y(t)) - \nabla g(y^*) \right\|^2 \, \mathrm{d}t < +\infty;
				\end{array}
				\right. \\[40pt]
				\textbf{(Limit and Rate Estimates)} &
				\lim\limits_{t \to +\infty} \, \mathcal{L}\left(x(t),y^*\right) - \mathcal{L}\left(x^*,y(t)\right) = 0;\\
				\ \text{Specially, when } \beta>0, \ \text{we can get } \\[5pt]
				&
				\left\{
				\begin{array}{l}
					\mathcal{L}\left(x(t),y^*\right) - \mathcal{L}\left(x^*,y(t)\right) = o\left(t^{-\beta}\right), \\[5pt]
					\left\| \nabla f(x(t)) - \nabla f(x^*) \right\| = o\left(t^{-\frac{\beta}{2}}\right), \\[5pt]
					\left\| \nabla g(y(t)) - \nabla g(y^*) \right\| = o\left(t^{-\frac{\beta}{2}}\right).
				\end{array}
				\right.
			\end{array}
		\end{flalign*}
	\end{adjustwidth}
\end{theorem}

\begin{proof}
	With specific formulations of parameters in Theorem \ref{teli2}, we can obtain by simple calculating that
	\begin{equation*}
		\begin{gathered}
			\alpha(t)+t\dot{\alpha}(t) = \frac{2(\beta+2)^2t^{\beta+3}}{\left(t^{\beta+3}+2t\right)^2} \leq \left(t\beta(t)-\gamma(t)-t\dot{\gamma}(t)\right)\epsilon(t) = \frac{2(\beta+2)^2}{t^2}, \\
			(2\theta - 1)t\beta(t) + \theta t^2 \dot{\beta}(t) = \left((2+\beta)\theta-1\right)t^{\beta+1} < 0,
		\end{gathered}
	\end{equation*}
	which together with $0<\gamma\leq\frac{1}{\beta+2}$ implies the parameters satisfy the conditions \eqref{tiaojian1}, \eqref{tiaojian2}, \eqref{tiaojian3}, \eqref{tiaojian4} and all conditions in Theorem \ref{theorem2} when $t \geq t_0 >0$. Therefore, according to Theorem \ref{theorem1} and Theorem \ref{theorem2}, we can obtain the conclusions of Theorem \ref{teli2}. 
	
	This completes the proof of Theorem \ref{teli2}.
\end{proof}

\begin{remark}
	In view of this example and the corresponding Remark \ref{remark_zhendang}, $\int_{t_0}^{+\infty} \, \left(2t^{\beta+1}+t^{2\beta+3}\right) \Delta(t) \, \mathrm{d}t < +\infty,$ $\left\|\dot{x}(t) + \gamma(t) \nabla_x \mathcal{L}_t \left(x(t), y(t) + \theta t \dot{y}(t)\right) \right\| = \mathcal{O}\left( \frac{1}{t} \right),$ and $\left\|\dot{y}(t) - \gamma(t) \nabla_y \mathcal{L}_t \left(x(t) + \theta t \dot{x}(t), y(t)\right) \right\| = \mathcal{O}\left( \frac{1}{t} \right)$ clearly demonstrate the effective suppression of oscillations.
\end{remark}

\section{Strong Convergence of the Trajectory}\label{sec4}
This section starts with the formulation of a convex optimization problem and a strongly convex optimization problem. By virtue of Tikhonov regularization techniques, we establish the convergence result for the trajectory of the dynamical system \eqref{eq15} as the Tikhonov regularization parameter $\epsilon(t)$ tends to zero at a proper decay rate. Furthermore, we prove that the trajectory $\left(x(t),y(t)\right)$ generated by \eqref{eq15} admits strong convergence to the minimum-norm solution of problem \eqref{eq13}.

\subsection{Classical Facts}\label{subsec3}
For and fixed $z^* := (x^*, y^*) \in \Omega$ and any $z:=(x,y) \in \mathbb{R}^n \times \mathbb{R}^m$, we consider the following convex optimization problem
\begin{equation} \label{eq85}
	\min_{z=(x,y) \in \mathbb{R}^n \times \mathbb{R}^m} \; \Psi_{z^*}(z) := \mathcal{L}\left(x, y^*\right) - \mathcal{L}\left(x^*, y\right),
\end{equation}
where we denote the minimal norm of the solution set of the problem \eqref{eq85} by $\hat{z}^*$. And according to the definition of $\hat{z}^*$, it has a remarkable geometrical property 
\begin{equation*}
	\hat{z}^* = \text{Proj}_\Omega 0. 
\end{equation*}
According to Definition \ref{andian}, we know that the optimal value of \eqref{eq85} is $0$. The optimal points of problem \eqref{eq85} satisfy \eqref{eq20}, which indicates that the solution set of \eqref{eq85} coincides with the saddle point set $\Omega$ of the convex-concave bilinear saddle point problem \eqref{eq13}.

For each $\sigma > 0$, the strongly convex minimization problem established on the basis of the problem \eqref{eq85} is as follows
\begin{equation}\label{eq86}
	\min_{z \in \mathbb{R}^n \times \mathbb{R}^m} \; \Psi_{z^*}^\sigma(z) := \Psi_{z^*}(z) + \frac{\sigma}{2} \|z\|^2.
\end{equation}
We denote by $z_\sigma$ the unique solution of the strongly convex minimization problem \eqref{eq86}. Owing to the classical properties of the Tikhonov regularization\cite{ref47,ref22}, for $\forall \sigma>0$, the Tikhonov approximation curve $\sigma \to z_\sigma$ satisfies
\begin{equation*}\label{eq88}
	\nabla \Psi_{z^*}^{\sigma}(z_{\sigma}) = \nabla \Psi_{z^*}(z_{\sigma}) + \sigma z_{\sigma} = 0,
\end{equation*} 
and
\begin{equation} \label{eq89}
	\lim_{\sigma \to 0} \, \left\| z_\sigma - \hat{z}^* \right\| = 0, \quad \left\| z_\sigma \right\| \leq \left\| \hat{z}^* \right\|.
\end{equation}

Below, we state a lemma, whose result plays a crucial role in the subsequent proof of strong convergence.
\begin{lemma}\label{lemma3}
	Suppose that $\sigma : [t_0, +\infty) \to [0, +\infty)$ with $\lim\limits_{t \to +\infty} \, \sigma(t) = 0$. Let $\hat{z}^* := (\hat{x}^*, \hat{y}^*)$ denote the minimum-norm solution, and let $z(t) := (x(t), y(t))$ be a global solution to the dynamical system \eqref{eq15}. Then,
	\begin{equation*}\label{eq90}
		\frac{\sigma(t)}{2} \left( \left\| z(t) - z_{\sigma(t)} \right\|^2 + \left\| z_{\sigma(t)} \right\|^2 - \left\| \hat{z}^* \right\|^2 \right) \leq \Psi_{\hat{z}^*}^{\sigma(t)} \left( z(t) \right)- \Psi_{\hat{z}^*}^{\sigma(t)} \left( \hat{z}^* \right).
	\end{equation*}
\end{lemma}

\begin{proof}
	The proof follows a line of reasoning analogous to that of \cite[Lemma4.1]{ref36}, thus we omit the detailed derivation here.
\end{proof}

\subsection{Strong Convergence}\label{subsec4}
\begin{theorem}\label{theorem3}
	Supposed that $ f:\mathbb{R}^n \to \mathbb{R} $ and $ g:\mathbb{R}^m \to \mathbb{R} $ be twice differentiable convex functions, $f$ is $L_1$-smooth and $g$ is $L_3$-smooth, while $\nabla^2f$ and $\nabla^2g$ are $L_2$-Lipschitz and $L_4$-Lipschitz, respectively. Set $ \epsilon:[t_0, +\infty) \to (0, +\infty) $ is $\mathcal{C}^1$ and non-increasing, $\alpha,\gamma,\beta:[t_0,+\infty) \in (0,+\infty)$ be $\mathcal{C}^1$ functions satisfying  conditions \eqref{tiaojian1}, \eqref{tiaojian2}, \eqref{tiaojian3}, \eqref{tiaojian4}, $ \theta$ and $ t_0$ are positive constants, $K$ is a continuous linear operator and $K^*$ is its adjoint operator. In addition, $\gamma(t)^2\theta^2t^2 \notin \{-\frac{1}{\sigma_1},-\frac{1}{\sigma_2},...,-\frac{1}{\sigma_j}\}$, where $\{\sigma_1,\sigma_2,...,\sigma_j\}$ denotes the set of non-zero eigenvalues of $KK^*$ and $K^*K$. Assume that $\int_{t_0}^{+\infty} \, \frac{\beta(t)\epsilon(t)}{t} \,\mathrm{d}t < +\infty$, $\liminf\limits_{t \to +\infty} \, \beta(t) \neq 0 $ and $\lim\limits_{t \to +\infty} \, t^2 \beta(t) \epsilon(t) = +\infty$. Let $(x(t), y(t))$ be a global solution of the dynamical system \eqref{eq15}. Then, for the unique element of minimal norm $\left(\hat{x}^*, \hat{y}^*\right) = \text{Proj}_\Omega 0$, we have
	\begin{equation*}\label{eq93}
		\liminf_{t \to +\infty} \; \left\| \left(x(t), y(t)\right) - \left(\hat{x}^*, \hat{y}^*\right) \right\| = 0.
	\end{equation*}
	
	Further, if there exists a large enough $T$ such that the trajectory $(x(t), y(t))_{t \geq T}$ stays in either the open ball $B\left(0, \left\| \left(\hat{x}^*, \hat{y}^*\right) \right\|\right)$ or its complements, then,
	\begin{equation*}\label{eq94}
		\lim_{t \to +\infty} \; \left\| (x(t), y(t)) - \left(\hat{x}^*, \hat{y}^*\right) \right\| = 0.
	\end{equation*}
\end{theorem}

\begin{proof}
	Based on the sign of $\left\| (x(t), y(t)) \right\| - \left\| \left(\hat{x}^*, \hat{y}^*\right) \right\|$, we analyze this theorem in the following three cases.
	
	\textbf{Case I:} There exists a sufficiently large constant $T$ such that the trajectory $(x(t), y(t))$ remains in the complement of $B\left(0, \left\| \left(\hat{x}^*, \hat{y}^*\right) \right\|\right)$. By the definition of the norm on the Cartesian product, this is equivalent to
	\begin{equation*} \label{eq96}
		\left\| x(t) \right\|^2 + \left\| y(t) \right\|^2 \geq \left\| \hat{x}^* \right\|^2 + \left\| \hat{y}^* \right\|^2, \quad \forall t \geq T.
	\end{equation*}
	
	For any fixed point $\left(\hat{x}^*, \hat{y}^*\right) \in \Omega$, we define the function $\widehat{\mathcal{E}} : [t_0, +\infty) \to \mathbb{R}$ by
	\begin{equation} \label{eq97}
		\widehat{\mathcal{E}}(t) :=  \bar{\mathcal{E}}(t) - \frac{1}{2}\epsilon(t)\beta(t) \left(\left\| \hat{x}^* \right\|^2 + \left\| \hat{y}^* \right\|^2\right).
	\end{equation}
	
	By adopting similar arguments as in the proofs of Lemma \ref{lemma2} and Theorem \ref{theorem2}, we deduce that
	\begin{equation*}\label{eq98}
		\begin{aligned}
			\frac{2}{t} \widehat{\mathcal{E}}(t) + \dot{\widehat{\mathcal{E}}}(t) &\leq \left( \dot{\beta}(t)+\frac{2}{t}\beta(t)-\frac{\beta(t)}{\theta t} \right) \left( \mathcal{L}\left(x(t), \hat{y}^*\right) - \mathcal{L}\left(\hat{x}^*, y(t)\right) \right) \\
			&+ l(t) \left( \left\| x(t) \right\|^2 + \left\| y(t) \right\|^2 - \left\| \hat{x}^* \right\|^2 - \left\| \hat{y}^* \right\|^2 \right) \\
			&+ \frac{\beta(t)}{2\theta t}\left(\frac{t\dot{\alpha}(t)+\alpha(t)}{t\beta(t)-\gamma(t)-t\dot{\gamma}(t)}-\epsilon(t)\right) \left( \left\| x(t) - \hat{x}^* \right\|^2 + \left\| y(t) - \hat{y}^* \right\|^2 \right) \\
			&-\frac{\gamma(t)+t\dot{\gamma}(t)}{\gamma(t)} \left(\left\|\dot{x}(t)\right\|^2+\left\|\dot{y}(t)\right\|^2\right) + \gamma(t)\left(\dot{\gamma}(t)-\beta(t)+\frac{1}{t}\gamma(t)\right) \Delta(t),
		\end{aligned}
	\end{equation*}
	where $\Delta(t)$ is defined in the same manner as before, and $l(t)$ is given by
	\begin{equation*}
		l(t):=\frac{1}{2}\left(\left(\frac{2}{t}\beta(t)+\dot{\beta}(t)-\frac{1}{\theta t}\beta(t)\right)\epsilon(t)+\beta(t)\dot{\epsilon}(t)\right).
	\end{equation*}
	
	Combining the assumptions on the parameters $\epsilon(t),\beta(t),\gamma(t),\alpha(t)$, the conditions \eqref{tiaojian3} and \eqref{tiaojian4}, and the nonnegativity condition $\mathcal{L}(x(t), \hat{y}^*) - \mathcal{L}(\hat{x}^*, y(t)) \geq 0$, we obtain
	\begin{equation*}\label{eq99}
		\frac{2}{t} \widehat{\mathcal{E}}(t) + \dot{\widehat{\mathcal{E}}}(t) \leq 0, \quad \forall t \geq T.
	\end{equation*}
	
	Multiplying both sides of the above inequality by $t^2$ and integrating the resulting inequality from $T$ to $t$ yields
	\begin{equation*}\label{eq101}
		\widehat{\mathcal{E}}(t) \leq \frac{T^2 \widehat{\mathcal{E}}(T)}{t^2}, \quad \forall t \geq T.
	\end{equation*}
	
	Combining this estimate with \eqref{eq86} and \eqref{eq97}, we further conclude that
	\begin{equation*}\label{eq103}
		\beta(t) \left( \Psi_{z^*}^{\epsilon(t)}(z(t)) - \Psi_{z^*}^{\epsilon(t)}\left(\hat{z}^*\right) \right) \leq \frac{T^2 \widehat{\mathcal{E}}(T)}{t^2}.
	\end{equation*}
	
	Thus, in view of Lemma \ref{lemma3}, one has
	\begin{equation*}\label{eq104}
		\left\| z(t) - z_{\epsilon(t)} \right\|^2 + \left\| z_{\epsilon(t)} \right\|^2 - \left\| \hat{z}^* \right\|^2 \leq \frac{2 T^2 \widehat{\mathcal{E}}(T)}{t^2 \beta(t) \epsilon(t)}.
	\end{equation*}
	
	Together with \eqref{eq89} and the limit condition $\lim\limits_{t \to +\infty} t^2 \beta(t) \epsilon(t) = +\infty$, this implies that
	\begin{equation*}\label{eq105}
		\lim_{t \to +\infty} \; \left\| z(t) - \hat{z}^* \right\| = 0,
	\end{equation*}
	i.e.,
	\begin{equation*}\label{eq106}
		\lim_{t \to +\infty} \; \left\| (x(t), y(t)) - (\hat{x}^*, \hat{y}^*) \right\| = 0.
	\end{equation*}
	
	\textbf{Case II:}
	There exists a sufficiently large constant $T$ such that $ (x(t), y(t)) \in B\left(0,\left\| \left(\hat{x}^*,\hat{y}^* \right) \right\| \right) $ for all $t \geq T$, i.e.,
	\begin{equation} \label{eq108}
		\left\| z(t) \right\| < \left\| \hat{z}^* \right\|, \quad \forall t \geq T.
	\end{equation}
	Let $\bar{z}$ be a weak sequential cluster point of $z(t)$. Then there exists a sequence $\{ t_n \}_{n \in \mathbb{N}}$ with $t_n \to +\infty$ such that $z(t_n) \rightharpoonup \bar{z}$ as $n \to +\infty$. Since $\Psi$ is weakly lower semicontinuous, one has
	\begin{equation}\label{eq109}
		\Psi_{\hat{z}^*}(\bar{z}) \leq \liminf_{n \to +\infty} \, \Psi_{\hat{z}^*}(z(t_n)).
	\end{equation}
	Meanwhile, we have
	\begin{equation*}\label{eq110}
		\lim_{t \to +\infty} \, \Psi_{\hat{z}^*}(z(t)) = \lim_{t \to +\infty} \, \left( \mathcal{L}(x(t), \hat{y}^*) - \mathcal{L}(\hat{x}^*, y(t)) \right) = 0.
	\end{equation*}
	Combining this with \eqref{eq109}, we deduce
	\begin{equation*}\label{eq111}
		\Psi_{\hat{z}^*}(\bar{z}) = 0,
	\end{equation*}
	which implies $\bar{z} \in \Omega$. In addition, the norm functional is weakly lower semicontinuous, so it follows that 
	\begin{equation}\label{eq112}
		\left\| \bar{z} \right\| \leq \liminf_{n \to +\infty} \, \left\| z(t_n) \right\| \leq \left\| \hat{z}^* \right\|,
	\end{equation}
	where the second inequality follows from \eqref{eq108}. Recall that $\hat{z}^*$ is the unique minimal-norm element in $\Omega$ and $\bar{z} \in \Omega$, which immediately yields $\bar{z} = \hat{z}^*$.
	Hence, $z(t) \rightharpoonup \hat{z}^*$, namely, the trajectory $z(t)$ converges weakly to $\hat{z}^*$. Substituting this into \eqref{eq112}, we obtain
	\begin{equation*}\label{eq114}
		\left\| \hat{z}^* \right\| \leq \liminf_{t \to +\infty} \, \left\| z(t) \right\| \leq \limsup_{t \to +\infty} \, \left\| z(t) \right\| \leq \left\| \hat{z}^* \right\|.
	\end{equation*}
	Consequently,
	\begin{equation*}\label{eq115}
		\lim_{t \to +\infty} \, \left\| z(t) \right\| = \left\| \hat{z}^* \right\|.
	\end{equation*}
	Together with the weak convergence $z(t) \rightharpoonup \hat{z}^*$, this further implies the strong convergence of $z(t)$, i.e.,
	\begin{equation*}\label{eq116}
		\lim_{t \to +\infty} \, \left\| z(t) - \hat{z}^* \right\| = 0.
	\end{equation*}
	Equivalently,
	\begin{equation*}\label{eq117}
		\lim_{t \to +\infty} \, \left\| (x(t), y(t)) - \left(\hat{x}^*, \hat{y}^*\right) \right\| = 0.
	\end{equation*}
	
	\textbf{Case III:}
	For any $T \geq t_0 > 0$, there exist $\tau_1 \geq T$ and $\tau_2 \geq T$ such that
	\[
	\left\| (x(\tau_1), y(\tau_1)) \right\| \geq \left\| (\hat{x}^*, \hat{y}^*) \right\|,\qquad
	\left\| (x(\tau_2), y(\tau_2)) \right\| < \left\| (\hat{x}^*, \hat{y}^*) \right\|.
	\]
	By the continuity of $(x(t), y(t))$, one can extract a sequence $\{ t_n \}_{n \in \mathbb{N}} \subseteq [t_0, +\infty)$ satisfying $t_n \to +\infty$ as $n \to +\infty$ and $\|z(t_n)\| = \|\hat{z}^*\|$. This ensures the existence of weak sequential cluster points of $\{z(t_n)\}$. Without loss of generality, we suppose that the entire sequence $\{z(t_n)\}$ converges weakly.
	Using similar arguments as in Case II, we obtain
	\begin{equation*}\label{eq119}
		\lim_{n \to +\infty} \left\| z(t_n) \right\| = \left\| \hat{z}^* \right\|, \quad z(t_n) \rightharpoonup \hat{z}^* \; \text{as } \; n \to +\infty.
	\end{equation*}
	Therefore,
	\begin{equation*}\label{eq120}
		\lim_{n \to +\infty} \; \left\| z(t_n) - \hat{z}^* \right\| = 0.
	\end{equation*}
	It then follows that
	\begin{equation*}\label{eq121}
		\liminf_{t \to +\infty} \; \left\| z(t) - \hat{z}^* \right\| = 0,
	\end{equation*}
	i.e.,
	\begin{equation*}\label{eq122}
		\liminf_{t \to +\infty} \; \left\| (x(t), y(t)) - (\hat{x}^*, \hat{y}^*) \right\| = 0.
	\end{equation*}
	
	This completes the proof of Theorem \ref{theorem3}. 
\end{proof}

\section{Numercial Experiments}\label{sec5}
In this section, to verify the theoretical results associated with the proposed dynamical system \eqref{eq15}, numerical experiments are carried out on a min-max optimization problem and an $\bm{\ell_2}$-regularized problem, respectively. Furthermore, numerical solutions are obtained using the ode45 adaptive Runge-Kutta method in MATLAB R2023b. All codes are executed on a personal computer (equipped with a 1.60GHz Intel Core i5-10210U processor and 16GB of memory).

\begin{example}\label{example1}
	\textbf{A Min-Max Optimization Problem}

Let $x:=(x_1,x_2)^T \in \mathbb{R}^2$ and $y:=(y_1,y_2)^T \in \mathbb{R}^2$.
\begin{equation}\label{e1}
	\min_{x \in \mathbb{R}^2}\max_{y \in \mathbb{R}^n} \; \begin{pmatrix}
		x_1 \; x_2
	\end{pmatrix} \begin{pmatrix}
		m^2 & mn\\
		mn  & n^2
	\end{pmatrix}\begin{pmatrix}
		x_1\\
		x_2
	\end{pmatrix}+\langle \begin{pmatrix}
		mj & nj\\
		mk & nk
	\end{pmatrix}\begin{pmatrix}
		x_1\\
		x_2
	\end{pmatrix},\begin{pmatrix}
		y_1\\
		y_2
	\end{pmatrix} \rangle-\begin{pmatrix}
		y_1 \; y_2
	\end{pmatrix} \begin{pmatrix}
		j^2 & jk\\
		jk  & k^2
	\end{pmatrix}\begin{pmatrix}
		y_1\\
		y_2
	\end{pmatrix},
\end{equation}
where $m,n,j,k \in \mathbb{R}\setminus\{0\}$. According to the definition of the convex-concave bilinear saddle point problem \eqref{eq13}, combined with \eqref{e1}, it can be known that $f(x)=\left(mx_1+nx_2\right)^2$, $g(y)=\left(jy_1+ky_2\right)^2$ and $K=\left( \begin{smallmatrix} m_j & n_j \\ m_k & n_k \end{smallmatrix} \right)$. By simple calculation, it is very easy to check that the solution set of this optimization problem is $\left\{ (x,y) \in \mathbb{R}^2 \times \mathbb{R}^2 \mid mx_1 + nx_2 = 0 \text{ and } jy_1 + ky_2 = 0 \right\}$, $\left( \hat{x}^*,\hat{y}^* \right)=\left(0,0,0,0\right)^T$ is the minimal norm solution of the convex-concave bilinear saddle point problem \eqref{eq13} and the optimal objective function value is $0$. 

We set the parameters of the dynamical system \eqref{eq15}: $\alpha=19,\beta=1,\alpha(t)=\frac{\alpha}{t},\beta(t)=t^\beta,\theta=\frac{1}{15\left(1-3\gamma\right)},\epsilon(t)=\frac{2}{t^2}$, and $\gamma(t)=\gamma t^{\beta+1}$ (with $0<\gamma\leq\frac{4}{15}$). Additionally, we fix the initial time $t_0 = 1$ and the initial point as given below:
\begin{equation*}
	x(t_0)=\begin{pmatrix}
		1\\
		1.5
	\end{pmatrix},\quad \dot{x}(t_0)=\begin{pmatrix}
		1\\
		1
	\end{pmatrix},\quad y(t_0)=\begin{pmatrix}
		1\\
		1.5
	\end{pmatrix},\quad \dot{y}(t_0)=\begin{pmatrix}
		1\\
		1
	\end{pmatrix}.
\end{equation*}
All three numerical experiments for this example are conducted under the following two cases:
\begin{itemize}
	\item[$\bullet$] $m=1, n=6, j=4, k=10$.
	\item[$\bullet$] $m=1, n=10, j=15, k=1$.
\end{itemize}

In the first experiment, for any fixed $\left(x^*, y^*\right) \in \Omega$, We investigate the effect of Hessian-driven damping on the convergence of the primal-dual gap, namely the convergence rate and properties of $\mathcal{L}(x(t),y^*)-\mathcal{L}(x^*,y(t))$ under different values of $\gamma \in \{\frac{2}{15},\frac{3}{20},\frac{1}{6}\}$. The dynamical system \eqref{eq15} is solved over the time interval $[1,30]$ using the ode45 solver in MATLAB. Fig. \ref{Fig1} presents results for two scenarios: one incorporating both Hessian-driven damping and the Tikhonov regularization, and another only with Tikhonov regularization. The second scenario corresponds to the Tikhonov regularized dynamical system in \cite{hessian2}. It should be noted that when the Hessian-driven term is excluded, $\theta$ is calculated using $\theta=\frac{1}{15-45\gamma}$ to ensure consistency in non-zero $\theta$ and $\gamma$ values, thereby enabling more accurate comparisons.

\begin{figure}[ht]
	\centering
	\subfloat[$m=1, n=6, j=4, k=10$]{%
		\includegraphics[width=0.475\textwidth]{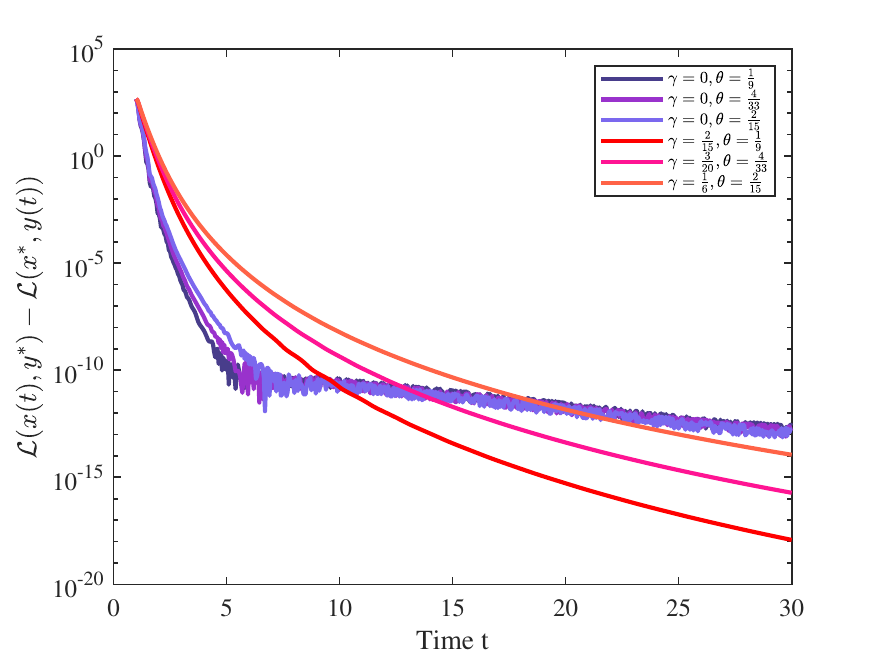} }
	\quad
	\subfloat[$m=1, n=10, j=15, k=1$]{%
		\includegraphics[width=0.475\textwidth]{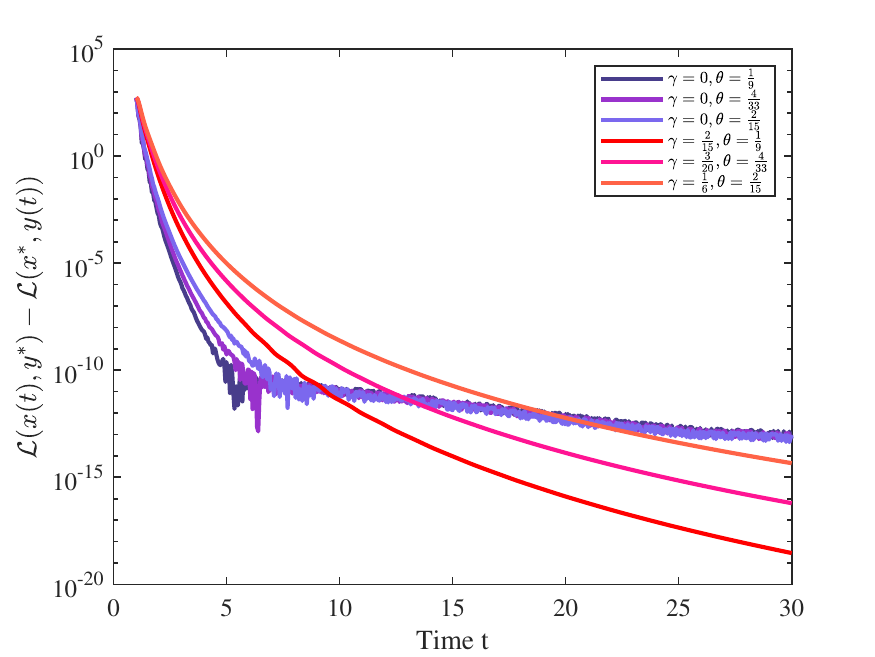} }
	\caption{Convergence analysis of the dynamical system \eqref{eq15} under different parameters $\gamma$}
	\label{Fig1} 
\end{figure}

From Fig. \ref{Fig1}, we observe that in each case, the numerical results align with the theoretical claims: when the dynamical system employs Hessian-driven damping control, it effectively suppresses oscillations, rendering the convergence process of $\mathcal{L}(x(t),y^*)-\mathcal{L}(x^*,y(t))$ smoother. Moreover, the gap function achieves higher convergence accuracy.

In the second experiment, we fix the parameter $\gamma=\frac{2}{15}$, $\theta=\frac{1}{9}$, and set the time interval $[1,80]$. Then examine the strong convergence of the trajectory to the minimal norm solution $(\hat{x}^*,\hat{y}^*)=(0,0,0,0)^T$, along with the impact of Hessian-driven damping and on trajectory behaviors.

\begin{figure}[ht]
	\centering
	\subfloat[Both Hessian-driven damping and Tikhonov regularization]{%
		\includegraphics[width=0.3\textwidth]{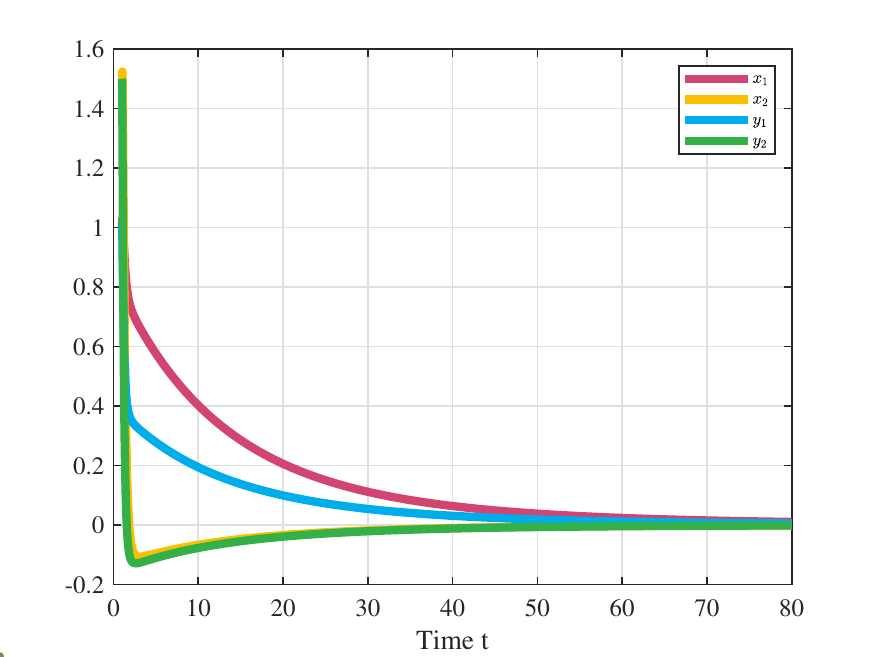} }
	\quad
	\subfloat[Only Hessian-driven damping]{%
		\includegraphics[width=0.3\textwidth]{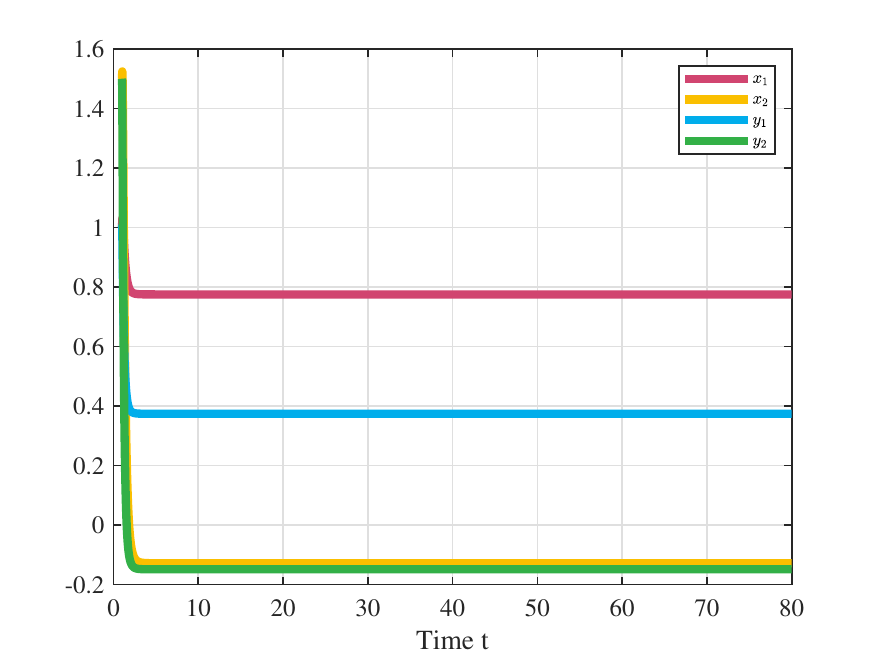} }
	\quad
	\subfloat[Neither Hessian-driven damping nor Tikhonov regularization]{%
		\includegraphics[width=0.3\textwidth]{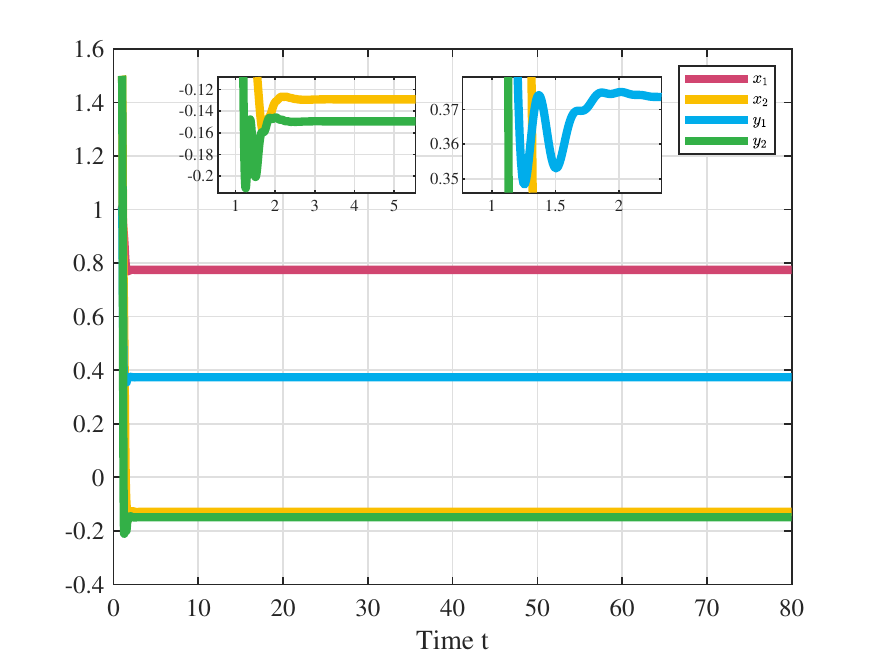} }
	\caption{Convergence of trajectories with different terms when $m=1,n=6,j=4,k=10$}
	\label{Fig2} 
\end{figure}

\begin{figure}[ht]
	\centering
	\subfloat[Both Hessian-driven damping and Tikhonov regularization]{%
		\includegraphics[width=0.3\textwidth]{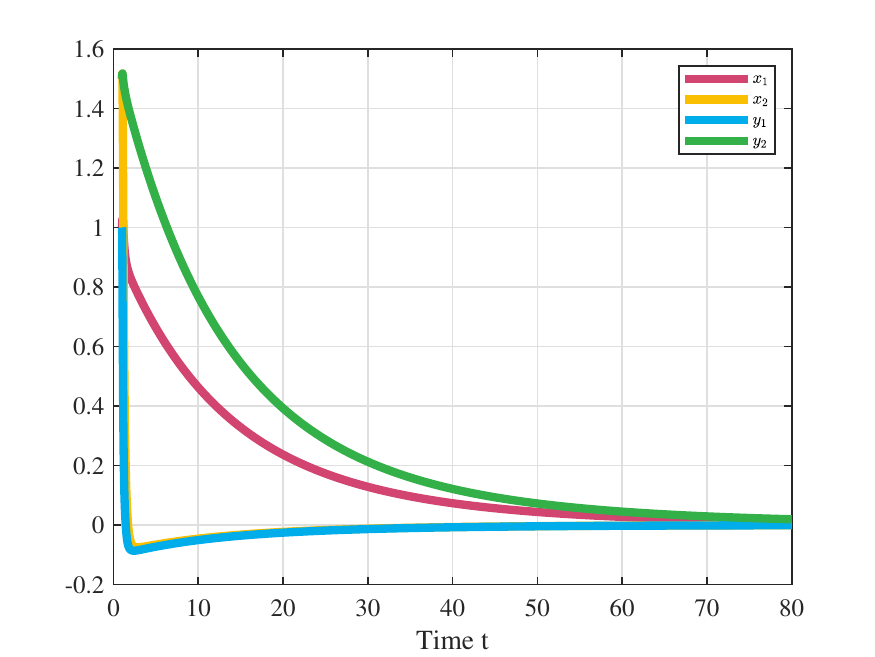} }
	\quad
	\subfloat[Only Hessian-driven damping]{%
		\includegraphics[width=0.3\textwidth]{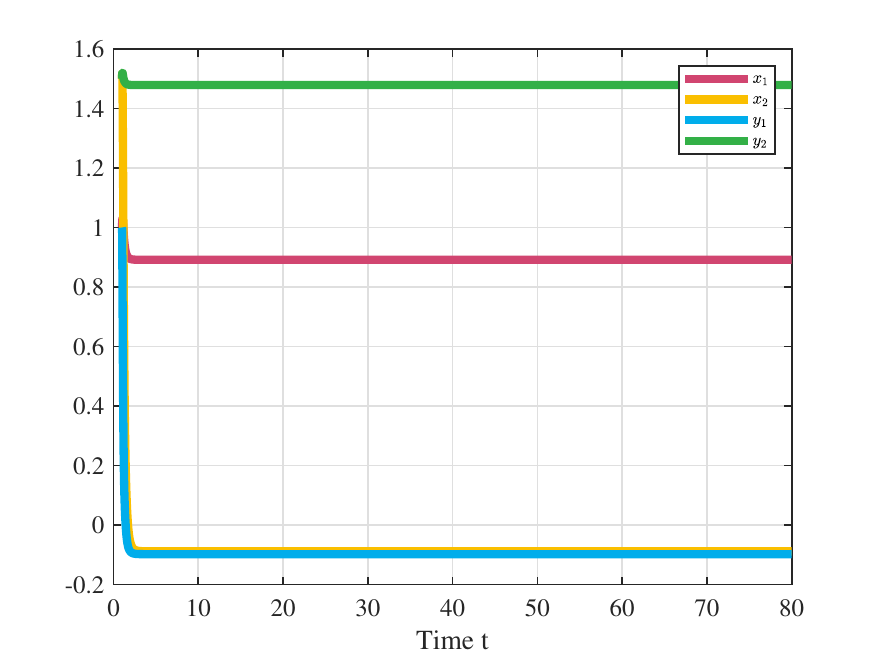} }
	\quad
	\subfloat[Neither Hessian-driven damping nor Tikhonov regularization]{%
		\includegraphics[width=0.3\textwidth]{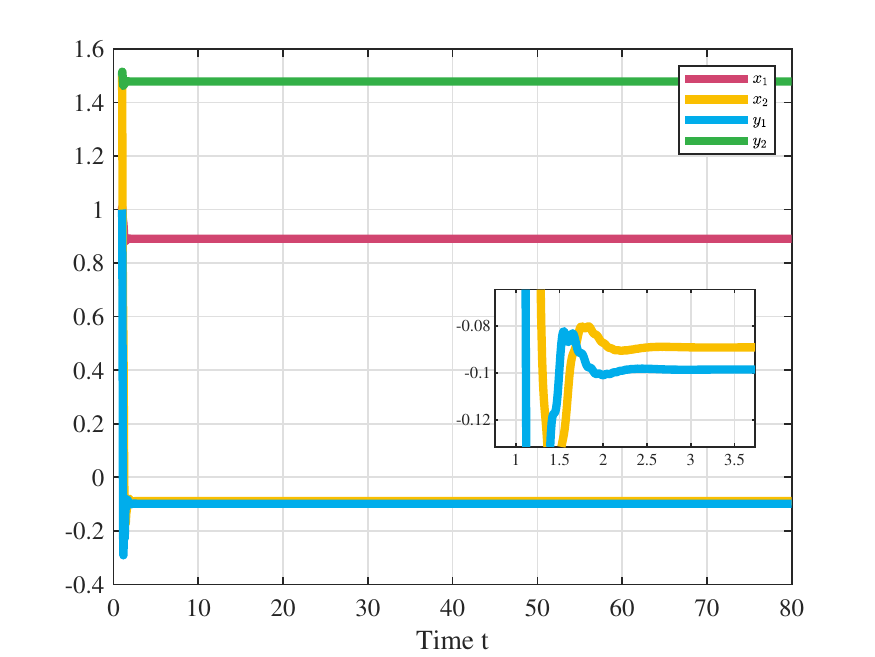} }
	\caption{Convergence of trajectories with different terms when $m=1,n=10,j=15,k=1$}
	\label{Fig6} 
\end{figure}

As illustrated in Fig. \ref{Fig2} and Fig. \ref{Fig6}, we can get that
\begin{enumerate}
	\item[(i)] We observe that the trajectory $\left(x(t),y(t)\right)$ generated by the Tikhonov-regularized dynamical system \eqref{eq15} involving Hessian-driven damping converges effectively to the minimal norm solution $\left(\hat{x}^*,\hat{y}^*\right)$. In contrast, the trajectory of the dynamical system \eqref{eq15} fails to converge to $\left(\hat{x}^*,\hat{y}^*\right)$ once the Tikhonov regularization terms $\epsilon(t)x(t)$ and $\epsilon(t)y(t)$ are removed.
	\item[(ii)] We observe that when the dynamical system \eqref{eq15} incorporates Hessian-driven damping control, its trajectory $(x(t), y(t))$ converges more smoothly than that in the case without Hessian-driven damping.
\end{enumerate}

In the third experiment, we compare the gap convergence rates among the dynamical system \eqref{eq15}, the scheme in \cite{ref42}, and the scheme in \cite{ref45} over the time interval $[1,30]$. For convenience, the dynamical system proposed in this paper is denoted as Han, that in \cite{ref42} as $(MPDD)$, and that in \cite{ref45} as Sun.

For the dynamical system \eqref{eq15}, all conditions from the first experiment are retained, while for the dynamical system in \cite{ref45}, the following parameters are adopted: $\alpha=2.5$, $q=0.42$, $s=0.05$, $p=0.268$, $c=10$, and $\gamma \in \{0.8, 1.0, 1.2\}$. For the dynamical system in \cite{ref42}, we enforce the removal of Hessian-driven damping and the Tikhonov regularization from dynamical system \eqref{eq15}, while keeping all other parameter settings unchanged.

\begin{figure}[ht]
	\centering
	\subfloat[$m=1, n=6, j=4, k=10$]{%
		\includegraphics[width=0.475\textwidth]{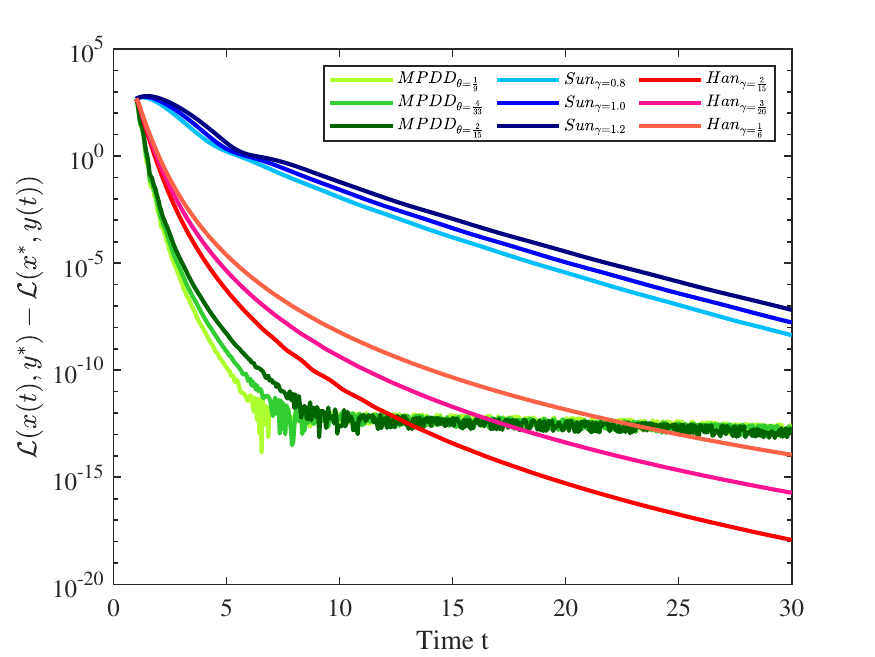} }
	\quad
	\subfloat[$m=1, n=10, j=15, k=1$]{%
		\includegraphics[width=0.475\textwidth]{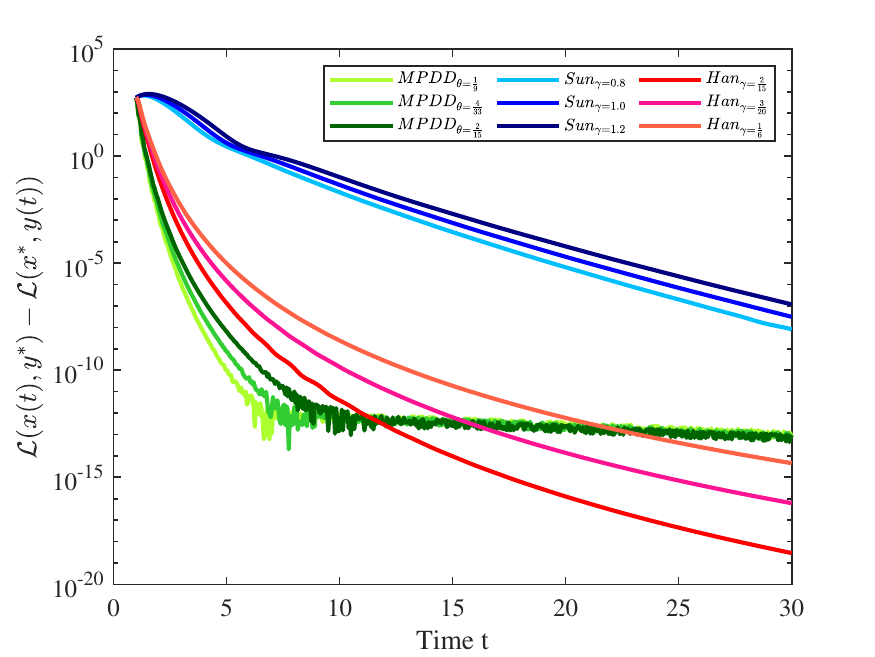} }
	\caption{Comparison of Convergence Rates of the Gap Among Dynamical System \eqref{eq15}, $(MPDD)$ and Sun}
	\label{Fig3} 
\end{figure}

In Fig. \ref{Fig3}, both the dynamical system \eqref{eq15} and Sun can effectively mitigate oscillations during gap convergence, but the convergence performance of the dynamical system \eqref{eq15} is significantly superior to that of Sun. Moreover, the convergence process of the dynamical system \eqref{eq15} is significantly smoother than that of $(MPDD)$. Furthermore, in the latter half of the time interval, the dynamical system \eqref{eq15} achieves a higher-precision convergence performance.
\end{example}

\begin{remark}\label{shuzhishiyan1}
	It can be observed from Fig.~\ref{Fig1} that the convergence process of the dynamical system \eqref{eq15} is smoother than that of the dynamical system in \cite{hessian2}. In the second experiment, the scenario with $\gamma = 0$ and $\epsilon(t) \equiv 0$ corresponds to the dynamical system satisfying $\delta(t) = \theta t$ in \cite{ref42}. Further comparison with Fig.~\ref{Fig2} reveals that the trajectory generated by dynamical system \eqref{eq15} converges strongly to the minimal norm solution, whereas its counterpart fails to do so.
\end{remark}

\begin{example}
	\textbf{An $\bm{\ell_2}$-regularized Problem}

Consider the following $\ell_2$-regularized problem:
\begin{equation}\label{l2}
	\min_{x \in \mathbb{R}^n} \; \Phi(x) = \frac{1}{2} \left\| Kx - b \right\|^2 + \omega \left\| x \right\|^2,
\end{equation}
where $K \in \mathbb{R}^{m \times n}$ and $b \in \mathbb{R}^m$. We can rewrite its saddle point formulation as:
\begin{equation*}
	\min_{x \in \mathbb{R}^n} \max_{y \in \mathbb{R}^m} \; \omega \left\| x \right\|^2 + \langle Kx-b, y \rangle -  \frac{1}{2} \left\| y \right\|^2 .
\end{equation*}

In the subsequent numerical experiment, we focus on investigating how Hessian-driven damping and Tikhonov regularization term affects the convergence of the objective function value for problem \eqref{l2}.

We set $\omega = 1$, with $K$ and $b$ sampled from a standard Gaussian distribution. We specify the time interval as $[1, 85]$, and the following numerical experiment is conducted under four different dimensional settings:
\begin{itemize}
	\item[$\bullet$] $n = 10$, $m = 3$.
	\item[$\bullet$] $n = 50$, $m = 20$.
	\item[$\bullet$] $n = 100$, $m = 50$.
	\item[$\bullet$] $n = 200$, $m = 100$.
\end{itemize}

In the following numerical experiments, we investigate the evolution of the objective error $\Phi(x(t))-\Phi(x^*)$ along the generated trajectories for different dynamical systems under various dimensional settings, as illustrated in Figure~\ref{Fig5}. For the dynamical system \eqref{eq15} in this paper, the following parameter settings are maintained: $\alpha = 19$, $\gamma=\frac{1}{6}$, $\beta = 1$, $\alpha(t) = \frac{\alpha}{t}$, $\beta(t) = t^\beta$, $\gamma(t)=\gamma t^{\beta+1}$, $\theta=\frac{2}{15}$ and $\epsilon(t)=\frac{1}{t^8}$. For the dynamical system Sun in \cite{ref45}, the parameters are set as follows: $\alpha = 3$, $\gamma=0.2$, $c=5$, $s=0.4$, $p=2.3$, $q \in \{0.6,0.7,0.8\}$. To be consistent with the dynamical systems considered in \cite{ref42} and \cite{hessian2}, we perform two separate modifications to equation \eqref{eq15} while keeping all other parameters fixed: simultaneously removing both the Hessian-driven damping term and the Tikhonov regularization term, and removing only the Hessian-driven damping term. Retaining the abbreviations of the first three dynamical systems, the dynamical system in \cite{hessian2} is designated as Han. Owing to the absence of Hessian-driven damping, we adopt the notation Han$_{\gamma=0}$ for this dynamical system in the plots of the subsequent numerical experiments.

\begin{figure}[ht]
	\centering
	\subfloat[$n = 10$, $m = 3$]{%
		\includegraphics[width=0.475\textwidth]{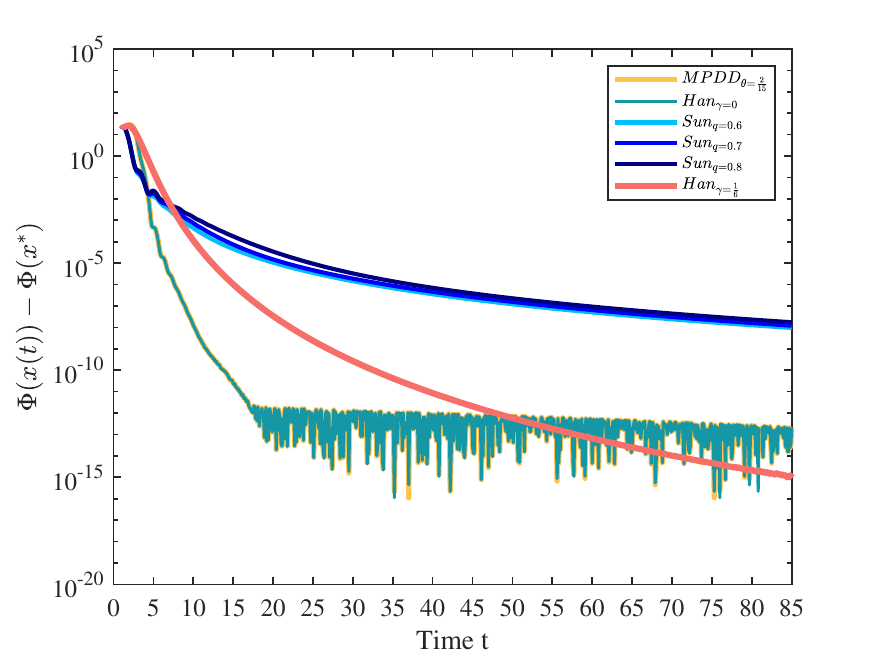} }
	\quad
	\subfloat[$n = 50$, $m = 20$]{%
		\includegraphics[width=0.475\textwidth]{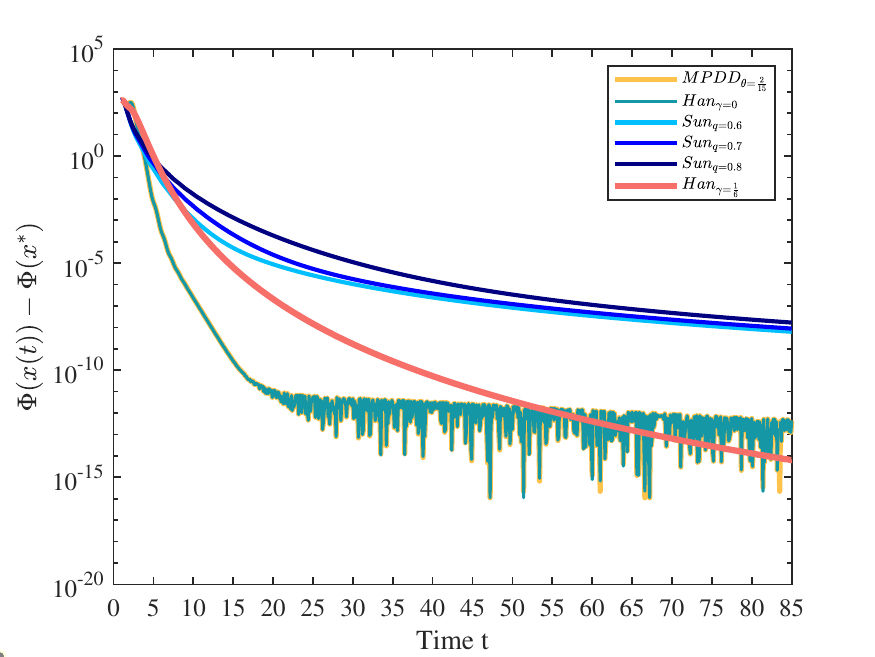} }
	\quad
	\subfloat[$n = 100$, $m = 50$]{%
		\includegraphics[width=0.475\textwidth]{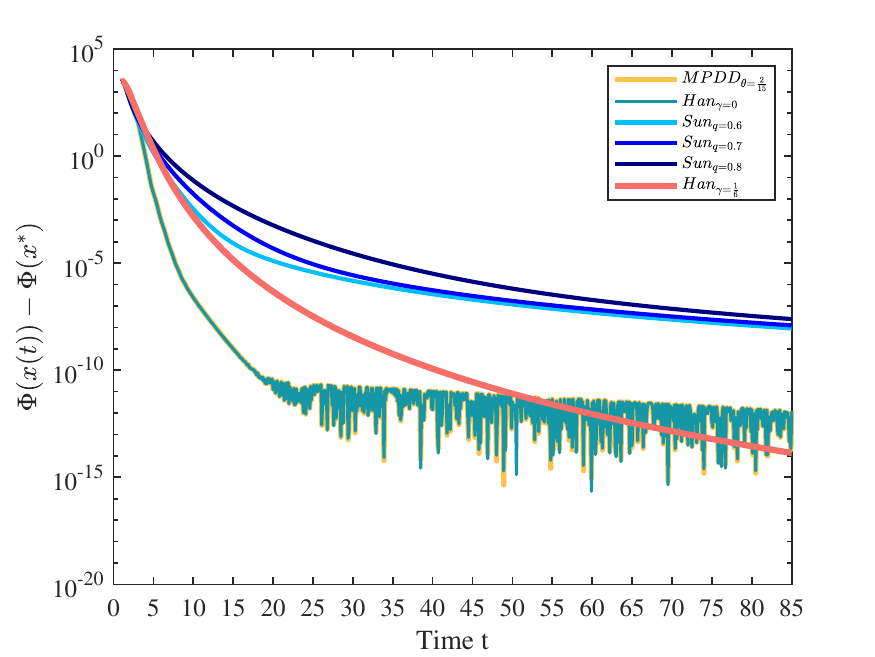} }
	\quad
	\subfloat[$n = 200$, $m = 100$]{%
		\includegraphics[width=0.475\textwidth]{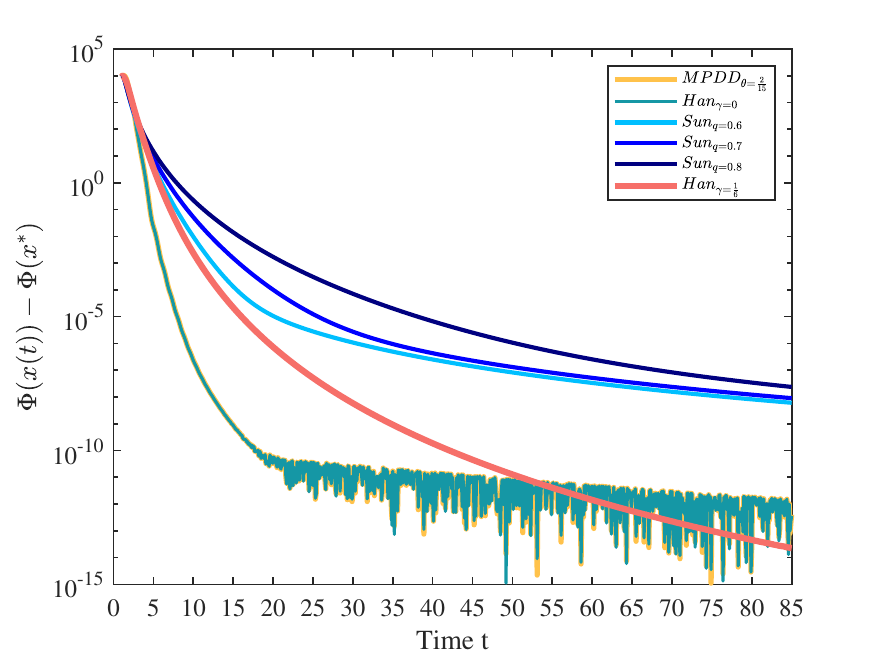} }
	\caption{Comparison of convergence under different dimension settings}
	\label{Fig5} 
\end{figure}

As shown in Fig.~\ref{Fig5}, under different dimensional settings, the dynamical system \eqref{eq15} outperforms the Sun dynamical system in \cite{ref45} in terms of convergence performance. Moreover, compared with $(MPBB)$ and the dynamical system in \cite{hessian2}, it can significantly mitigate oscillations during the trajectory convergence process while maintaining fast convergence performance.
\end{example}

\section{Conclusion}\label{sec13}

This paper explores a general second-order primal-dual dynamical system \eqref{eq15}, which is specifically constructed for the convex-concave bilinear saddle point problem \eqref{eq13}. Through the design of proper Lyapunov functions, we obtain the following asymptotic convergence properties:
\begin{itemize}
	\item[(i)] When the Tikhonov regularization parameter $\epsilon(t)$ fulfills the integral condition $\int_{t_0}^{+\infty} \, t\beta(t)\epsilon(t) \, \mathrm{d}t < +\infty$, we prove that the primal-dual gap corresponding to the trajectory generated by the dynamical system \eqref{eq15} converges at a rate of $\mathcal{O}\left(\frac{1}{t^2\beta(t)}\right)$.
	\item[(ii)] When the Tikhonov regularization parameter $\epsilon(t)$ satisfies the integral constraint $\int_{t_0}^{+\infty} \, \frac{\beta(t)\epsilon(t)}{t} \, \mathrm{d}t < +\infty$, we show that the primal-dual gap converges at a rate of $o\left(\frac{1}{\beta(t)}\right)$. Moreover, the trajectory induced by dynamical system \eqref{eq15} achieves strong convergence to the minimum-norm solution of the convex-concave bilinear saddle point problem \eqref{eq13}.
	\item[(iii)] In \eqref{eq24}, the Hessian-driven damping related term $t\gamma(t)\left(\gamma(t)+t\dot{\gamma}(t)-t\beta(t)\right)\Delta(t) < 0$ in the derivative $\dot{\mathcal{E}}(t)$ of the constructed Lyapunov function serves as a dissipation term. The derived formula \eqref{yizhizhendang1} strictly theoretically verifies the feasibility of suppressing oscillations via Hessian-driven damping.
\end{itemize}
The proposed dynamical system \eqref{eq15} integrates multiple time-varying parameters including slow viscous damping, extrapolation, time scaling, Hessian-driven damping and Tikhonov regularization coefficients. Numerical experiments confirm that Hessian-driven damping effectively alleviates oscillations during the iterative process, while Tikhonov regularization guarantees the strong convergence of trajectories to the minimum-norm solution.

\section*{Declaration of competing interest}
The authors declare that they have no known competing financial interests or personal relationships that could have appeared to influence the work reported in this paper.

\section*{Data availability}
No data was used for the research described in the article.

\section*{Acknowledgments}
This research was supported by the National Natural Science Foundation of China
(62371094).

\begin{appendices}	

\section{Existence and Uniqueness of Trajectories}\label{Appendix A}
In this appendix, we will give the proof of the existence and uniqueness of a global solution of the dynamical system \eqref{eq15} with the initial condition $\left(x(t_0),y(t_0)\right) = \left(x_0,y_0\right)$, $\left(\dot{x}(t_0),\dot{y}(t_0)\right) = \left(u_0,v_0\right)$. 
\begin{definition}\label{definition1A1}
	For $t>t_0$, $z=(x,y): [t_0,+\infty) \times [t_0,+\infty) \to \mathbb{R}^n \times \mathbb{R}^m$ is a strong global solution of dynamical system \eqref{eq15} if it satisfies the following properties:\\
	\noindent (i) $x : [t_0, +\infty) \to \mathbb{R}^n$  and  $y : [t_0, +\infty) \to \mathbb{R}^m$  are locally absolutely continuous;\\
	\noindent (ii) For $\forall t \in [t_0,+\infty)$, $(x(t),y(t))$ satisfies the dynamical system \eqref{eq15};\\
	\noindent (iii) $x(t_0)=x_0\in\mathbb{R}^n$, $\dot{x}(t_0)=u_0\in\mathbb{R}^m$, $y(t_0)=y_0\in\mathbb{R}^n$, $\dot{y}(t_0)=v_0\in\mathbb{R}^m$.
\end{definition}
By analogy with the argument in \cite [Theorem 5]{ref12}, we can easily establish the existence and uniqueness of the global solution to the dynamical system \eqref {eq15}.
\begin{theorem}\label{theorem1A1}
	Suppose that both $f$ and $g$ are twice continuously differentiableSuppose Assumption \eqref{eqH0} is fulfilled, $f$ is $L_1$ -smooth on $\mathbb{R}^n$ with $L_1>0$ and $g$ is $L_3$-smooth on $\mathbb{R}^m$ with $L_3>0$, $\nabla^2 f$ and $\nabla^2 g$ are Lipschitz continuous of order $L_2$ and $L_4$ respectively. Let $\alpha$, $\gamma$, $\beta$, $\epsilon: [t_0,+\infty) \to (0,+\infty)$ be continuous functions, both $ \theta $ and $ t_0 $ are constants greater than $0$, $K$ is a continuous linear operator and $K^*$ is its adjoint operator. In addition, $\gamma(t)^2\theta^2t^2 \notin \{-\frac{1}{\sigma_1},-\frac{1}{\sigma_2},...,-\frac{1}{\sigma_j}\}$, where $\{\sigma_1,\sigma_2,...,\sigma_j\}$ denotes the set of non-zero eigenvalues of $KK^*$ and $K^*K$. Then, for any given initial condition $(x(t_0), y(t_0), \dot{x}(t_0), \dot{y}(t_0)) \in \mathbb{R}^n \times \mathbb{R}^m \times \mathbb{R}^n \times \mathbb{R}^m$, the dynamical system \eqref{eq15} has a unique strong global solution.
\end{theorem}

\begin{proof}
	Denote $\Lambda(t):=(\lambda_1(t),\lambda_2(t),\lambda_3(t),\lambda_4(t))=(x(t),y(t),\dot{x}(t),\dot{y}(t))$, $\Lambda_0:=(x_0,y_0,u_0,v_0)$. To simplify the notation, we denote $\Lambda(t)$ as $\Lambda:=(\lambda_1,\lambda_2,\lambda_3,\lambda_4)$. Then, the dynamical system \eqref{eq15} can be equivalently written as the Cauchy problem
	\begin{equation}\label{oula}
		\begin{cases}
			\frac{d\Lambda}{dt}=\mathcal{G}(t,\Lambda), \\
			\Lambda(t_0)=\Lambda_0,
		\end{cases}
	\end{equation}
	where
	\begin{equation*}
		\mathcal{G}(t,\Lambda) := \begin{pmatrix} 
			\lambda_3 \\
			\lambda_4 \\
			E_1(t)^{-1}\left(M(t)-\gamma(t)\theta tK^*N(t)\right) \\
			E_2(t)^{-1}\left(N(t)+\gamma(t)\theta tKM(t)\right)
		\end{pmatrix}.
	\end{equation*}
	We denote by $E_1$ the $m \times m$ and by $E_2$ the $n \times n$ identity matrix, respectively. In the above equation, $E_1(t)=E_1+\gamma(t)^2\theta^2t^2K^*K$, $E_2(t)=E_2+\gamma(t)^2\theta^2t^2KK^*$. Based on the constraint on $\gamma(t)^2\theta^2t^2$ given in the theorem \ref{theorem1A1}, it follows that $E_1(t)$ and $E_2(t)$ are invertible. And $M(t)$ and $N(t)$ are defined as follows
	\begin{equation*}
		\begin{aligned}
			M(t) =& -\alpha(t)\lambda_3 - \beta(t)\left(\nabla f(\lambda_1)+\epsilon(t)\lambda_1+K^*(\lambda_2+\theta t\lambda_4)\right) \\
			&-\gamma(t)\left(\nabla^2f(\lambda_1)\lambda_3+\dot{\epsilon}(t)\lambda_1+\epsilon(t)\lambda_3+K^*(\lambda_4+\theta\lambda_4)\right),
		\end{aligned}
	\end{equation*}
	and
	\begin{equation*}
		\begin{aligned}
			N(t) =& -\alpha(t)\lambda_4 + \beta(t)\left(-\nabla g(\lambda_2)-\epsilon(t)\lambda_2+K(\lambda_1+\theta t\lambda_3)\right) \\
			&+\gamma(t)\left(-\nabla^2g(\lambda_2)\lambda_4-\dot{\epsilon}(t)\lambda_2-\epsilon(t)\lambda_4+K(\lambda_3+\theta\lambda_3)\right).
		\end{aligned}
	\end{equation*}
	Since $K$ is linear and continuous, it holds that for any $\Lambda, \widetilde{\Lambda} \in \mathbb{R}^n \times \mathbb{R}^m \times \mathbb{R}^n \times\mathbb{R}^m$,
	\begin{equation}\label{oula1}
		\left\|\mathcal{G}(t,\Lambda)-\mathcal{G}(t,\widetilde{\Lambda})\right\| \leq \left\|\lambda_3-\widetilde{\lambda}_3\right\|+\left\|\lambda_4-\widetilde{\lambda}_4\right\|+I_1(t)\left\|M(t)-\widetilde{M}(t)\right\|+I_2(t)\left\|N(t)-\widetilde{N}(t)\right\|,
	\end{equation}
	where $I_1(t)=\left\|E_1(t)^{-1}\right\|+\gamma(t)\theta t\left\|E_2(t)^{-1}K\right\|$, $I_2(t)=\left\|E_2(t)^{-1}\right\|+\gamma(t)\theta t\left\|E_1(t)^{-1}K^*\right\|$, and the norm is deined as $\|(x, y)\| = \sqrt{\|x\|^2 + \|y\|^2}$. Now, we first consider scaling the norm $\left\|M(t)-\widetilde{M}(t)\right\|$, which is done as follows:
	\begin{equation}\label{oula2}
		\begin{aligned}
			\left\|M(t)-\widetilde{M}(t)\right\| &\leq  \left(\beta(t)\epsilon(t)+\gamma(t)\dot{\epsilon}(t)\right)\left\|\lambda_1-\widetilde{\lambda}_1\right\|+\beta(t)\left\|K^*\right\|\left\|\lambda_2-\widetilde{\lambda}_2\right\| \\ 
			&\quad +\left(\alpha(t)+\gamma(t)\epsilon(t)\right)\left\|\lambda_3-\widetilde{\lambda}_3\right\| + \left(\theta t\beta(t)+(1+\theta)\gamma(t)\right)\left\|K^*\right\|\left\|\lambda_4-\widetilde{\lambda}_4\right\| \\
			&\quad +\beta(t)\left\|\nabla f(\lambda_1)-\nabla f(\widetilde{\lambda}_1)\right\|+\gamma(t)\left\|\nabla^2f(\lambda_1)\lambda_3-\nabla^2f(\widetilde{\lambda}_1)\widetilde{\lambda}_3\right\| \\
			&\leq \left(\beta(t)L_1+\beta(t)\epsilon(t)+\gamma(t)\dot{\epsilon}(t)+\gamma(t)C_2\right)\left\|\lambda_1-\widetilde{\lambda}_1\right\|+\beta(t)\left\|K^*\right\|\left\|\lambda_2-\widetilde{\lambda}_2\right\| \\
			&\quad +\left(\alpha(t)+\gamma(t)C_1+\gamma(t)\epsilon(t)\right)\left\|\lambda_3-\widetilde{\lambda}_3\right\| \\
			&\quad +\left(\theta t\beta(t)+(1+\theta)\gamma(t)\right)\left\|K^*\right\|\left\|\lambda_4-\widetilde{\lambda}_4\right\|,
		\end{aligned}
	\end{equation}
	the second inequality follows from the fact that $f$ is $L_1$-smooth and $\nabla^2f$ is $L_2$-Lipschitz, where $C_1$ and $C_2$ are constants. Similarly, the scaled form of $\left\|N(t)-\widetilde{N}(t)\right\|$ is as follows:
	\begin{equation}\label{oula3}
		\begin{aligned}
			\left\|N(t)-\widetilde{N}(t)\right\| &\leq \beta(t)\|K\|\left\|\lambda_1-\widetilde{\lambda}_1\right\| +\left(\beta(t)L_3+\beta(t)\epsilon(t)+\gamma(t)\dot{\epsilon}(t)+\gamma(t)C_4\right)\left\|\lambda_2-\widetilde{\lambda}_2\right\| \\
			&\quad +\left(\theta t\beta(t)+(1+\theta)\gamma(t)\right)\|K\|\left\|\lambda_3-\widetilde{\lambda}_3\right\| \\
			&\quad +\left(\alpha(t)+\gamma(t)C_3+\gamma(t)\epsilon(t)\right)\left\|\lambda_4-\widetilde{\lambda}_4\right\| ,
		\end{aligned}
	\end{equation}
	where $C_3$ and $C_4$ are constants.
	
	Now, by combining \eqref{oula1}, \eqref{oula2} and \eqref{oula3}, we can obtain
	\begin{equation*}
		\begin{aligned}
			&\quad \left\|\mathcal{G}(t,\Lambda)-\mathcal{G}(t,\widetilde{\Lambda})\right\| \\
			&\leq \left(I_1(t)\left(\beta(t)L_1+\beta(t)\epsilon(t)+\gamma(t)\dot{\epsilon}(t)+\gamma(t)C_2\right)+I_2(t)\beta(t)\|K\|\right) \left\|\lambda_1-\widetilde{\lambda}_1\right\| \\
			&\quad +\left(I_2(t)\left(\beta(t)L_3+\beta(t)\epsilon(t)+\gamma(t)\dot{\epsilon}(t)+\gamma(t)C_4\right)+I_1(t)\beta(t)\left\|K^*\right\|\right) \left\|\lambda_2-\widetilde{\lambda}_2\right\| \\
			&\quad +\left( 1+I_1(t)\left(\alpha(t)+\gamma(t)C_1+\gamma(t)\epsilon(t)\right)+I_2(t)\left(\theta t\beta(t)+(1+\theta)\gamma(t)\right)\|K\| \right)  \left\|\lambda_3-\widetilde{\lambda}_3\right\| \\
			&\quad +\left( 1+I_2(t)\left(\alpha(t)+\gamma(t)C_3+\gamma(t)\epsilon(t)\right)+I_1(t)\left(\theta t\beta(t)+(1+\theta)\gamma(t)\right)\left\|K^*\right\| \right)  \left\|\lambda_4-\widetilde{\lambda}_4\right\| \\
			&\leq S_1(t) \left\|\Lambda-\widetilde{\Lambda}\right\|,
		\end{aligned}
	\end{equation*}
	where $D_1=max\{L_1,L_3\}$, $D_2=max\{C_1,C_2,C_3,C_4\}$ and we denote $S_1(t)$ as
	\begin{equation*}
		\begin{aligned}
			S_1(t) &= 1+\left(I_1(t)+I_2(t)\right) \Bigl( \left((1+\theta t)\beta(t)+(1+\theta)\gamma(t)\right)\left(\|K\|+\left\|K^*\right\|\right) \\
			&\quad + \alpha(t)+\beta(t)\left(D_1+\epsilon(t)\right)+\gamma(t)\left(D_2+\epsilon(t)+\dot{\epsilon}(t)\right) \Bigr).
		\end{aligned}
	\end{equation*}
	Since $\epsilon,\alpha,\gamma,\beta:[t_0,+\infty) \to (0,+\infty)$ are continuous functions, we have $S_1(t) \in L^1_{loc}[t_0, +\infty)$. Further, by the Lipschitz continuity of $\nabla f$ and $\nabla ^2 f$, there exists a constant $D_3$ such that
	\begin{equation}\label{oula7}
		\left\|\nabla f(\lambda_1)\right\| \leq \left\|\nabla f(0)\right\|+L_1\left\|\lambda_1\right\|
	\end{equation}
	and
	\begin{equation}\label{oula8}
		\left\|\nabla^2 f(\lambda_1)\lambda_3\right\| \leq D_3 \left\|\lambda_3\right\|.
	\end{equation}
	Therefore, for any given $\Lambda \in \mathbb{R}^n \times \mathbb{R}^m \times\mathbb{R}^n \times \mathbb{R}^m$ and $t_0<T<+\infty$, the following inequality holds
	\begin{equation}\label{oula4}
		\int_{t_0}^{T} \, \left\| \mathcal{G}(t, \Lambda) \right\| \mathrm{d}t \leq \int_{t_0}^{T} \, \left(\left\|\lambda_3\right\|+\left\|\lambda_4\right\|+I_1(t)\left\|M(t)\right\|+I_2\left\|N(t)\right\|\right) \, \mathrm{d}t.
	\end{equation}
	Substituting \eqref{oula7} and \eqref{oula8} into $\left\|M(t)\right\|$, we can obtain
	\begin{equation}
		\begin{aligned}\label{oula5}
			\left\|M(t)\right\| &\leq \beta(t)\left\|\nabla f(0)\right\|+\left(\beta(t)\epsilon(t)+\gamma(t)\dot{\epsilon}(t)+\beta(t)L_1\right) \left\|\lambda_1\right\| + \beta(t)\left\|K^*\right\| \left\|\lambda_2\right\| \\
			&\quad +\left(\alpha(t)+\gamma(t)\epsilon(t)+\gamma(t)D_3\right) \left\|\lambda_3\right\| + \left(\theta t\beta(t)+(1+\theta)\gamma(t)\right)\left\|K^*\right\| \left\|\lambda_4\right\| \\
			&\leq \left(\beta(t)\epsilon(t)+\gamma(t)\dot{\epsilon}(t)+\beta(t)L_1+\beta(t) \left\|\nabla f(0)\right\|\right)(1+\|\Lambda\|) + \beta(t)\left\|K^*\right\| (1+\left\|\Lambda\right\|) \\
			&\quad +\left(\alpha(t)+\gamma(t)\epsilon(t)+\gamma(t)D_3\right) (1+\left\|\Lambda\right\|) + \left(\theta t\beta(t)+(1+\theta)\gamma(t)\right)\left\|K^*\right\| (1+\left\|\Lambda\right\|) .
		\end{aligned}
	\end{equation}
	Similarly, it follows that
	\begin{equation}\label{oula6}
		\begin{aligned}
			\left\|N(t)\right\| &\leq \beta(t)\|K\| (1+\|\Lambda\|) +  \left(\beta(t)\epsilon(t)+\gamma(t)\dot{\epsilon}(t)+\beta(t)L_3+\beta(t)\left\|\nabla g(0)\right\|\right)(1+\|\Lambda\|) \\
			&\quad +\left(\theta t\beta(t)+(1+\theta)\gamma(t)\right)\|K\| (1+\|\Lambda\|)  + \left(\alpha(t)+\gamma(t)\epsilon(t)+\gamma(t)D_4\right) (1+\|\Lambda\|),
		\end{aligned}
	\end{equation}
	where $D_4$ is a constant which satisfies $\left\|\nabla^2g(\lambda_2)\lambda_4\right\| \leq D_4\left\|\lambda_4\right\|$.
	Combining \eqref{oula4}, \eqref{oula5} and \eqref{oula6}, we can get
	\begin{equation*}
		\int_{t_0}^{T} \, \left\| \mathcal{G}(t, \Lambda) \right\| \, \mathrm{d}t \leq \int_{t_0}^{T} \, S_2(t)(1+\left\|\Lambda\right\|) \, \mathrm{d}t,
	\end{equation*}
	where we denote $S_2(t)$ as
	\begin{equation*}
		\begin{aligned}
			S_2(t) &= \left(I_1(t)+I_2(t)\right) \Bigl( 1+\alpha(t)+\gamma(t)\left(D_3+D_4+\epsilon(t)+\dot{\epsilon}(t)\right) \\ 
			&\quad + \beta(t)\left(L_1+L_3+\epsilon(t)+\|\nabla f(0)\|+\|\nabla g(0)\|\right) +  \left((1+\theta t)\beta(t)+(1+\theta)\gamma(t)\right)\left(\|K\|+\left\|K^*\right\|\right) \Bigr).
		\end{aligned}		
	\end{equation*}
	It is clear that $S_2(t) \in L^1_{loc}[t_0, +\infty)$ and for any $\Lambda \in \mathbb{R}^n \times \mathbb{R}^m \times \mathbb{R}^n \times \mathbb{R}^m$, $\mathcal{G}(\cdot,\Lambda) \in L^1_{loc}([t_0, +\infty); \mathbb{R}^n \times \mathbb{R}^m \times \mathbb{R}^n \times \mathbb{R}^m)$. According to \cite[Proposition 6.2.1]{ref46} and \cite[Corollary A.2]{refjia}, the Cauchy problem \eqref{oula} has a unique strong global solution $\Lambda \in W^{1,1}_{loc}([t_0, +\infty);\mathbb{R}^n \times \mathbb{R}^m \times \mathbb{R}^n \times \mathbb{R}^m)$. This in turn leads to the existence and uniqueness of a strong solution $(x,y)$ of the dynamical system \eqref{eq15}.
	
	This completes the proof of Theorem \ref{theorem1A1}.
\end{proof}

\section{An Auxiliary Result}\label{Appendix B}
\begin{lemma}\cite[Lemma A.3]{ref21}\label{limit}
	Suppose that $s > 0$, $\zeta \in L^1([s, +\infty))$ is a nonnegative and continuous function. Additionally, $\varphi : [s, +\infty) \to (0, +\infty)$ is a nondecreasing function with $\lim\limits_{t \to +\infty} \, \varphi(t) = +\infty$. Then,
	\begin{equation}\label{eqlimit}
		\lim_{t \to +\infty} \; \frac{1}{\varphi(t)} \int_{s}^{t} \, \varphi(\tau)\zeta(\tau) \, \mathrm{d}\tau = 0.
	\end{equation}
\end{lemma}

\end{appendices}



\begin{thebibliography}{plain}  
	\bibitem{yueshu1}
	Adly, S., Attouch, H.:
	\emph{Accelerated optimization through time-scale analysis of inertial dynamics with asymptotic vanishing and Hessian-driven dampings}.
	Optimization. (2024) \href{https://doi.org/10.1080/02331934.2024.2359540}{Doi: 10.1080/02331934.2024.2359540}
	
	\bibitem{ref23}
	Alecsa, C. D., L\'{a}szl\'{o}, S. C.: 
	\emph{Tikhonov regularization of a perturbed heavy ball system with vanishing damping}. 
	SIAM J. Optim. 31, 2921-2954 (2021)
	
	\bibitem{alvarez2002second}
	Alvarez, F., Attouch, H., Bolte, J., Redont, P.:
	\emph{A second-order gradient-like dissipative dynamical system with Hessian-driven damping. Application to optimization and mechanics}.
	J. Math. Pures Appl. 81, 747-779 (2002)
	
	\bibitem{newton}
	Alvarez D., F., P{\'e}rez C., J. M.:
	\emph{A dynamical system associated with Newton’s method for parametric approximations of convex minimization problems}.
	Appl. Math. Optim. 38, 193-217 (1998)
	
	\bibitem{intro15}
	Alvarez-Sanchez, L. G., I{\~n}{\'a}rritu, P. G. Q., {\v{S}}ip{\v{c}}i{\'c}, N., Kohrangi, M., Bazzurro, P.:
	\emph{Hazard-consistent simulated earthquake ground motions for PBEE applications on stiff soil and rock sites}. 
	Earthq. Eng. Struct. Dyn. 52, 4900–4918 (2023)
	
	\bibitem{ref24}
	Attouch, H., Balhag, A., Chbani, Z., Riahi, H.:
	\emph{Accelerated gradient methods combining Tikhonov regularization with geometric damping driven by the Hessian}. 
	Appl. Math. Optim. 88, 29 (2023) 
	
	\bibitem{ref12}
	Attouch, H., Chbani, Z., Fadili, J., Riahi, H.:
	\emph{First-order optimization algorithms via inertial systems with Hessian driven damping}. 
	Math. Program. 193, 113-155 (2022)
	
	\bibitem{ref21}
	Attouch, H., Chbani, Z., Riahi, H.:
	\emph{Combining fast inertial dynamics for convex optimization with Tikhonov regularization}. 
	J. Math. Anal. Appl. 457, 1065-1094 (2018)
	
	\bibitem{ref47}
	Attouch, H., Cominetti, R.: 
	\emph{A dynamical approach to convex minimization coupling approximation with the steepest descent method}.
	J. Differ. Equ. 128, 519-540 (1996)
	
	\bibitem{ref20}
	Attouch, H., Peypouquet, J., Redont, P.:
	\emph{Fast convex optimization via inertial dynamics with Hessian driven damping}. 
	J. Differ. Equ. 261, 5734-5783 (2016)
	
	\bibitem{shangjie1}
	Attouch, H., L\'aszl\'o, S. C.: 
	\emph{Convex optimization via inertial algorithms with vanishing Tikhonov regularization: fast convergence to the minimum norm solution}.
	Math. Methods Oper. Res. 99, 307-347 (2024)
	
	\bibitem{yueshu2}
	Aujol, J.-F., Dossal, C., Ho{\`a}ng, V. H., Labarri{\`e}re, H., Rondepierre, A.: 
	\emph{Fast convergence of inertial dynamics with Hessian-driven damping under geometry assumptions}.
	Appl. Math. Optim. 88, 81 (2023)
	
	\bibitem{ref26}
	Bagy, A. C., Chbani, Z., Riahi, H.:
	\emph{Strong convergence of trajectories via inertial dynamics combining Hessian driven damping and Tikhonov regularization for general convex minimizations}. 
	Numer. Funct. Anal. Optim. 44, 1481-1509 (2023) 
	
	\bibitem{ref22}
	Bo\c{t}, R. I., Csetnek, E. R., L\'{a}szl\'{o}, S. C.:
	\emph{Tikhonov regularization of a second-order dynamical system with Hessian driven damping}. 
	Math. Program. 189, 151-186 (2021)
	
	\bibitem{bg3}
	B{\"o}hm, A., Sedlmayer, M., Csetnek, E. R., Bo{\c{t}}, R. I.:
	\emph{Two steps at a time -- taking GAN training in stride with Tseng's method}.
	SIAM. J. Math. Data. Sci. 4, 750-771 (2022)
	
	\bibitem{refjia}
	Br\'{e}zis, H.: 
	\emph{Op\'{e}rateurs maximaux monotones et semi-groupes de contractions dans les espaces de Hilbert}. 
	North-Holland Publishing Co. (1973)
	
	\bibitem{popular3}
	Bubeck, S.:
	\emph{Convex optimization: algorithms and complexity}.
	Found. Trends Mach. Learn. 8, 231-357 (2015)
	
	\bibitem{ref3}
	Chambolle, A., Pock, T.: 
	\emph{A first-order primal-dual algorithm for convex problems with applications to imaging}.
	J. Math. Imaging Vis. 40, 120-145 (2011)
	
	\bibitem{ref6}
	Chambolle, A., Pock, T.: 
	\emph{An introduction to continuous optimization for imaging}. 
	Acta Numer. 25, 161-319 (2016)
	
	\bibitem{ref39}
	Csetnek, E. R., L\'{a}szl\'{o}, S. C.:
	\emph{Strong convergence and fast rates for systems with Tikhonov regularization}. 
	(2024) \href{https://doi.org/10.48550/arXiv.2411.17329}{arXiv.2411.17329} 
	
	\bibitem{intro18}
	Defazio, A., Bach, F., Lacoste-Julien, S.: 
	\emph{SAGA: A fast incremental gradient method with support for non-strongly convex composite objectives}. 
	NeurIPS (2014)
	
	\bibitem{bg2}
	De Montbru, E., Renault, J.:
	\emph{Optimistic gradient descent ascent in general-sum bilinear games}.
	J. Dyn. Games 12, 267-301 (2025)
	
	\bibitem{ref41}
	Ding, K.-W., Fliege, J., Vuong, P. T.:
	\emph{Fast convergence of the primal-dual dynamical system and corresponding algorithms for a nonsmooth bilinearly coupled saddle point problem}. 
	Comput. Optim. Appl. 90, 151-192 (2025)
	
	\bibitem{intro3}
	Dowell, E. H., Clark, R., Cox, D., Crawley, E. F., Curtiss, H. C., Jr., Peters, D. A., Scanlan, R. H., Sisto, F.:
	\emph{A modern course in aeroelasticity (4th ed.)}.  
	Springer (2005)
	
	\bibitem{lilun1}
	Hale, J. K.:
	\emph{Asymptotic behavior of dissipative systems}.
	AMS 25 (1988)
	
	\bibitem{ref46}
	Haraux, A.:
	\emph{Syst\`{e}mes dynamiques dissipatifs et applications}. 
	Masson 17 (1991)
	
	\bibitem{intro13}
	Harnefors, L., Wang, X., Yepes, A. G., Blaabjerg, F.: 
	\emph{Passivity-based stability assessment of grid-connected VSCs—An overview}. 
	IEEE J. Emerg. Sel. Topics Power Electron. 4, 116-125 (2016)
	
	\bibitem{ref42}
	He, X., Hu, R., Fang, Y. P.: 
	\emph{A second-order primal-dual dynamical system for a convex-concave bilinear saddle point problem}. 
	Appl. Math. Optim. 89, 30 (2024)
	
	\bibitem{ref37}
	He, X., Tian, F., Li, A. Q., Fang, Y. P.:
	\emph{Convergence rates of mixed primal-dual dynamical systems with Hessian driven damping}. 
	Optimization 74, 365-390 (2025)
	
	\bibitem{intro9}
	Kova{\v{c}}i{\'c}, I., Brennan, M. J.: 
	\emph{The Duffing equation: nonlinear oscillators and their behaviour}. 
	John Wiley \& Sons (2011)
	
	\bibitem{bg4}
	Kovalev, D., Gasnikov, A., Richt{\'a}rik, P.:
	\emph{Accelerated primal-dual gradient method for smooth and convex-concave saddle-point problems with bilinear coupling}.
	NeurIPS (2022)
	
	\bibitem{wending2}
	LaSalle, J. P.:
	\emph{Some extensions of Liapunov's second method}. 
	IRE Trans. Circuit Theory 7, 520–527 (1960)
	
	\bibitem{wending3}
	LaSalle, J. P.:
	\emph{The stability of dynamical systems}. 
	SIAM (1976)
	
	\bibitem{intro19}
	Lessard, L., Recht, B., Packard, A.:
	\emph{Analysis and design of optimization algorithms via integral quadratic constraints}. 
	SIAM J. Optim. 26, 57–95 (2016)
	
	\bibitem{ref43}
	Luo, H.: 
	\emph{A continuous perspective on the inertial corrected primal-dual proximal splitting}. 
	Optimization (2025) \href{https://doi.org/10.1080/02331934.2025.2588432}{Doi: 10.1080/02331934.2025.2588432}
	
	\bibitem{intro12}
	Luo, J., Liu, B., Guo, X., Yang, L., Liu, C., Xu, Y., Bu, S.: 
	\emph{Oscillation stability induced by wind power integration: Incidents, mechanism, countermeasures and future challenges}. 
	Renew. Sustain. Energy Rev. 217, 115692 (2025)
	
	\bibitem{wending1}
	Lyapunov, A. M.:
	\emph{The general problem of the stability of motion (Fuller, A. T. Trans.)}.
	Int. J. Control 55, 531–534 (1992)
	
	\bibitem{ref29}
	L\'{a}szl\'{o}, S. C.:
	\emph{Solving convex optimization problems via a second-order dynamical system with implicit Hessian damping and Tikhonov regularization}. 
	Comput. Optim. Appl. 90, 113-149 (2025)
	
	\bibitem{intro14}
	Miari, M., Choong, K. K., Jankowski, R.:
	\emph{Seismic pounding between adjacent buildings: Identification of parameters, soil interaction issues and mitigation measures}. 
	Soil Dyn. Earthquake Eng. 121, 135–150 (2019)
	
	\bibitem{intro17}
	Nesterov, Y.: 
	\emph{Lectures on convex optimization (2nd ed.)}. 
	Springer (2018)
	
	\bibitem{suanfa4}
	Nesterov, Y.: 
	\emph{A method of solving a convex programming problem with convergence rate  $\mathcal{O}(\frac{1}{k^2})$}. 
	Sov. Math. Dokl. 27, 372–376 (1983)
	
	\bibitem{intro16}
	Nocedal, J., Wright, S. J.:
	\emph{Numerical optimization (2nd ed.)}. 
	Springer (2006)
	
	\bibitem{lilun0}
	Palis, J. J., de Melo, W.:
	\emph{Geometric theory of dynamical systems: an introduction}. 
	Springer-Verlag (1982)
	
	\bibitem{intro8}
	Park, J. Y., Kim, J. Y.: 
	\emph{Dynamic three-wave coupling between local sporadic spokes and emergence of global breathing oscillation in partially magnetized cross-field plasmas}. 
	Commun. Phys. 8, 380 (2025)
	
	\bibitem{polyak1}
	Polyak, B. T.:
	\emph{Some methods of speeding up the convergence of iteration methods}.
	USSR Comput. Math. Math. Phys. 4, 1-17 (1964)
	
	\bibitem{polyak2}
	Polyak, B. T.:
	\emph{Introduction to Optimization}.
	Optimization Software (1987)
	
	\bibitem{wending4}
	Rouche, N., Habets, P., Laloy, M.:
	\emph{Stability theory by Liapunov's direct method}. 
	Springer-Verlag (1977)
	
	\bibitem{smale1961gradient}
	Smale, S.:
	\emph{On gradient dynamical systems}.
	Ann. Math. 74, 199-206 (1961)
	
	\bibitem{ref44}
	Sun, X. K., He, L., Long, X.-J.: 
	\emph{Tikhonov regularized inertial primal-dual dynamics for convex-concave bilinear saddle point problems}. 
	(2024) \href{https://doi.org/10.48550/arXiv.2409.05301}{arXiv:2409.05301} 
	
	\bibitem{ref45}
	Sun, X. K., He, L., Long, X.-J.: 
	\emph{Inertial primal-dual dynamics with Hessian-driven damping and Tikhonov regularization for convex-concave bilinear saddle point problems}. 
	Optimization (2025) \href{https://doi.org/10.1080/02331934.2025.2578403}{Doi: 10.1080/02331934.2025.2578403}
	
	\bibitem{lilun2}
	Temam, R.:
	\emph{Infinite dimensional dynamical systems in mechanics and physics (2nd ed.)}.
	Springer (2012)
	
	\bibitem{wending5}
	Vidyasagar, M.:
	\emph{Nonlinear systems analysis (2nd ed.)}. 
	SIAM (2002)
	
	\bibitem{bg1}
	Wang J.-K., Abernethy J., Levy K. Y.:
	\emph{No-regret dynamics in the Fenchel game: a unified framework for algorithmic convex optimization}.
	Math. Program. 205, 203-268 (2024)
	
	\bibitem{intro4}
	Wu, K. X., Fan, X. Y., S., A. K., Fu, S., Kim, H. D., Sethuraman, V. R. P.:
	\emph{Study on shock train oscillations in a rectangular diverging isolator based on large eddy simulation}.
	J. Fluid Mech. 1030 (2026)
	
	\bibitem{ref40}
	Zeng, X. L., Dou, L., Chen, J.:
	\emph{Accelerated first-order continuous-time algorithm for solving convex-concave bilinear saddle point problem}. 
	IFAC-PapersOnLine. 53, 7362-7367 (2020)
	
	\bibitem{hessian2}
	Zhang, B. H., Zhang, X. J.:
	\emph{A general Tikhonov regularized second-order dynamical system for convex-concave bilinear saddle point problems}.
	(2026) \href{https://doi.org/10.48550/arXiv.2601.23120}{arXiv:2601.23120}
	
	\bibitem{ref28}
	Zhong, G. F., Hu, X. Z., Tang, M., Zhong, L. Q.:
	\emph{Fast convex optimization via differential equation with Hessian-driven damping and Tikhonov regularization}. 
	J. Optim. Theory Appl. 203, 42-82 (2024)
	
	\bibitem{ref36}
	Zhu, T. T., Hu, R., Fang, Y. P.:
	\emph{Tikhonov regularized second-order plus first-order primal-dual dynamical systems with asymptotically vanishing damping for linear equality constrained convex optimization problems}. 
	Optimization 75, 121-148 (2026)

	
	
\end{thebibliography}
\end{document}